\pdfoutput=1
\documentclass[12pt,reqno]{amsart}
\usepackage{amsmath,amssymb,tikz,pgfplots,mathrsfs,framed,pgf,latexsym,amsthm,footnpag,url,multicol}

\usepackage{mathtools}
\usetikzlibrary{cd}
\usepackage[colorlinks,breaklinks=true]{hyperref}
\usepackage[figure,table]{hypcap}
\hypersetup{
	bookmarksnumbered,
	pdfstartview={FitH},
	citecolor={black},
	linkcolor={black},
	urlcolor={black},
	pdfpagemode={UseOutlines}
}

\usepackage{cleveref}

\usepackage{relsize} 
\usepackage[bbgreekl]{mathbbol}
\DeclareSymbolFontAlphabet{\mathbb}{AMSb}
\DeclareSymbolFontAlphabet{\mathbbl}{bbold}
\newcommand{\Prism}{{ \mathbbl{\Delta}}}

\newtheorem{theorem}{Theorem}[section]
\crefname{theorem}{Theorem}{Theorems}
\newtheorem{lemma}[theorem]{Lemma}
\crefname{lemma}{Lemma}{Lemmas}
\newtheorem{corollary}[theorem]{Corollary}
\crefname{corollary}{Corollary}{Corollaries}
\newtheorem{proposition}[theorem]{Proposition}
\crefname{proposition}{Proposition}{Propositions}

\theoremstyle{definition}
\newtheorem{definition}[theorem]{Definition}
\crefname{definition}{Definition}{Definitions}
\newtheorem{remark}[theorem]{Remark}
\crefname{remark}{Remark}{Remarks}
\newtheorem{example}[theorem]{Example}
\crefname{example}{Example}{Examples}

\crefname{conjecture}{Conjecture}{Conjectures}
\newtheorem{question}[theorem]{Question}
\crefname{question}{Question}{Questions}
\numberwithin{equation}{theorem}

\renewcommand{\a}{\mathring{a}}

\renewcommand{\inf}{\mathrm{inf}}
\newcommand{\crys}{\mathrm{crys}}
\newcommand{\rig}{\mathrm{rig}}
\newcommand{\dR}{\mathrm{dR}}

\newcommand{\cyc}{\mathrm{cyc}}
\newcommand{\fh}{\mathrm{fh}}
\newcommand{\et}{\mathrm{\acute{e}t}}
\newcommand{\fin}{{\text{$\varphi$\rm{-fin}}}}
\newcommand{\icrys}{{\text{$\varphi$\rm{-isoc}}}}
\newcommand{\Ker}{\mathop{\mathrm{Ker}}\nolimits}
\newcommand{\Coker}{\mathop{\mathrm{Coker}}\nolimits}

\newcommand{\Ima}{\mathop{\mathrm{Im}}\nolimits}

\newcommand{\AKp}{\widetilde{\mathbb{A}}_K^{+}}
\newcommand{\ALp}{\widetilde{\mathbb{A}}_L^{+}}
\newcommand{\AK}{\widetilde{\mathbb{A}}_K}
\newcommand{\BKp}{\widetilde{\mathbb{B}}_K^{+}}
\newcommand{\BK}{\widetilde{\mathbb{B}}_K}
\newcommand{\Kinf}{\widehat{K_\infty}}

\DeclareMathOperator{\Hom}{Hom}
\DeclareMathOperator{\id}{id}
\DeclareMathOperator{\Rep}{Rep}
\DeclareMathOperator{\Vect}{Vect}
\DeclareMathOperator{\Gal}{Gal}

\DeclareMathOperator{\Frac}{Frac}
\DeclareMathOperator{\ev}{ev}
\DeclareMathOperator{\Mod}{Mod}
\DeclareMathOperator{\rank}{rank}

\DeclareMathOperator{\Fil}{Fil}

\pgfplotsset{compat=1.18}
\date{}
\title{\vspace{-1cm}On the $(\varphi,\Gamma)$-modules corresponding to crystalline representations}

\author{Takumi Watanabe}
\address{Graduate School of Mathematical Sciences, The University of Tokyo, 
3-8-1 Komaba, Meguro-ku, Tokyo, 153-8914, Japan}
\email{takumi0426@g.ecc.u-tokyo.ac.jp}
\begin{document}

\begin{abstract}
  Let $K$ be a complete discrete valuation field of characteristic $0$ with perfect residue field of characteristic $p>0$. We introduce the notion of \emph{crystalline $(\varphi,\Gamma)$-modules over $\AKp$} and show that their category is equivalent to the category of crystalline $\mathbb{Z}_p$-representations of the absolute Galois group of $K$. In other words, we determine the $(\varphi,\Gamma)$-modules over $\AK$ that correspond to crystalline representations. This equivalence generalizes, in certain respects, that of L. Berger in the unramified case. 

\vspace{2ex}
  Soit $K$ un corps de valuation discr\`{e}te complet de caract\'{e}ristique $0$ \`{a} corps r\'{e}siduel parfait de caract\'{e}ristique $p>0$. Nous introduisons la notion de {\em $(\varphi,\Gamma)$-modules cristallins sur $\AKp$} et montrons que leur cat\'{e}gorie est \'{e}quivalente \`{a} la cat\'{e}gorie des $\mathbb{Z}_p$-repr\'{e}sentations cristallines du groupe de Galois absolu de $K$. Autrement dit, nous d\'{e}terminons les $(\varphi,\Gamma)$-modules sur $\AK$ qui correspondent aux repr\'{e}sentations cristallines. 
  Cette \'{e}quivalence g\'{e}n\'{e}ralise, \`{a} certains \'{e}gards, celle de L. Berger dans le cas non ramifi\'{e}.
\end{abstract}
\maketitle
\tableofcontents

\section*{Introduction}
Let $K$ be a mixed characteristic complete discrete valuation field with perfect residue field $k$ of characteristic $p$ (for example, a finite extension of $\mathbb{Q}_p$) and let $G_K$ denote its absolute Galois group $\Gal \left(\overline{K}/K\right)$. 
A {\em $p$-adic Galois representation} (resp. {\em free $\mathbb{Z}_p$-representation}) of $G_K$ is a finite dimensional $\mathbb{Q}_p$-vector space (resp. finite free $\mathbb{Z}_p$-module) equipped with a continuous linear $G_K$-action. 
One of the main theories to study them is the theory of $(\varphi,\Gamma)$-modules developed by J.-M. Fontaine (\cite{Fontaine1984,Fontaine1990}). Let us recall it for free $\mathbb{Z}_p$-representations. 
We introduce some notation. We put $K_0:=W(k)[1/p]$. We fix an algebraic closure $\overline{K}$ of $K$ and let $K_\infty$ denote the $p$-cyclotomic extension $\bigcup_{n}K(\zeta_{p^n})$ of $K$ in $\overline{K}$, where $\zeta_{p^n}$ is a primitive $p^n$-th root of unity.  
We define $\Gamma_K:=\Gal \left(K_\infty/K\right)$ and $H_K:=\Gal \left(\overline{K}/K_\infty\right)$. Let $\mathbb{C}_p$ denote the completion of $\overline{K}$ and let $\mathcal{O}_{\mathbb{C}_p}$ denote the ring of integers. 
Let $\mathcal{O}_{\mathbb{C}_p}^\flat$ denote the tilt $\varprojlim_{x\mapsto x^p} \mathcal{O}_{\mathbb{C}_p}/p \mathcal{O}_{\mathbb{C}_p}$ of $\mathcal{O}_{\mathbb{C}_p}$ and let $\mathbb{C}_p^\flat:=\Frac \mathcal{O}_{\mathbb{C}_p}^\flat$ be its field of fractions. 
Fontaine constructed subrings $\mathbb{A}_K,\mathbb{A}\subseteq W(\mathbb{C}_p^\flat)$ stable under the Frobenius map and the $G_K$-action. Since $\mathbb{A}_K$ is fixed by $H_K$, it has the $\Gamma_K$-action. When $K$ is absolutely unramified, i.e., $K=K_0$, $\mathbb{A}_K$ is equal to $W(k)((\mu))^{\wedge}_p$, where $\epsilon:=(1,\zeta_p,\zeta_{p^2},\dots )\in \mathcal{O}_{\mathbb{C}_p}^\flat$ and $\mu:=[\epsilon]-1\in A_\inf:=W(\mathcal{O}_{\mathbb{C}_p}^\flat)$. A {\em $(\varphi,\Gamma)$-module over $\mathbb{A}_K$} is a finite free $\mathbb{A}_K$-module $M$ equipped with a semi-linear continuous $\Gamma_K$-action and a $\Gamma_K$-equivariant $\mathbb{A}_K$-linear isomorphism $M\otimes_{\mathbb{A}_K,\varphi}\mathbb{A}_K \xrightarrow{\cong } M$. Since the $\Gamma_K$-action is determined by its topological generator, $(\varphi,\Gamma)$-modules are simpler than free $\mathbb{Z}_p$-representations of $G_K$. 
One of the main theorems on them is the following:

\begin{theorem}[{\cite[Th\'{e}or\`{e}me 3.4.3]{Fontaine1990}}]\label{thm:classical phi Gamma module}
There exists an equivalence of categories between the category of free $\mathbb{Z}_p$-representations of $G_K$ denoted by $\Rep_{\mathbb{Z}_p}(G_K)$ and the category of $(\varphi,\Gamma)$-modules over $\mathbb{A}_K$ denoted by $\Mod_{\varphi,\Gamma}^{\et}(\mathbb{A}_K)$ via the functors
\[\begin{tikzcd}[row sep=tiny]
     \Mod_{\varphi,\Gamma}^{\et}(\mathbb{A}_K)\arrow[r, <->, "\sim "]&\Rep_{\mathbb{Z}_p}(G_K) \\ 
     M\arrow[r, mapsto]\arrow[u, "\rotatebox{90}{$\in$}",phantom]& (M\otimes_{\mathbb{A}_K} \mathbb{A})^{\varphi=1} \arrow[u, "\rotatebox{90}{$\in$}",phantom]\\[-0.25cm]
     (T\otimes_{\mathbb{Z}_p}\mathbb{A})^{H_K}&\arrow[l, mapsto] T.
\end{tikzcd}\]
\end{theorem}

We have explained a theory dealing with all $p$-adic Galois representations or all free $\mathbb{Z}_p$-representations. However, there exist some classes of them which are important in number theory. Main examples are {\em crystalline representations}, {\em semi-stable representations} and {\em de Rham representations}. In this paper, we consider the following question:

\begin{question}
     Which $(\varphi,\Gamma)$-modules correspond to crystalline representations?
\end{question}

As we mentioned before, $(\varphi,\Gamma)$-modules are simpler than $p$-adic Galois representations, so it is important and helpful to determine the $(\varphi,\Gamma)$-modules corresponding to crystalline representations. This question was studied by J.-M. Fontaine \cite{Fontaine1990}, N. Wach \cite{Wach1996,Wach1997} and P. Colmez \cite{Colmez1999}. Then L. Berger \cite{berger2004} finally constructed a theory of Wach modules when $K$ is absolutely unramified. Let $\mathbb{A}_K^+$ denote the ring $W(k)[[\mu]]$ and $\xi:=\mu/\varphi^{-1}(\mu)$. A {\em Wach module} is a finite free $\mathbb{A}_K^+$-module $N$ equipped with a continuous semi-linear $\Gamma_K$-action which is trivial on $N/\mu N$ and a $\Gamma_K$-equivariant $\mathbb{A}_K^+[1/\varphi(\xi)]$-linear isomorphism $N \otimes_{\mathbb{A}_K^+, \varphi}\mathbb{A}_K^+[1/\varphi(\xi)] \xrightarrow{\cong } N[1/\varphi(\xi)]$. One of the main theorems on Wach modules is the following.

\begin{theorem}[{\cite[Proposition III.4.2]{berger2004}}]
Assume $K$ is absolutely unramified. Then the category of Wach modules is equivalent to the category of free crystalline $\mathbb{Z}_p$-representations of $G_K$ via the functor $N\mapsto (N \otimes_{\mathbb{A}_K^+}\mathbb{A})^{\varphi=1}$ for any Wach module $N$.
\end{theorem}

Let $T$ be a free crystalline $\mathbb{Z}_p$-representation of $G_K$ and let $N$ be the corresponding Wach module. Then the $(\varphi,\Gamma)$-module over $\mathbb{A}_K$ corresponding to $T$ is given by $N\otimes_{\mathbb{A}_K^+}\mathbb{A}_K$. So this theorem says that when $K=K_0$, the $(\varphi,\Gamma)$-modules over $\mathbb{A}_K$ that correspond to crystalline representations can be determined by the existence of $\mathbb{A}_K^+$-lattices satisfying some conditions.

Our motivation is to remove the unramified condition. The difficulty is the construction of $\mathbb{A}_K^+$. There are some conditions which $\mathbb{A}_K^+$ should satisfy, but the author does not know how to construct $\mathbb{A}_K^+$ in the ramified case (cf. \cite[Remarque on page 393]{Wach1996}, \cite{WatanabeAKp}).
To settle this problem, we use $(\varphi,\Gamma)$-modules over $\AK:=W({\Kinf}^\flat)$, where ${\Kinf}^\flat$ is the field of fractions of the tilt $\mathcal{O}_{\Kinf}^\flat:= \varprojlim_{x\mapsto x^p}\mathcal{O}_{\Kinf}/p \mathcal{O}_{\Kinf}$ of $\mathcal{O}_{\Kinf}$. 
Similarly to \cref{thm:classical phi Gamma module}, the category of $(\varphi,\Gamma)$-modules over $\AK$ is equivalent to the category of free $\mathbb{Z}_p$-representations of $G_K$. 
In this case, we can use the ring $\AKp:=W(\mathcal{O}_{\Kinf}^\flat)$ as a substitute for $\mathbb{A}_K^+$. 

\begin{definition}[\cref{def:phi Gamma module over AKp and AK,def:crystalline (phi Gamma)-modules}]
A {\em crystalline $(\varphi,\Gamma)$-module over $\AKp$} is a finite free $\AKp$-module $N$ equipped with a semi-linear continuous $\Gamma_K$-action and a $\Gamma_K$-equivariant $\AKp[1/\varphi(\xi)]$-linear isomorphism $\varPhi_N\colon N\otimes_{\AKp,\varphi}\AKp[1/\varphi(\xi)] \xrightarrow{\cong } N[1/\varphi(\xi)] $ such that the following conditions hold.
\begin{enumerate}
     \item (Triviality modulo $\mu$) $(N/\mu)^{\Gamma_K}$ is a finite free $(\AKp/\mu)^{\Gamma_K}$-module and the canonical map
     \begin{align*}
          (N/\mu)^{\Gamma_K} \otimes_{(\AKp/\mu)^{\Gamma_K}} \AKp/\mu  \rightarrow N/\mu
     \end{align*}
     is an isomorphism.

     \item $(N/[\a])_\fin^{\Gamma_K}$ generates $N/[\a]$ as an $\AKp/[\a]$-module.
\end{enumerate}
Here, $\a\in \mathcal{O}_{\Kinf}^{\flat}$ is an element satisfying $v(\a)=p v_p(\pi)$, where $\pi$ is a uniformizer of $K$.
See \cref{def:phi-finite part} for the definition of $\varphi$-fin. 
\end{definition}

\begin{remark}\leavevmode
\begin{itemize}
     \item The condition (1) is an analog of the condition on Wach modules.
     \item The condition (2) can be changed to other conditions. See \cref{def:crystalline (phi Gamma)-modules}.
     \item The $\varphi$-finite part denoted by $\fin$ is a Frobenius analog of $K$-finite in the sense of \cite{Sen1980}. For the motivation of it, see \cref{cor:D_crysphifin,prop:B_crys fin is very small,thm:phi-finite part of B_crys and Bbar}.
\end{itemize}
\end{remark}
Our main theorem in this article is the following:

\begin{theorem}[{\cref{thm:Main theorem for crystalline}}]\label{thm:Main theorem for crystalline intro}
     There exists an equivalence of categories between the category of crystalline $(\varphi,\Gamma)$-modules over $\AKp$ denoted by $\Mod_{\varphi,\Gamma}^{\fh, \crys}(\AKp)$ and the category of free crystalline $\mathbb{Z}_p$-representations of $G_K$ denoted by $\Rep^\crys_{\mathbb{Z}_p}(G_K)$ via the functor
     \begin{align*}
          T\colon \Mod_{\varphi,\Gamma}^{\fh, \crys}(\AKp)  \xrightarrow{\sim} \Rep_{\mathbb{Z}_p}^{\crys}(G_K),\qquad N \mapsto T(N):=\bigl(N\otimes_{\AKp} W(\mathbb{C}_p^\flat)\bigr)^{\varphi=1}.
     \end{align*}
\end{theorem}

Similarly to Wach modules, for any crystalline $(\varphi,\Gamma)$-module $N$ over $\AKp$, $N\otimes_{\AKp}\AK$ is the $(\varphi,\Gamma)$-module over $\AK$ which corresponds to $T(N)$. Therefore, this theorem says that the $(\varphi,\Gamma)$-modules over $\AK$ corresponding to crystalline representations can be determined by the existence of $\AKp$-lattices satisfying some conditions. 

Let us explain our strategy to prove the theorem. We first show the well-definedness of the functor. Let $N$ be a crystalline $(\varphi,\Gamma)$-module over $\AKp$. We see that $N\otimes_{\AKp}\AK$ is the $(\varphi,\Gamma)$-module over $\AK$ and hence $T(N)$ is a free $\mathbb{Z}_p$-representation of $G_K$. 
We tried to show that it is crystalline from the condition (1), but we failed.
So we follow H. Du's idea (\cite{Du2022}), which leads to the condition (2). One of the reasons why we think of the $\varphi$-finite part is that we do not know whether $(\AKp/\mathfrak{q}_\infty)^{\Gamma_K}=W(k)$ or not, but we know that $(\AKp/\mathfrak{q}_\infty)_\fin^{\Gamma_K}=W(k)$ (see \cref{prop:G-fixed part of Bbar}). Note that Du claims the first equality in \cite[Theorem 4.9. (2)]{Du2022}, but it seems that there is a gap in the proof.
We can easily show the faithfulness of the functor, and we deduce the fullness from the condition (1). This argument is due to T. Tsuji (\cite[Proposition 76]{Tsuji_simons}). In order to show that this functor is  essentially surjective and to construct the quasi-inverse functor, we use an equivalence of categories between the category of prismatic $F$-crystals in $\mathcal{O}_\Prism$-modules on $\mathcal{O}_{K,\Prism}$ and the category of free crystalline $\mathbb{Z}_p$-representations of $G_K$ (\cite[Theorem 5.6]{BhattScholze2023}). To be more precise, we prove that for any prismatic $F$-crystal $\mathcal{N}$ in $\mathcal{O}_\Prism$-modules on $\mathcal{O}_{K,\Prism}$, $\mathcal{N}\bigl(\AKp, \varphi(\xi)\bigr)$ is a crystalline $(\varphi,\Gamma)$-module over $\AKp$. 
The following commutative diagram summarizes the relationship of the categories, where $\Vect^\varphi(\mathcal{O}_{K,\Prism}, \mathcal{O}_{\Prism})$ denotes the category of prismatic $F$-crystals in $\mathcal{O}_\Prism$-modules on $\mathcal{O}_{K,\Prism}$.
\[\begin{tikzcd}
\Mod_{\varphi,\Gamma}^{\et}(\AK) \arrow[rr, "\text{\cite{Fontaine1990}}","\cong "']&{}&\Rep_{\mathbb{Z}_p}(G_K) \\ 
\Mod_{\varphi,\Gamma}^{\fh, \crys}(\AKp)\arrow[rr, "\text{\cref{thm:Main theorem for crystalline intro}}","\cong "']\arrow[u, "\otimes_{\AKp}\AK",hook]&{} &\Rep_{\mathbb{Z}_p}^{\crys}(G_K) \arrow[u, hook]\\ 
&\Vect^{\varphi}(\mathcal{O}_{K,\Prism},\mathcal{O}_{\Prism})\arrow[ru, "\text{\cite[Theorem 5.6]{BhattScholze2023}}"',"\cong"]\arrow[lu, "\text{evaluate at $(\AKp,\varphi(\xi))$}","\cong "'] &
\end{tikzcd}\]

Our theorem has connections with arithmetic Breuil-Kisin-Fargues modules developed in \cite{Du2022}. The biggest difference between them is the coefficient rings. 
Arithmetic Breuil-Kisin-Fargues modules use the ring $A_\inf:=W(\mathcal{O}_{\mathbb{C}_p}^\flat)$ and crystalline $(\varphi,\Gamma)$-modules use the ring $\AKp:=W(\mathcal{O}_{\Kinf}^{\flat})$. As a result, we consider $G_K$-actions in the former case, whereas we consider $\Gamma_K$-actions in the latter case. In this sense, crystalline $(\varphi,\Gamma)$-modules are simpler than arithmetic Breuil-Kisin-Fargues modules. Another difference is that \cite{Du2022} relies on the Fargues-Fontaine curve and filtered $(\varphi,N)$-modules, whereas we rely on prisms.

Let us also mention the theory of $(\varphi,\widehat{G})$-modules developed in \cite{Liu2010}. The category of $(\varphi,\widehat{G})$-modules is equivalent to the category of free semi-stable $\mathbb{Z}_p$-representations of $G_K$. Using this theory, we can also classify free crystalline $\mathbb{Z}_p$-representations of $G_K$, see for example \cite[Theorem 3.7]{Ozeki2018} and \cite[F.11 Theorem]{EmertonGee}. 
The main difference between $(\varphi,\widehat{G})$-modules and crystalline $(\varphi,\Gamma)$-modules is that ($\varphi,\widehat{G}$)-modules use the field generated by $p$-power roots of a uniformizer $\pi\in \mathcal{O}_K$ over $K$. 
Since it depends on the choice of $p$-power roots of a uniformizer, the theory of $(\varphi,\widehat{G})$-modules does not behave well with respect to the change of $p$-adic field $K$ to its finite extension. However, our theory behaves well in this respect, see \cref{prop:scalar extension of crystalline phi gamma module}. Moreover, this field is not a Galois extension of $K$. As a result, we consider its Galois closure and use descent arguments in that theory.
In contrast, we do not need descent arguments in our theory since we only use the $p$-cyclotomic extension of $K$.

Finally, we note that we will give similar results for semi-stable representations and de Rham representations in the forthcoming papers.


\subsection*{Acknowledgements}
I would like to express my sincere gratitude to my advisor Takeshi Tsuji for everything he has done for me including guiding me to $p$-adic Galois representations and $p$-adic Hodge theory, suggesting many ideas, pointing out some mistakes and checking the draft. This article would not exist without his help. I am deeply grateful to Ahmed Abbes for his advice on academic writing. 
I would also like to thank Abhinandan for kindly answering my questions about Wach modules and checking the draft. 
I heartily thank Laurent Berger for informing me of \cite[Remarque on page 393]{Wach1996} and the analogy between the $\varphi$-finite part and \cite[Remark 2.6.5]{BFYSB2022}. 
I am also grateful to Adriano Marmora for pointing out the importance of \cref{prop:scalar extension of crystalline phi gamma module}.
This work was supported by the WINGS-FMSP program at Graduate School of Mathematical Sciences, the University of Tokyo. This work was partially supported by JSPS KAKENHI Grant Number 25KJ1165.

\subsection*{Notation}
Throughout this paper, all rings are assumed commutative and unital. We fix a prime number $p$ and a $p$-adic field $K$, i.e., a complete discrete valuation field $K$ of characteristic $0$ with perfect residue field $k$ of characteristic $p$. 
We fix an algebraic closure $\overline{K}$ of $K$ and let $\mathbb{C}_p$ denote the completion of $\overline{K}$. For any ring $R$, $W(R)$ denotes the ring of $p$-typical Witt vectors, and for any $x\in R$, $[x]\in W(R)$ denotes the Teichm\"{u}ller lift of $x$. We define $K_0:=W(k)[1/p]$ and $P_0:=W(\overline{k})[1/p]$, where $\overline{k}$ is the residue field of $\mathbb{C}_p$. 
For any subextension $L$ in $\mathbb{C}_p/K_0$, let $\mathcal{O}_L$ denote the ring of integers and $\mathfrak{m}_{L}$ the maximal ideal. Let $v_p$ denote the valuation of $\mathbb{C}_p$ normalized by $v_p(p)=1$. We fix a uniformizer $\pi$ of $K$. 
Let $G_K$ denote the absolute Galois group $\Gal \left(\overline{K}/K\right)$ of $K$. We fix a compatible sequence of $p$-power roots of unity $(\zeta_{p^n})_n$ in $\overline{K}$ and let $K_\infty$ denote the $p$-cyclotomic extension $\cup_{n=1}^\infty K(\zeta_{p^n})$ of $K$ and $\Kinf$ its completion. 
We define $\Gamma_K:=\Gal \left(K_\infty/K\right)$ and $H_K:=\Gal \left(\overline{K}/K_\infty\right)$. Let $k_\infty$ denote the residue field of $\Kinf$. 

Let $\mathcal{O}_{\Kinf}^\flat $ denote the tilt $ \varprojlim_{x\mapsto x^p}\mathcal{O}_{\Kinf}/p \mathcal{O}_{\Kinf}$ of $ \mathcal{O}_{\Kinf}$ and $\Kinf^\flat:=\Frac (\mathcal{O}_{\Kinf}^\flat)$ the field of fractions of $\mathcal{O}_{\Kinf}^\flat$. Similarly, we define $\mathcal{O}_{\mathbb{C}_p}^\flat$ and $\mathbb{C}_p^\flat$. Let $v$ denote the valuation of $\mathbb{C}_p^\flat$ induced by $v_p$.
We set $\AKp:=W(\mathcal{O}_{\Kinf}^\flat)$, $\AK:=W(\Kinf^\flat)$, and $A_\inf:=W(\mathcal{O}_{\mathbb{C}_p}^\flat)$, and we write $\varphi$ for their Frobenius endomorphisms. Note that $\AKp,\AK$ are the notation of \cite{Colmez1999} and Colmez writes $\widetilde{\mathbb{A}}$ for $A_\inf$. We also set $\BKp:=\AKp[1/p]$, $\BK:=\AK[1/p]$, and $B_\inf:=A_\inf[1/p]$.
Let $\mathfrak{m}_{\Kinf}^\flat$ (resp. $\mathfrak{m}_{\mathbb{C}_p}^\flat$) denote the maximal ideal of $\mathcal{O}_{\Kinf}^\flat$ (resp. $\mathcal{O}_{\mathbb{C}_p}^\flat$) and $W(\mathfrak{m}_{\Kinf}^\flat)$ (resp. $W(\mathfrak{m}_{\mathbb{C}_p}^\flat)$) the kernel of the homomorphism $\AKp  \twoheadrightarrow  W(k_\infty)$ (resp. $A_\inf \twoheadrightarrow  W(\overline{k})$). 
We define $\epsilon:=(1,\zeta_p,\zeta_{p^2},\dots )\in \mathcal{O}_{\Kinf}^\flat$ and $\mu:=[\epsilon]-1\in \AKp$; then $\xi:=\mu/\varphi^{-1}(\mu)\in \AKp$ is a generator of the kernel of Fontaine's map $\theta\colon A_\inf  \rightarrow \mathcal{O}_{\mathbb{C}_p}$. We fix an element $\a\in \mathfrak{m}_{\Kinf}^\flat$ such that $v(\a)=p v_p(\pi)$. Let $A_\crys, B_\crys^+, B_\crys, B_\dR^+$, and $B_\dR$ denote the $p$-adic period rings defined by Fontaine (\cite{Fontaine1982,Fontaine1984}). We define $t:=\log [\epsilon]:=\sum_{i=1}^\infty (-1)^{i+1}\mu^i/i\in A_\crys$ as usual.

For any set $X$ equipped with an action of a group $G$, we define 
\begin{align*}
    X^G:=\{x\in X\mid gx=x \text{ for all }g\in G\}.
\end{align*}
When $X$ is a ring, every action of a group $G$ is always assumed to be compatible with the ring structure, i.e., we assume that the map
\begin{align*}
     X  \rightarrow X,\quad x\mapsto gx
\end{align*}
is a ring homomorphism for any $g\in G$.
When $X$ and $G$ have topologies, we say the action of $G$ on $X$ is continuous if the map
\begin{align*}
     G\times X  \rightarrow X,\qquad (g,x)\mapsto gx
\end{align*}
is continuous, where the left-hand side is equipped with the product topology. For any module $M$ equipped with an endomorphism $\varphi\colon M \rightarrow M$ and $h, d\in \mathbb{Z}_{\geq 0}$, we define 
\begin{align*}
    M^{\varphi^h=p^d}:=\{x\in M\mid \varphi^h(x)=p^d x\}.
\end{align*}
When $M$ is a $\mathbb{Z}[1/p]$-module, we also define $M^{\varphi^h=p^d}$ for any $h\in \mathbb{Z}_{\geq 0}$ and $d\in \mathbb{Z}$ in a similar way.

\newpage
\section{Preliminaries}
In this section, we summarize basic properties on canonical topologies, semi-linear representations and $(\varphi,\Gamma)$-modules. 
\subsection{Canonical topologies} 

\begin{lemma}\label{lem:topology on finitely generated module}
Let $B$ be a topological ring and let $M$ be a finitely generated $B$-module. Consider the quotient topology on $M$ via a $B$-linear surjection
\begin{align*}
     q\colon B^n \twoheadrightarrow M,
\end{align*}
where $B^n$ has the product topology. Then, the topology on $M$ does not depend on the choice of the surjection $q$.
\end{lemma}
\begin{proof}
Let $q'\colon B^m \twoheadrightarrow M$ be another $B$-linear surjective homomorphism. We can construct a $B$-linear homomorphism $f\colon B^n \rightarrow B^m$ such that $q=q'\circ f$.
\[\begin{tikzcd}
B^n\arrow[rr, "f"]\arrow[rd, "q"',->>]&{} &B^m\arrow[ld, "q'",->>]\\ 
&M&
\end{tikzcd}\]
Since $f$ can be expressed by a matrix over $B$, $f$ is continuous with respect to the product topologies on $B^n$ and $B^m$. Thus any open subset in $M$ with respect to the quotient topology induced by $q'$ is also open with respect to the quotient topology induced by $q$. The same argument with $q$ and $q'$ interchanged shows the other implication, so this completes the proof.
\end{proof}

\begin{definition}\label{def:definition of the canonical topology}
For a topological ring $B$ and a finitely generated $B$-module $M$, we call the topology on $M$ defined in \cref{lem:topology on finitely generated module} the \emph{canonical topology}\footnote{This term is not standard. }.
\end{definition}

In the rest of this subsection, let $B$ be a topological ring.  The goal of this subsection is to show that every finitely generated $B$-module is a topological $B$-module with respect to the canonical topology. Note that if $M$ is a finite free $B$-module, it is clear.

\begin{proposition}
Let $M$ and $N$ be finitely generated $B$-modules and let $q\colon M\twoheadrightarrow N$ be a surjective $B$-linear homomorphism. Then the canonical topology on $N$ is the same as the quotient topology induced from $M$.
\end{proposition}
\begin{proof}
Let $B^n \twoheadrightarrow M$ be a $B$-linear surjection. Then, the composite $B^n\twoheadrightarrow M \twoheadrightarrow N$ is a $B$-linear surjective homomorphism. By \cref{def:definition of the canonical topology}, the canonical topology on $M$ is the quotient topology induced by $B^n \twoheadrightarrow M$ and the canonical topology on $N$ is the quotient topology induced by $B^n\twoheadrightarrow M \twoheadrightarrow N$. This implies the assertion.
\end{proof}

\begin{proposition}\label{prop:linear map is continuous}
Let $M$ and $N$ be finitely generated $B$-modules and let $f\colon M \rightarrow N$ be a $B$-linear homomorphism. Then $f$ is continuous with respect to the canonical topologies.
\end{proposition}
\begin{proof}
Let $q_M\colon B^m \twoheadrightarrow M$ and $q_N\colon B^n \twoheadrightarrow N$ be $B$-linear surjections. Then there exists a $B$-linear homomorphism $\psi\colon B^m  \rightarrow B^n$ such that $q_N \circ \psi=f \circ q_M$.
\[\begin{tikzcd}
B^m \arrow[r,"\exists\psi",dashed]\arrow[d,->>,"q_M"]&B^n \arrow[d,->>,"q_N"] \\ 
M \arrow[r,"f"]&N
\end{tikzcd}\]
Since $\psi$ is continuous, $(f \circ q_M)^{-1}(U)=(q_N\circ \psi)^{-1}(U)$ is open in $B^m$ for any open subset $U$ in $N$. Since the canonical topology on $M$ is the quotient topology induced by $q_M$, this shows that $f^{-1}(U)$ is open in $M$. Thus $f$ is continuous.
\end{proof}

\begin{lemma}\label{lem:open map}
Let $M$ be a finitely generated $B$-module equipped with the canonical topology. Then every $B$-linear surjective map $q\colon B^n \twoheadrightarrow M$ is an open map.
\end{lemma}
\begin{proof}
Let $U\subseteq B^n$ be an open subset. It is enough to show that $q^{-1}(q(U))\subseteq B^n$ is open. Since 
\begin{align*}
     q^{-1}(q(U))=\bigcup_{x\in \Ker q}(x+U)
\end{align*} 
and $x+U$ is open in $B^n$, $q^{-1}(q(U)) $ is also open.
\end{proof}

\begin{proposition}\label{prop:product top and canonical top}
For any finitely generated $B$-modules $M$ and $N$, the canonical topology on $M\times N$ is the same as the product topology of the canonical topologies.
\end{proposition}
\begin{proof}
Take surjections $q_1\colon B^m \twoheadrightarrow M$ and $q_2\colon B^n \twoheadrightarrow N$. For any open subsets $U\subseteq M$ and $ V\subseteq N$, 
\begin{align*}
     (q_1 \times q_2)^{-1}(U\times V)=q_1^{-1}(U) \times q_2^{-1}(V)
\end{align*}
is open in $B^{n+m}$. This shows that the canonical topology is finer than the product topology.

On the other hand, take an open subset $W\subseteq M\times N$ with respect to the canonical topology on $M\times N$. Since the topology on $B^{m+n}$ is the same as the product topology of $B^m$ and $B^n$, we can put
\begin{align*}
     (q_1 \times q_2)^{-1}(W)=\bigcup_{i\in I}U_i\times V_i
\end{align*}
for some open subsets $U_i\subseteq B^m$ and $V_i\subseteq B^n$. Because $q_1\times q_2$ is surjective, we have
\begin{align*}
     W&=q_1\times q_2\bigl((q_1\times q_2)^{-1}(W)\bigr)\\
     &=q_1\times q_2\left(\bigcup_{i\in I}U_i\times V_i\right)\\
     &=\bigcup_{i\in I}q_1(U_i)\times q_2(V_i).
\end{align*}
Since $q_1$ and $q_2$ are open maps by \cref{lem:open map}, we conclude that $W$ is open with respect to the product topology on $M\times N$.
\end{proof}

\begin{theorem}\label{thm:topological module}
For any finitely generated $B$-module $M$, $M$ is a topological $B$-module with respect to the canonical topology.
\end{theorem}
\begin{proof}
As we mentioned before, if $M$ is a finite free $B$-module, one can easily show that it is a topological $B$-module. In general, 
\begin{align*}
     M\times M  \rightarrow M,\qquad (m,n)\mapsto m+n
\end{align*}
is a $B$-linear map, so it is continuous with respect to the canonical topologies by \cref{prop:linear map is continuous}. Since the canonical topology on $M\times M$ is the same as the product topology by \cref{prop:product top and canonical top}, this shows the continuity of addition. Also, since
\begin{align*}
     M  \rightarrow M\colon m\mapsto -m
\end{align*}
is a $B$-linear map, it is continuous. It remains to show the continuity of scalar multiplication. Take
a surjection $q\colon B^n \twoheadrightarrow M$. We have the following commutative diagram.
\[\begin{tikzcd}
B\times B^n \arrow[r, "\text{scalar}"]\arrow[d, "\id\times q"',->>] & B^n\arrow[d, "q",->>]\\ 
B\times M\arrow[r, "\text{scalar}"]&M
\end{tikzcd}\]
As the product topology on $B\times M$ is the same as the canonical topology, an open subset of $B\times M$ is a subset which is open after pulling back to $B\times B^n$. As $B$ is a topological ring, the scalar multiplication on $B^n$ is continuous. Then, the continuity of scalar multiplication on $M$ follows from them.
\end{proof}

\begin{proposition}\label{prop:adic topology}
Let $B$ be a ring with the $I$-adic topology for an ideal $I\subseteq B$ and let $M$ be a finitely generated $B$-module. Then the canonical topology on $M$ coincides with the $I$-adic topology on $M$. 
\end{proposition}
\begin{proof}
It is easy to see that the assertion holds when $M$ is a finite free $B$-module. For a general finitely generated $B$-module $M$, take a $B$-linear surjective map $q\colon B^n \twoheadrightarrow M$. Let $U\subseteq M$ be an open neighborhood of $0$ with respect to the canonical topology. Then $q^{-1}(U)\subseteq B^n$ contains $I^rB^n$ for some $r\geq 0$, which means that $U$ contains $I^rM$. Together with \cref{thm:topological module}, this implies that the $I$-adic topology is finer than the canonical topology. It remains to show that for any $r\geq 0$, $q^{-1}(I^rM)\subseteq B^n$ is open with respect to the $I$-adic topology. This follows from the fact that $q^{-1}(I^rM)$ is a group that contains the open subgroup $I^rB^n$. 
\end{proof}

\subsection{Semi-linear representations}

\begin{definition}[cf. {\cite[Definition 1.1.]{MorrowTsuji2021}}]\label{def:semi-linear representations}
Let $G$ be a topological group and let $B$ be a topological ring equipped with a continuous action by $G$. A \emph{semi-linear representation of $G$ over $B$} is a finite projective $B$-module equipped with a semi-linear continuous action by $G$ with respect to the canonical topology.
\end{definition}


\begin{proposition}\label{prop:equivalent condition of continuity}
Let $G$ be a topological group and let $B$ be a topological ring equipped with a continuous action by $G$. Let $M$ be a finitely generated $B$-module equipped with a semi-linear $G$-action. Let $e_1,\dots ,e_d$ be generators of $M$ over $B$. We equip $M$ with the canonical topology.
Then, the $G$-action on $M$ is continuous if and only if the maps
     \begin{align*}
         G \rightarrow M,\quad g\mapsto ge_i
     \end{align*}
are continuous for $i=1,\dots ,d$.
\end{proposition}
\begin{proof}
The ``only if'' part follows from the continuity of the inclusions $G\rightarrow G\times M,\ g\mapsto (g,e_i)$ for $i=1,\dots ,d$, so we prove the ``if'' part. By using \cref{lem:open map}, we see that the product topology on $G\times M$ is the same as the quotient topology induced by the surjection $G\times B^d \twoheadrightarrow G\times M,\ (g, (b_i)_i)\mapsto (g, b_1e_1+\dots+ b_de_d)$ (cf. the proof of \cref{prop:product top and canonical top}). Thus, it suffices to show that the map 
\begin{align}\label{eq:action}
     G\times B^d  \rightarrow M,\quad \bigl(g,(b_i)_i\bigr)\mapsto \sum_{i=1}^d g(b_i)g(e_i)
\end{align}
is continuous. Consider the following maps
\begin{alignat*}{2}
     G\times B^d  &\rightarrow G\times B^d,&\quad \bigl(g, (b_i)_i\bigr)&\mapsto\bigl( g, (gb_i)_i\bigr),\\
     G\times B^d  &\rightarrow M^d\times B^d,&\quad \bigl(g, (b_i)_i\bigr) &\mapsto \bigl((ge_i)_i, (b_i)_i\bigr), \text{ and}\\
     M^d\times B^d  &\rightarrow M^d  \rightarrow M,&\quad \bigl((m_i)_i, (b_i)_i\bigr)&\mapsto (b_im_i)_i\mapsto \sum_{i=1}^d b_im_i.
\end{alignat*}
Then, the composite is the same as the map \eqref{eq:action}. So it is enough to show that every map is continuous. Since the $G$-action on $B$ is continuous, the first map is continuous. The continuity of the second map follows from the assumption. The last map is continuous by \cref{thm:topological module}.
\end{proof}

Semi-linear representations are stable under scalar extensions:
\begin{proposition}[Scalar extension]\label{prop:scalar extension of semi-linear representations}
Let $A$ and $B$ be topological rings equipped with a continuous action by a topological group $G$. Let $f\colon A \rightarrow B$ be a $G$-equivariant continuous ring homomorphism. Then for any finitely generated $A$-module $M$ equipped with a semi-linear and continuous $G$-action with respect to the canonical topology, the $G$-action on $M\otimes_A B$ defined by
\begin{align}\label{eq:scalar extension of semi-linear representations}
     G\times (M\otimes_A B)  \rightarrow M \otimes_A B,\qquad (g, m\otimes b)\mapsto gm \otimes gb.
\end{align}
is semi-linear and continuous with respect to the canonical topology.
\end{proposition}
\begin{proof}
Semi-linearity is obvious by definition. Let us check the continuity. Let $e_1,\dots ,e_d$ be generators of $M$ over $A$. Then, $e_1 \otimes 1,\dots ,e_d \otimes 1$ are generators of $M\otimes_A B$ over $B$. By \cref{prop:equivalent condition of continuity}, it is enough to show that the natural map $\iota\colon M  \rightarrow M\otimes_A B$ is continuous. Take an $A$-linear surjective homomorphism $q\colon A^d \twoheadrightarrow M$. We have a commutative diagram
\[\begin{tikzcd}
 A^d \arrow[r, ->>,"q"]\arrow[d,"f^d"']&M\arrow[d,"\iota"] \\ 
 A^d \otimes_A B\arrow[r, ->>,"q\otimes\id_B"'] &M\otimes_A B.
\end{tikzcd}\]
Since the topology on $M$ is the quotient topology induced by $q$, it suffices to prove that for any open subset $U$ of $M\otimes_A B$, $(\iota\circ q)^{-1}(U)$ is open in $A^d$. It follows from the commutativity of the diagram and the continuity of $q \otimes \id_B$ and $f$.
\end{proof}

\begin{lemma}\label{lem:continuity of subgroup}
     Let $G$ be a topological group and let $B$ be a topological ring equipped with a continuous action by $G$.
     Assume that the topology on $B$ is the $I$-adic topology for some $G$-stable ideal $I$. Let $M$ be a $B$-module equipped with a semi-linear $G$-action. Let $H$ be an open subgroup of $G$. Then, if the $H$-action on $M$ is continuous with respect to the $I$-adic topology, the $G$-action is also continuous.
     \end{lemma}
     \begin{proof}
          Fix an arbitrary $(g,x)\in G\times M$. We deduce the continuity of the $G$-action at $(g,x)$ from the continuity of the $H$-action. Take any $n\in \mathbb{Z}_{\geq 0}$.  
          Since 
          \begin{align*}
               H \hookrightarrow  H\times M  \xrightarrow{\text{$H$-action}} M,\qquad h\mapsto (h, gx)\mapsto hgx
          \end{align*}
          is continuous, there exists an open neighborhood $1\in H_1\subseteq H$ such that for every $h\in H_1$, $hgx\in gx+I^nM$. Then for any $hg\in H_1g$ and $x+y\in x+ I^nM$, we have $hg(x+y)=hgx+hgy$. We see that $hgx\in gx+I^nM$ by the definition of $H_1$, and $hgy\in I^nM$ as $I^nM$ is stable under the $G$-action. Therefore, $hg(x+y)\in gx+I^nM$ as desired.
\end{proof}

The following definition is due to M. Morrow and T. Tsuji.
\begin{definition}[{\cite[Definition 1.1.]{MorrowTsuji2021}}]\label{def:trivial mod I}
Let $G$ be a group and let $B$ be a ring equipped with a $G$-action. Let $I\subseteq B$ be an ideal of $B$ stable under the action of $G$ and let $M$ be a $B$-module equipped with a semi-linear $G$-action. We say the $G$-action on $M$ is \emph{trivial modulo $I$} if the $(B/I)^G$-module $(M/I)^G$ is finite projective and the canonical map $(M/I)^G \otimes_{(B/I)^G}B/I  \rightarrow M/I$ is an isomorphism.
\end{definition}

\begin{proposition}\label{prop:Galois fix part}
Let $A$ be a ring, $M$ a finite projective $A$-module and $G$ a group. We endow $M$ (resp. $A$) with the trivial $G$-action, i.e., $gm=m$ (resp. $ga=a$) for every $g\in G$ and $m\in M$ (resp. $a\in A$). Let $B$ be an $A$-algebra equipped with a $G$-action which is compatible with the algebra structure $A \rightarrow B$. Then, we have an isomorphism
\begin{align*}
     M \otimes_{A} B^G  \xrightarrow{\cong } (M\otimes_{A}B)^G. 
\end{align*}
\end{proposition}
\begin{proof}
This proposition is obvious when $M$ is a free $A$-module. For a general finite projective $A$-module $M$, let $N$ be a finite projective $A$-module such that $M\oplus N$ is a finite free $A$-module. We equip $N$ with the trivial $G$-action. We have the following commutative diagram. 
\[\begin{tikzcd}
(M\oplus N)\otimes_{A} B^G \arrow[r, "\cong "] \arrow[d, "\cong ", phantom,sloped]&\Bigl((M\oplus N)\otimes_{A}B\Bigr)^G \arrow[d, "\cong ", phantom,sloped]  \\ 
\bigl(M\otimes_{A} B^G\bigr) \oplus \bigl(N\otimes_{A}B^G\bigr) \arrow[r]&(M\otimes_{A}B)^G \oplus (N\otimes_{A}B)^G
\end{tikzcd}\]
Since $M\oplus N$ is free, the top arrow is an isomorphism. Thus the bottom arrow is also an isomorphism, whose left component gives the desired isomorphism.
\end{proof}

\begin{proposition}\label{prop:trivial mod I implies mod J}
Let $G$ be a group and let $B$ be a ring equipped with a $G$-action. Let $I\subseteq J\subseteq B$ be ideals of $B$ stable under the action of $G$. Let $M$ be a $B$-module equipped with a semi-linear $G$-action. If the $G$-action on $M$ is trivial modulo $I$, then it is trivial modulo $J$.
\end{proposition}
\begin{proof}
By tensoring the isomorphism $(M/I)^G\otimes_{(B/I)^G} B/I  \xrightarrow{\cong } M/I$ with $B/J$ and by \cref{prop:Galois fix part}, we have isomorphisms
\begin{align*}
     (M/I)^G\otimes_{(B/I)^G} B/J  &\xrightarrow{\cong } M/J,\\
     (M/I)^G\otimes_{(B/I)^G} (B/J)^G  &\xrightarrow{\cong } (M/J)^G.
\end{align*}
Thus $(M/J)^G$ is a finite projective $(B/J)^G$-module and the canonical map
\begin{align*}
     (M/J)^G\otimes_{(B/J)^G}B/J  \rightarrow M/J
\end{align*}
is an isomorphism.
\end{proof}

\begin{corollary}\label{cor:mod mu implies mod xi}
Let $N$ be an $\AKp$-module equipped with a semi-linear $\Gamma_K$-action. If the $\Gamma_K$-action is trivial modulo $\mu$, then it is trivial modulo $\xi$ and trivial modulo $W(\mathfrak{m}_{\Kinf}^\flat)$.
\end{corollary}
\begin{proof}
This follows from \cref{prop:trivial mod I implies mod J} and the facts that $\mu=\xi \varphi^{-1}(\mu)$ and $\mu\in W(\mathfrak{m}_{\Kinf}^\flat)$.
\end{proof}

\begin{proposition}\label{prop:triviality of subgroup}
Let $G$ be a group and let $B$ be a ring equipped with a $G$-action. 
Let $H\subseteq G$ be a subgroup of $G$ and $I\subseteq B$ be an ideal of $B$ stable under the $G$-action. Let $M$ be a $B$-module equipped with a semi-linear $G$-action. If the $G$-action on $M$ is trivial modulo $I$, then the $H$-action on $M$ is also trivial modulo $I$.
\end{proposition}
\begin{proof}
By taking the $H$-fixed part of the isomorphism $(M/I)^G \otimes_{(B/I)^G} B/I  \xrightarrow{\cong } M/I $, we obtain an isomorphism
\begin{align*}
     (M/I)^G \otimes_{(B/I)^G} (B/I)^H  \xrightarrow{\cong } (M/I)^H 
\end{align*}
by \cref{prop:Galois fix part}. Thus $(M/I)^H$ is a finite projective $(B/I)^H$-module and the canonical map 
\begin{align*}
     (M/I)^H \otimes_{(B/I)^H} B/I  \rightarrow M/I
\end{align*}
is an isomorphism.
\end{proof}

\subsection{$(\varphi,\Gamma)$-modules}

\begin{definition}\label{def:(phi G)-modules}
Let $B$ be a ring equipped with an endomorphism $\varphi_B\colon B \rightarrow B$. We fix a non-zero divisor $\widetilde{\xi}_B\in B$.
\begin{enumerate}
     \item A \emph{$\varphi$-module over $(B,\varphi_B, \widetilde{\xi}_B)$} is a finite projective $B$-module $M$ equipped with a $B[1/\widetilde{\xi}_B]$-linear homomorphism
     \begin{align*}
         \varPhi_M\colon  M\otimes_{B,\varphi_B} B[1/\widetilde{\xi}_B]  \rightarrow M[1/\widetilde{\xi}_B].
     \end{align*}
     Following J.-M. Fontaine's notation (\cite{Fontaine1990}), $\varphi_M\colon M \rightarrow M[1/\widetilde{\xi_B}]$ denotes the semi-linear map induced by $\varPhi_M$. We call $\varphi_M$ a \emph{Frobenius map on $M$}.
     A \emph{morphism of $\varphi$-modules} $f\colon (M_1, \varPhi_{M_1}) \rightarrow (M_2, \varPhi_{M_2})$ is a $B$-linear map $f\colon M_1 \rightarrow M_2$ which makes the following diagram commutative.
     \[\begin{tikzcd}
     M_1\otimes_{B,\varphi_B}B[1/\widetilde{\xi_B}]\arrow[r, "\varPhi_{M_1}"]\arrow[d, "f\otimes \id"']&M_1[1/\widetilde{\xi_B}]\arrow[d, "\text{$f[1/\widetilde{\xi_B}]$}"]\\ 
     M_2\otimes_{B,\varphi_B}B[1/\widetilde{\xi_B}]\arrow[r, "\varPhi_{M_2}"']&M_2[1/\widetilde{\xi_B}]
     \end{tikzcd}\]  

     \item A $\varphi$-module $ (M,\varPhi_M)$ is said to be \emph{of finite height} if $\varPhi_M$ is an isomorphism. 
     
     \item We further assume that $B$ is a topological ring equipped with a continuous action by a topological group $G$ which commutes with $\varphi_B\colon B  \rightarrow B$. A \emph{$(\varphi,G)$-module $(M,\varPhi_M, G)$ over $(B,\varphi_B,\widetilde{\xi}_B,G)$} is a $\varphi$-module of finite height over $(B,\varphi_B,\widetilde{\xi}_B)$ equipped with a semi-linear continuous $G$-action with respect to the canonical topology which commutes with $\varphi_M$.
     A \emph{morphism of $(\varphi,G)$-modules} $(M_1, \varPhi_{M_1}, G)  \rightarrow (M_2,\varPhi_{M_2},G)$ is a $G$-equivariant morphism of $\varphi$-modules. 
\end{enumerate}
\end{definition}

\begin{remark}\label{rem:(phi G)-modules}
If there is no risk of confusion, we just say a $\varphi$-module $M$ over $B$ (resp. $(\varphi,G)$-module $M$ over $B$) and abbreviate $\varPhi_M$ (resp. $\varphi_M$) to $\varPhi$ (resp. $\varphi$).
\end{remark}



\begin{proposition}[Scalar extension]\label{prop:scalar extension of phi-modules}
     Let $A$ and $B$ be rings equipped with endomorphisms $\varphi_A\colon A  \rightarrow A$ and $\varphi_B\colon B \rightarrow B$. Let $f\colon A  \rightarrow B$ be a ring homomorphism such that $f\circ \varphi_{A}=\varphi_B\circ f$. We fix non-zero divisors $\widetilde{\xi}_A\in A$ and $\widetilde{\xi}_B\in B$ such that $\widetilde{\xi}_B\in f(\widetilde{\xi}_A)B$. Then we have following:
     \begin{enumerate}
          \item Let $M$ be a $\varphi$-module over $(A,\varphi_A, \widetilde{\xi}_A)$. Then $M\otimes_A B$ is a $\varphi$-module over $(B,\varphi_B,\widetilde{\xi}_B)$ via
          \begin{align*}
               (M\otimes_A B)\otimes_{B,\varphi_B} B[1/\widetilde{\xi}_B] \cong (M\otimes_{A,\varphi_A}A[1/\widetilde{\xi}_{A}])\otimes_{A[1/\widetilde{\xi}_A]} B[1/\widetilde{\xi}_B] \xrightarrow{\varPhi_M \otimes\id_{B[1/\widetilde{\xi}_B]}} M\otimes_{A}B[1/\widetilde{\xi}_B].
          \end{align*}
          \item Let $M$ be a $\varphi$-module of finite height over $A$. Then the $\varphi$-module $M\otimes_A B$ is of finite height.
          \item We further assume that $A$ (resp. $B$) is a topological ring equipped with a continuous action by a topological group $G$ which commutes with $\varphi_A$ (resp. $\varphi_B$) and that $f\colon A  \rightarrow B$ is $G$-equivariant and continuous. Let $M$ be a $(\varphi,G)$-module over $A$. Then $M \otimes_A B$ is a $(\varphi,G)$-module over $B$ via the $G$-action defined by
          \begin{align*}
               G\times (M\otimes_A B)  \rightarrow M \otimes_A B,\qquad (g, m\otimes b)\mapsto gm \otimes gb.
          \end{align*}
     \end{enumerate}
\end{proposition}
\begin{proof}
The claims (1) and (2) are obvious by definition. As for the claim (3), the only non-trivial part is the continuity of the $G$-action on $M\otimes_A B$, which we proved in \cref{prop:scalar extension of semi-linear representations}.
\end{proof}

In the following, we use the following special cases of $(\varphi,G)$-modules.

\begin{definition}\label{def:phi Gamma module over AKp and AK}
We endow the weak topology on $\AKp$ and $\AK$, i.e., the product topology on $\AKp= (\mathcal{O}_{\Kinf}^\flat)^{\mathbb{N}}$ and $\AK= (\Kinf^\flat)^{\mathbb{N}}$. Note that the equalities are those of sets induced by the definition of $p$-typical Witt vectors. Then $\AKp$ and $\AK$ are topological rings with respect to the weak topologies and the natural $\Gamma_K$-actions are continuous. Let $\varphi\colon \AKp  \rightarrow \AKp$ (resp. $\varphi\colon \AK  \rightarrow \AK$) denote the Frobenius map of Witt vectors.
\begin{enumerate}
     \item A {\em$(\varphi,\Gamma)$-module over $\AKp$} is a finite free $\AKp$-module $N$ equipped with a semi-linear continuous $\Gamma_K$-action and a $\Gamma_K$-equivariant $\AKp[1/\varphi(\xi)]$-linear isomorphism
     \begin{align*}
          \varPhi_N\colon N\otimes_{\AKp,\varphi}\AKp[1/\varphi(\xi)] \xrightarrow{\cong } N[1/\varphi(\xi)].
     \end{align*}
     We write $\Mod_{\varphi,\Gamma}^{\fh}(\AKp)$ for the category of $(\varphi,\Gamma)$-modules over $\AKp$.
     \item A {\em$(\varphi,\Gamma)$-module over $\AK$} is a finite free $\AK$-module $M$ equipped with a semi-linear continuous $\Gamma_K$-action and a $\Gamma_K$-equivariant $\AK$-linear isomorphism 
     \begin{align*}
          \varPhi_M\colon M\otimes_{\AK,\varphi}\AK  \xrightarrow{\cong }M.
     \end{align*}
     We write $\Mod_{\varphi,\Gamma}^{\et}(\AK)$ for the category of $(\varphi,\Gamma)$-modules over $\AK$.
\end{enumerate}
\end{definition}

\begin{remark}
     Our terminology and definitions are slightly different from \cite{Fontaine1990}. First, he only assumed that the underlying modules are finitely generated. Also, the Frobenius map $\varPhi_M$ is not necessarily isomorphic and he defined a $(\varphi, \Gamma)$-module to be \'{e}tale if the Frobenius map is an isomorphism. Thus, our $(\varphi,\Gamma)$-module over $\AK$ is a ``free \'{e}tale $(\varphi,\Gamma)$-module over $\AK$'' in the sense of \cite{Fontaine1990}. Finite height in our sense is also different from that in \cite{Fontaine1990}.
\end{remark}

These definition are compatible with \cref{def:(phi G)-modules} by the lemmas below.

\begin{lemma}[{\cite[Chap. II, \S 3, no.2, Corollaire 2]{Bourbaki1961}}]\label{lem:finite projective over local ring}
Every finite projective module over a local ring is free.
\end{lemma}

\begin{lemma}[{\cite[Lemma 2 (1)]{Tsuji_simons}}]\label{lem:Ainf is a local ring}
     $\AKp$ is $(p,\xi)$-adically complete and separated. It is a local ring with maximal ideal $\Ker (\AKp \overset{\theta}{\twoheadrightarrow} \mathcal{O}_{\Kinf}\twoheadrightarrow  k_\infty )=(p, [\mathfrak{m}_{\Kinf}^\flat])$.
\end{lemma}
\begin{proof}
The first assertion follows from \cite[Lemma 2 (1)]{Tsuji_simons}. We include the proof here for convenience.
We first construct a non-zero divisor $\xi_p\in\AKp$ which is another generator of $\Ker \theta$. 
It is known that $v(\epsilon-1)=p/(p-1)$. We set $\tau:= (\epsilon^{1/p}-1)^{p-1} \in \mathcal{O}_{\Kinf}^\flat$, then $v(\tau)=1$. Take an element $u\in \AKp$ such that $\theta(pu)=\tau^{(0)}$. We define $\xi_p:=pu-[\tau]\in \AKp$. Then $\theta(\xi_p)=0$. We claim that the sequence
\begin{align*}
     0 \rightarrow \AKp \xrightarrow{\times \xi_p}\AKp  \xrightarrow{\theta} \mathcal{O}_{\Kinf}  \rightarrow 0 
\end{align*}
is exact. Since each term is $p$-torsion free and $p$-adically complete and separated, it is enough to show the exactness after modulo $p$ and it follows from the fact $v(\tau)=1$. Thus we obtain a non-zero divisor $\xi_p\in \AKp$ that is a generator of $\Ker \theta$.

Then, we have $(\xi,p)\AKp=([\tau],p)\AKp$. It is known that $\AKp/p \cong \mathcal{O}_{\Kinf}^\flat$ is $\tau$-adically complete and separated. As $\AKp$ is $p$-torsion free, we have an exact sequence $0  \rightarrow \AKp/p^n  \xrightarrow{\times p} \AKp/p^{n+1} \rightarrow \AKp/p  \rightarrow 0 $ for every $n\in \mathbb{Z}_{\geq 1}$. Since $\AKp/p\cong \mathcal{O}_{\Kinf}^\flat$ is $[\tau]$-torsion free, its reduction modulo $[\tau]^m$ remains exact. By taking the inverse limit, we obtain exact sequences
\[\begin{tikzcd}
0\arrow[r]&\AKp/p^n\arrow[r,"\times p"]\arrow[d]&\AKp/p^{n+1}\arrow[r]\arrow[d]&\AKp/p\arrow[r]\arrow[d]&0 \\ 
0\arrow[r]& \varprojlim_{m}(\AKp/p^n)/[\tau]^m\arrow[r,"\times p"]&\varprojlim_{m}(\AKp/p^{n+1})/[\tau]^m\arrow[r]&\varprojlim_{m}(\AKp/p)/[\tau]^m\arrow[r]&0.
\end{tikzcd}\]
By induction on $n$, we see that $\AKp/p^n$ is $[\tau]$-adically complete and separated for every $n\geq 1$. 
Therefore, we conclude that
\begin{align*}
     \AKp\cong  \varprojlim_{n}\AKp/p^n \cong  \varprojlim_{n}\Bigl( \varprojlim_{m}\AKp/(p^n, [\tau]^m)\Bigr)\cong  \varprojlim_{n}\AKp/(p,[\tau])^n.
\end{align*}
In order to prove that $\AKp$ is a local ring, it is enough to show that every element $x\in\AKp\setminus (p, [\mathfrak{m}_{\Kinf}^\flat])$ is invertible in $\AKp$. We first remark that $(p,[\mathfrak{m}_{\Kinf}^\flat])$ is the kernel of the surjective homomorphism $\AKp \twoheadrightarrow \mathcal{O}_{\Kinf}^\flat \twoheadrightarrow k_\infty$. As $\mathcal{O}_{\Kinf}^\flat$ is a local ring, $x \bmod p$ is invertible in $\mathcal{O}_{\Kinf}^\flat$. So there exists $y\in \AKp$ such that $xy\in 1+p\AKp$. Since $\AKp$ is $p$-adically complete and separated, we see that $1+p\AKp \subseteq (\AKp)^\times$. Hence $x$ is invertible, as desired. Since the homomorphism $\AKp \overset{\theta}{\twoheadrightarrow} \mathcal{O}_{\Kinf}\twoheadrightarrow  k_\infty$ coincides with $\AKp \twoheadrightarrow \mathcal{O}_{\Kinf}^\flat \twoheadrightarrow k_\infty$, the equality of ideals hold.
\end{proof}

Let $\Rep_{\mathbb{Z}_p}(G_K)$ denote the category of free $\mathbb{Z}_p$-representations of $G_K$, i.e., finite free $\mathbb{Z}_p$-modules equipped with continuous and linear $G_K$-actions. The following theorem is the origin of the theory of $(\varphi,G)$-modules.

\begin{theorem}[J.-M. Fontaine]\label{thm:(phi Gamma) modules}
The category of $(\varphi,\Gamma)$-modules over $\AK$ is equivalent to the category of free $\mathbb{Z}_p$-representations of $G_K$ via the functors
\[\begin{tikzcd}[row sep=tiny]
     \Mod_{\varphi,\Gamma}^{\et}(\AK)\arrow[r, <->, "\sim "]&\Rep_{\mathbb{Z}_p}(G_K) \\ 
     M\arrow[r, mapsto]\arrow[u, "\rotatebox{90}{$\in$}",phantom]& \bigl(M\otimes_{\AK} W(\mathbb{C}_p^\flat)\bigr)^{\varphi=1}\arrow[u, "\rotatebox{90}{$\in$}",phantom] \\[-0.25cm]
     \bigl(T\otimes_{\mathbb{Z}_p}W(\mathbb{C}_p^\flat)\bigr)^{H_K}&\arrow[l, mapsto] T.
\end{tikzcd}\]
\end{theorem}

\begin{remark}\label{rem:(phi Gamma) modules}
    In \cite{Fontaine1984}, Fontaine claimed an equivalence of categories between the category of \'{e}tale $(\varphi,\Gamma)$-modules over $\AK[1/p]$ and the category of $p$-adic Galois representations of $G_K$ without giving a proof. Then in \cite{Fontaine1990}, he constructed a subring $\mathbb{A}_K$ of $\AK$ stable under the Frobenius map and the $\Gamma_K$-action, and showed an equivalence of categories between the category of \'{e}tale $(\varphi,\Gamma)$-modules over $\mathbb{A}_K[1/p]$ (resp. $\mathbb{A}_K$) and the category of $p$-adic Galois representations of $G_K$ (resp. $\mathbb{Z}_p$-representations of $G_K$, i.e., finitely generated $\mathbb{Z}_p$-modules equipped with continuous and linear $G_K$-actions.). The above theorem can be shown by a similar argument. Note that $\AK[1/p]$ is denoted by $\Frac W(R_L)$ in \cite{Fontaine1984} and $\mathbb{A}_K$ is denoted by $\mathcal{O}_{\mathcal{E}}$ and $\mathbb{A}_K[1/p]$ is denoted by $\mathcal{E}$ in \cite{Fontaine1990}.
\end{remark}

Combining it with \cref{prop:scalar extension of phi-modules}, we obtain the following corollary:

\begin{corollary}\label{cor:well-definedness as p-adic rep}
There exists a natural functor
\[\begin{tikzcd}[column sep=large,row sep=small]
T\colon \Mod_{\varphi,\Gamma}^{\fh}(\AKp)\arrow[r, "\otimes_{\AKp}\AK"]&\Mod_{\varphi,\Gamma}^{\et}(\AK)\arrow[r,"\cong "]&\Rep_{\mathbb{Z}_p}(G_K) \\ 
N\arrow[r,mapsto]\arrow[u, "\rotatebox{90}{$\in$}",phantom]&N\otimes_{\AKp}\AK\arrow[r, mapsto]\arrow[u, "\rotatebox{90}{$\in$}",phantom]&\bigl(N\otimes_{\AKp}W(\mathbb{C}_p^\flat)\bigr)^{\varphi=1}.\arrow[u, "\rotatebox{90}{$\in$}",phantom]
\end{tikzcd}\]
Also, we have an equality $\rank_{\AKp} N=\rank_{\mathbb{Z}_p} T(N)$.
\end{corollary}
\begin{proof}
     By \cref{prop:scalar extension of phi-modules}, the first base change functor is well-defined. Combining this with \cref{thm:(phi Gamma) modules}, we see that the functor $T$ is well-defined. The last assertion follows from the $W(\mathbb{C}_p^\flat)$-linear isomorphism (\cite[Proof of Proposition 1.2.6]{Fontaine1990})
     \begin{align}\label{eq:isomorphism of Fontaine}
          N\otimes_{\AKp}W(\mathbb{C}_p^\flat)\cong T(N)\otimes_{\mathbb{Z}_p} W(\mathbb{C}_p^\flat).
     \end{align}
\end{proof}

Let $\Rep_{\mathbb{Z}_p}^{\crys} (G_K)$ denote the category of free crystalline $\mathbb{Z}_p$-representations of $G_K$, i.e., free $\mathbb{Z}_p$-representations of $G_K$ such that the associated $p$-adic Galois representations are crystalline. 
Our goal is to construct an equivalence of categories between the category of $(\varphi,\Gamma)$-modules over $\AKp$ with additional conditions and the category of free crystalline $\mathbb{Z}_p$-representations of $G_K$ via the above functor $T$.

\newpage
\section{Crystalline $(\varphi,\Gamma)$-modules over $\AKp$}
\subsection{Triviality modulo $\mu$}

In this subsection, we study triviality modulo $\mu$ when the coefficient ring is $\AKp$.

\begin{lemma}\label{lem:G-fixed part of a local ring}
Let $(R,\mathfrak{m})$ be a local ring equipped with an action of a group $G$. Then, $R^G$ is a local ring with maximal ideal $\mathfrak{m}^G$.
\end{lemma}
\begin{proof}
It suffices to show that $\mathfrak{m}^G$ is an ideal of $R^G$ and $R^G\setminus \mathfrak{m}^G=(R^G)^\times$. Obviously, $0\in \mathfrak{m}^G$. For any $m_1,m_2\in \mathfrak{m}^G$, we have $g(m_1+m_2)=g(m_1)+g(m_2)=m_1+m_2$ for any $g\in G$, so $m_1+m_2\in R^G \cap \mathfrak{m}=\mathfrak{m}^G$. For any $r\in R^G$, $m\in \mathfrak{m}^G$, we have $g(rm)=g(r)g(m)=rm$, so $rm\in R^G \cap \mathfrak{m}=\mathfrak{m}^G$. Thus $\mathfrak{m}^G$ is an ideal.

We next show the equality $R^G\setminus \mathfrak{m}^G=(R^G)^\times$. On the one hand, for any $r\in R^G\setminus \mathfrak{m}^G$, we have $r\in R\setminus\mathfrak{m}$ since $\mathfrak{m}^G=R^G\cap \mathfrak{m}$. So, $r$ is invertible in $R$. Then, $g(r^{-1}) r=g(r^{-1}) g(r)=g(r r^{-1})=g(1)=1$ for any $g\in G$. So, $g(r^{-1})=r^{-1}$ for any $g\in G$. This shows that $r^{-1}\in R^G$ and $r\in (R^G)^\times$. On the other hand, for any $r\in (R^G)^\times$, $r$ is also invertible in $R$. So $r\in R\setminus \mathfrak{m}$. Thus $r\in R^G\setminus \mathfrak{m}^G$.
\end{proof}

\begin{proposition}\label{prop:G-fixed part of a local ring}
$(\AKp/\mu)^{\Gamma_K}$ is a local ring and every finite projective $(\AKp/\mu)^{\Gamma_K}$-module is free.
\end{proposition}
\begin{proof}
The assertion follows from \cref{lem:Ainf is a local ring,lem:finite projective over local ring,lem:G-fixed part of a local ring}.
\end{proof}

\begin{proposition}[{\cite[Lemma 1.3]{MorrowTsuji2021}}]\label{prop:automatic continuity}
Let $N$ be a finite free $\AKp$-module equipped with a semi-linear $\Gamma_K$-action. If the $\Gamma_K$-action on $N$ is trivial modulo $\mu$, then it is continuous.
\end{proposition}
\begin{proof}
We remark that the weak topology on $\AKp$ is the same as the $(p,\mu)$-adic topology, so the canonical topology on $N$ is the same as the $(p,\mu)$-adic topology by \cref{prop:adic topology}. Since $\Gamma_K$ is an open subgroup of $\mathbb{Z}_p^\times$ via the cyclotomic character, there exists an open subgroup $\Gamma'\subseteq \Gamma_K$ which is isomorphic to $\mathbb{Z}_p$ as a topological group. The $\Gamma'$-action on $N$ is trivial modulo $(p,\mu)$ by \cref{prop:trivial mod I implies mod J,prop:triviality of subgroup}. Hence the $\Gamma'$-action on $N$ is continuous by \cite[Lemma 1.3]{MorrowTsuji2021}. Then the assertion follows from \cref{lem:continuity of subgroup}.
\end{proof}

We prove two lemmas needed for \cref{prop:Tsuji Prop 76}.

\begin{lemma}[{\cite[Lemma 74]{Tsuji_simons}}]\label{lem:Tsuji Lemma 74}
We have 
\begin{align*}
     (\AKp/\mu)(r)^{\Gamma_K}=0
\end{align*}
for any $r\in \mathbb{Z}\setminus \{0\}$, where $(r)$ denotes the $r$-th Tate twist.
\end{lemma}
\begin{proof}
Let $\chi_\cyc\colon \Gamma_K  \rightarrow \mathbb{Z}_p^\times$ denote the cyclotomic character of $\Gamma_K$. Let $a\in \AKp$ be an element such that $\chi_\cyc^r(g)g(a)-a\in \mu\AKp$ for all $g\in \Gamma_K$. By taking $\varphi^m$ for $m\in \mathbb{Z}_{\geq 0}$, we obtain
\begin{align*}
     \chi_\cyc^r(g)g(\varphi^m(a))-\varphi^m(a)\in \varphi^m(\mu)\AKp\subseteq \xi\AKp.
\end{align*}
Thus for all $m\in \mathbb{Z}_{\geq 0}$, we have
\begin{align*}
     \varphi^m(a)\in \bigl(\AKp/\xi\bigr)(r)^{\Gamma_K}=\mathcal{O}_{\Kinf}(r)^{\Gamma_K}=\{0\}.
\end{align*}
Here the last equality follows from \cite[Theorem 2]{Tate1967}. This shows $\varphi^m(a)\in \xi\AKp$ for all $m\in \mathbb{Z}_{\geq 0}$. By the same argument as \cite[5.1.4. Lemme. i)]{Fontaine1994} or taking the $H_K$-invariant part of it, we have an equality
\begin{align*}
     \mu\AKp=\{a\in \AKp\mid \varphi^m(a)\in \Ker \theta\text{ for all }m\in \mathbb{Z}_{\geq 0}\}.
\end{align*}
Thus we conclude $a\in \mu\AKp$.
\end{proof}

\begin{lemma}[{\cite[Lemma 75]{Tsuji_simons}}]\label{lem:Tsuji Lemma 75}
For any $r\in \mathbb{Z}$, we have a $\Gamma_K$-equivariant $\AKp$-linear isomorphism
\begin{align*}
     \AKp/\mu (-r)  \xrightarrow[\cong]{\times \mu^{-r}} \mu^{-r}\AKp/\mu^{-r+1}\AKp. 
\end{align*}
\end{lemma}
\begin{proof}
The map is obviously $\AKp$-linear bijective map. We check that this is $\Gamma_K$-equivariant. It is enough to show that $g(\mu)^{-r}= \chi_\cyc(g)^{-r}\mu^{-r}$ in $\mu^{-r}\AKp/\mu^{-r+1}\AKp$. Since
\begin{align*}
     g(\mu)=(\mu+1)^{\chi_\cyc(g)}-1
     =\sum_{n=0}^\infty \begin{pmatrix}
     \chi_\cyc(g)\\
     n
     \end{pmatrix}\mu^n-1
     =\chi_\cyc(g)\mu(1+\mu a_g)
\end{align*}
for some $a_g\in \AKp$, we see
\begin{align*}
     g(\mu)^{-r}=\chi_\cyc(g)^{-r}\mu^{-r}(1+\mu a_g)^{-r}\equiv \chi_\cyc(g)^{-r}\mu^{-r}\qquad \bmod \mu^{-r+1}\AKp.
\end{align*}
This completes the proof.
\end{proof}

The next proposition is needed to show \cref{thm:Main theorem for crystalline}. More precisely, it shows the functor in \cref{thm:Main theorem for crystalline} is full. 

\begin{proposition}[{\cite[Proposition 76]{Tsuji_simons}}]\label{prop:Tsuji Prop 76}
We have the following:
\begin{enumerate}
     \item Let $N$ be a finite free $\AKp$-module equipped with a semi-linear $\Gamma_K$-action. Then there exists at most one $\Gamma_K$-stable finite free $\AKp$-submodule $N'$ of $N[1/\mu]$ satisfying the following properties:
     \begin{enumerate}
          \item The homomorphism $N'[1/\mu] \rightarrow N[1/\mu]$ is an isomorphism.
          \item The $\Gamma_K$-action on $N'$ is trivial modulo $\mu$.
     \end{enumerate}
     \item Let $N_1$ and $N_2$ be finite free $\AKp$-modules endowed with a semi-linear $\Gamma_K$-action trivial modulo $\mu$. Then any $\Gamma_K$-equivariant $\AKp[1/\mu]$-linear homomorphism $f\colon N_1[1/\mu] \rightarrow N_2[1/\mu]$ satisfies $f(N_1)\subseteq N_2$.
\end{enumerate}
\end{proposition}
\begin{proof}
The claim (1) can be deduced from (2), so let us prove (2). Assume $f(N_1)\nsubseteq N_2$. Then there exists $r>0$ such that 
\begin{align*}
     f(N_1)\subseteq \mu^{-r} N_2,\qquad f(N_1) \nsubseteq \mu^{-r+1} N_2.
\end{align*}
\cref{lem:Tsuji Lemma 75} implies $f$ induces a non-zero $\Gamma_K$-equivariant $\AKp$-linear homomorphism
\begin{align*}
     F\colon N_1/\mu  \xrightarrow{f\bmod \mu} \mu^{-r}N_2/\mu^{-r+1}N_2 \cong  N_2\otimes_{\AKp}\bigl(\mu^{-r}\AKp/\mu^{-r+1}\AKp\bigr) \xleftarrow[\cong ]{\times \mu^{-r}} N_2\otimes_{\AKp} \AKp/\mu(-r).
\end{align*}
By the triviality modulo $\mu$, $N_1/\mu$ has a $\Gamma_K$-invariant basis $e_1,\dots ,e_d$. Since $F\neq 0$, there exists some $e_i$ such that $F(e_i)\neq 0$. This shows that $N_2/\mu(-r)$ has a non-zero element fixed by $\Gamma_K$. However, we have the $\Gamma_K$-equivariant isomorphism
\begin{align*}
     (N_2/\mu)^{\Gamma_K} \otimes_{(\AKp /\mu)^{\Gamma_K}} \AKp/\mu(-r)  \xrightarrow{\cong } N_2/\mu(-r).
\end{align*}
\cref{prop:Galois fix part,lem:Tsuji Lemma 74} give us a contradiction.
\end{proof}

\subsection{The ring $\overline{B}$, the $\varphi$-finite part, and the $\varphi$-isocrystal part}
We first study the ring $\overline{B}$ defined by L. Fargues and J.-M. Fontaine and its analogs. 

\begin{lemma}\label{lem:Bbar}
     The subset 
     \begin{align*}
          \mathfrak{q}&:=\{[a]x \mid a\in \mathfrak{m}_{\mathbb{C}_p}^\flat, x\in A_\inf\}\subseteq A_\inf\\
          (\text{resp. }\mathfrak{q}_{\infty}&:=\{[a]x\mid a\in \mathfrak{m}_{\Kinf}^\flat, x\in \AKp\}\subseteq \AKp)
     \end{align*}
     is an ideal of $A_\inf$ (resp. $\AKp$).
\end{lemma}
\begin{proof}
We check that $\mathfrak{q}$ is an ideal. Let $a,b\in \mathfrak{m}_{\mathbb{C}_p}^\flat$ and $x,y\in A_\inf$. Take an element $c\in \mathfrak{m}_{\mathbb{C}_p}^\flat\setminus \{0\}$ such that $v(c)\leq v(a)$ and $v(c)\leq v(b)$, where $v(\cdot)$ is the valuation of $\mathbb{C}_p^\flat$. Then,
\begin{align*}
     [a]x+[b]y&=[c]\cdot [a/c]x+[c]\cdot [b/c]y\\
     &=[c]\cdot\bigl([a/c]x+[b/c]y\bigr)\\
     &\in \mathfrak{q}.
\end{align*}
For any $z\in A_\inf$, $z\cdot [a]x=[a](xz)\in \mathfrak{q}$. Thus $\mathfrak{q}$ is an ideal of $A_\inf$. One can show that $\mathfrak{q}_\infty$ is an ideal of $\AKp$ by the same argument.
\end{proof}

\begin{definition}[{\cite[D\'{e}finition 1.10.14]{FarguesFontaine2018}}]\label{def:Bbar}
We define ideals $\mathfrak{p}$ and $\mathfrak{p}_\infty$ and rings $\overline{A}, \overline{B},\overline{A}_\infty$, and $\overline{B}_\infty$ as follows:
\begin{align*}
     \mathfrak{p}&:=\mathfrak{q}[1/p],&\overline{A}&:=A_\inf/\mathfrak{q},&\overline{B}&:=\overline{A}[1/p]= B_\inf/\mathfrak{p},\\
     \mathfrak{p}_\infty&:=\mathfrak{q}_\infty[1/p],&\overline{A}_\infty&:=\AKp/\mathfrak{q}_\infty,&\overline{B}_\infty&:=\overline{A}_\infty[1/p]=\BKp/\mathfrak{p}_\infty.
\end{align*}
\end{definition}

\begin{remark}\label{rem:Bbar}
     The ideal $\mathfrak{p}$ (resp. $\mathfrak{p}_\infty$, resp. $\mathfrak{q}$, resp. $\mathfrak{q}_\infty$) is a prime ideal of $B_\inf$ (resp. $\BKp$, resp. $A_\inf$, resp. $\AKp$). Indeed, Fargues and Fontaine shows in \cite[Lemme 1.10.15]{FarguesFontaine2018} that $\mathfrak{p}$ (resp. $\mathfrak{p}_\infty$) is a prime ideal of $B_\inf$ (resp. $\BKp$) and \cref{prop:inclusion of Bbar} implies that the other ideals are also prime ideals.
     \end{remark}

\begin{lemma}\label{lem:Ainf/x is p-torsion free}
Let $R$ be a perfect ring of characteristic $p>0$. Let $x\in W(R)$ be an element such that $(x \bmod p) \in R$ is a non-zero divisor. Then $W(R)/xW(R)$ is $p$-torsion free and $p$-adically complete and separated. 
\end{lemma}
\begin{proof}
Applying the snake lemma to the commutative diagram
\[\begin{tikzcd}
0\arrow[r]&W(R)\arrow[r, "\times p"]\arrow[d, "\times x"] &W(R)\arrow[r]\arrow[d, "\times x"] &R \arrow[r]\arrow[d, "\times x",hook]&0 \\ 
0\arrow[r]&W(R)\arrow[r, "\times p"]&W(R)\arrow[r]&R \arrow[r]&0,
\end{tikzcd}\]
we see that $W(R)/xW(R)$ is $p$-torsion free. Since $W(R)$ is $p$-torsion free and $p$-adically separated and $(x \bmod p)\in R$ is a non-zero divisor, the sequence
\begin{align*}
     0 \rightarrow W(R) \xrightarrow{\times x} W(R)  \rightarrow W(R)/xW(R)  \rightarrow 0 
\end{align*}
is exact. As we have already shown that $W(R)/xW(R)$ is $p$-torsion free, this sequence remains exact after modulo $p^n$. By taking inverse limit, we get an exact sequence
\begin{align*}
     0 \rightarrow W(R)  \xrightarrow{\times x} W(R)  \rightarrow \bigl(W(R)/xW(R)\bigr)^{\wedge}_p  \rightarrow 0.
\end{align*}
Thus $W(R)/xW(R)$ is $p$-adically complete and separated.
\end{proof}

\begin{proposition}\label{prop:property of Bbar}
The ring $\overline{A}$ (resp. $\overline{A}_\infty$) is a $p$-torsion free local ring with maximal ideal $(p)$ which has a bijective Frobenius endomorphism induced by the Frobenius endomorphism of $A_\inf$ (resp. $\AKp$).
\end{proposition}
\begin{proof}
     We prove the assertion for $\overline{A}$ and the same argument can be applied to $\overline{A}_\infty$. Since $A_\inf$ is a local ring with maximal ideal $(p, [\mathfrak{m}_{\mathbb{C}_p}^\flat])$ (cf. \cref{lem:Ainf is a local ring}), $\overline{A}$ is also a local ring with maximal ideal $(p)$. Since $\varphi([a]x)=[a]^p \varphi(x)\in \mathfrak{q}$ for any $a\in \mathfrak{m}_{\mathbb{C}_p}^\flat$ and $x\in A_\inf$, $\overline{A}$ has a Frobenius endomorphism induced by the Frobenius endomorphism of $A_\inf$. As $\varphi\colon A_\inf  \rightarrow A_\inf $ is bijective, the Frobenius endomorphism $\varphi\colon \overline{A} \rightarrow \overline{A}$ is surjective. For any  $a\in \mathfrak{m}_{\mathbb{C}_p}^\flat$ and $x\in A_\inf$, we have
     \begin{align*}
          \varphi^{-1}([a]x)=[a^{1/p}]\varphi^{-1}(x)\in \mathfrak{q},
     \end{align*}
     so $\varphi\colon \overline{A} \rightarrow \overline{A}$ is injective. It remains to show that $\overline{A}$ is $p$-torsion free. Consider the exact sequence $
          0 \rightarrow [a]A_\inf  \rightarrow A_\inf  \rightarrow A_\inf/[a]A_\inf  \rightarrow 0$
     for any $a\in \mathfrak{m}_{\mathbb{C}_p}^\flat$. By taking inductive limit, we get the exact sequence
     \begin{align*}
          0 \rightarrow \mathfrak{q} \rightarrow A_\inf  \rightarrow  \varinjlim_{a\in \mathfrak{m}_{\mathbb{C}_p}^\flat}A_\inf/[a]A_\inf  \rightarrow 0,
     \end{align*}
     hence we have an isomorphism
     \begin{align*}
          A_\inf/\mathfrak{q}\cong  \varinjlim_{a\in \mathfrak{m}_{\mathbb{C}_p}^\flat}A_\inf/[a]A_\inf.
     \end{align*}
     Since $A_\inf/[a]A_\inf$ is $p$-torsion free by \cref{lem:Ainf/x is p-torsion free}, $A_\inf/\mathfrak{q}$ is also $p$-torsion free by the above isomorphism.
\end{proof}

\begin{lemma}\label{lem:Abar is mu-torsion free}
The ring $\overline{A}$ (resp. $\overline{A}_\infty$) is $\mu$-torsion free.
\end{lemma}
\begin{proof}
We treat only the case of $\overline{A}$ since the case of $\overline{A}_\infty$ can be shown by a similar argument (or we can use \cref{prop:inclusion of Bbar}). By the isomorphism $\overline{A}\cong \varinjlim_{a\in \mathfrak{m}_{\mathbb{C}_p}^\flat} A_\inf/[a]A_\inf$ proved in the proof of \cref{prop:property of Bbar}, it is enough to show that $A_\inf/[a]A_\inf$ is $\mu$-torsion free for any $a\in \mathfrak{m}_{\mathbb{C}_p}^\flat\setminus \{0\}$. By applying the snake lemma to the commutative diagram
\[\begin{tikzcd}
0\arrow[r]&A_\inf\arrow[r,"\times \mu"]\arrow[d,"{\times [a]}"]&A_\inf\arrow[r]\arrow[d,"{\times [a]}"]&A_\inf/\mu\arrow[r]\arrow[d,"{\times [a]}"]&0 \\ 
0\arrow[r]&A_\inf\arrow[r,"\times \mu"]&A_\inf\arrow[r]&A_\inf/\mu\arrow[r]&0,
\end{tikzcd}\]
it is enough to show that $A_\inf/\mu$ is $[a]$-torsion free. By \cite[5.1.4. Lemme]{Fontaine1994}, we have an injective homomorphism
\begin{align*}
    A_\inf/\mu \hookrightarrow \prod_{i\geq 0} \mathcal{O}_{\mathbb{C}_p},\qquad x\mapsto (\theta \circ \varphi^i(x))_{i\geq 0}.
\end{align*}
The image of $[a]$ under this map is $((a^{(0)})^{p^i})_{i\geq 0}$. Since each component is not zero, we see that $A_\inf/\mu$ is $[a]$-torsion free, as desired.
\end{proof}

\begin{proposition}\label{prop:inclusion of Bbar}
We have the following canonical injective homomorphisms:
\[\begin{tikzcd}
\overline{A}_\infty\arrow[r, hook]\arrow[d, hook] &\overline{A}\arrow[d, hook]\\ 
\overline{B}_\infty\arrow[r, hook]&\overline{B}.
\end{tikzcd}\]
\end{proposition}
\begin{proof}
The injectivities of $\overline{A}\hookrightarrow \overline{B}$ and $\overline{A}_\infty \hookrightarrow \overline{B}_\infty$ follow from \cref{prop:property of Bbar}. Since localization is exact, it is enough to show that the canonical map $\overline{A}_\infty \rightarrow \overline{A}$ is injective. This is equivalent to showing that $\mathfrak{q}_\infty=\mathfrak{q}\cap \AKp$. The inclusion $\mathfrak{q}_\infty \subseteq \mathfrak{q}\cap \AKp$ is obvious. Let $a\in \mathfrak{m}_{\mathbb{C}_p}^\flat$ and $x\in A_\inf$ be elements such that $[a]x\in \AKp$. Take an element $b\in \mathfrak{m}_{\Kinf}^\flat\setminus \{0\}$ such that $v(b)\leq v(a)$. Then we have $[a]x=[b]\cdot [a/b]x$.
Put $[a/b]x=(y_0,y_1,y_2,\dots )$ in $A_\inf$ with $y_i\in \mathcal{O}_{\mathbb{C}_p}^\flat$. Then a simple calculation shows
\begin{align*}
     [a]x=(by_0,b^py_1,b^{p^2}y_2,\dots )
\end{align*}
in $A_\inf$. Since this element lies in $\AKp=W(\mathcal{O}_{\Kinf}^\flat)$ and $b\in \mathcal{O}_{\Kinf}^\flat\setminus \{0\}$, $y_i$ is contained in $\mathcal{O}_{\mathbb{C}_p}^\flat \cap \Kinf^\flat=\mathcal{O}_{\Kinf}^\flat$ for any $i\geq 0$. Thus $[a/b] x\in \AKp$ and $[a]x\in \mathfrak{q}_\infty$.
\end{proof}

\begin{proposition}\label{prop:Abar and Acrys}
Let $A_{\crys,\infty}$ be the $p$-adic completion of the subring $\AKp[\xi^n/n!\mid n\in \mathbb{Z}_{\geq 0}]\subseteq  \AKp[1/p]$. Then there exist canonical surjections
\begin{align*}
     A_{\crys, \infty} \twoheadrightarrow \overline{A}_\infty,\qquad A_\crys \twoheadrightarrow \overline{A}.
\end{align*}
\end{proposition}
\begin{proof}
We only treat the former case. The latter can be shown by a similar argument. Recall the generator $\xi_p=pu-[\tau]$ of $\Ker \theta$ defined in the proof of \cref{lem:Ainf is a local ring}. We have $\AKp[\xi^n/n!\mid n\in \mathbb{Z}_{\geq 0}]=\AKp[\xi_p^n/n!\mid n\in \mathbb{Z}_{\geq 0}]$ in $\AKp[1/p]$. The image of $\xi_p$ under the surjection $\AKp \twoheadrightarrow \AKp/[\tau]$ is the same as the image of $pu$. As $\AKp/[\tau]$ is $p$-torsion free by \cref{lem:Ainf/x is p-torsion free}, we can construct a canonical surjection $\AKp[\xi_p^n/n!\mid n\in \mathbb{Z}_{\geq 0}] \twoheadrightarrow  \AKp/[\tau]$. Since $\AKp/[\tau]$ is $p$-adically complete and separated by \cref{lem:Ainf/x is p-torsion free}, we obtain a surjective map $A_{\crys,\infty} \twoheadrightarrow \AKp/[\tau]$. Hence we get the surjection $A_{\crys,\infty} \twoheadrightarrow \overline{A}_\infty$, as desired.
\end{proof}


The following definition is due to H. Du (\cite[Lemma 4.11]{Du2022}) and it can be seen as a Frobenius analog of $K$-finite in the sense of \cite{Sen1980}. See also \cite[Remark 2.6.5]{BFYSB2022}.

\begin{definition}[$\varphi$-finite part]\label{def:phi-finite part}
Let $l$ be a perfect field of characteristic $p$. We put $L_0:=W(l)[1/p]$.
Let $M$ be a $p$-torsion free $W(l)$-module equipped with a semi-linear morphism $\varphi\colon M \rightarrow M[1/p]$ with respect to the Frobenius endomorphism on $W(l)$.
\begin{enumerate}
     \item For any $x\in M$, we define $V^\varphi_{L_0}(x)$ to be the $L_0$-vector subspace of $M[1/p]$ generated by elements $\varphi^n(x)$ for all $n\geq 0$. If there is no risk of confusion, we abbreviate $V^\varphi_{L_0}(x)$ as $V^\varphi(x)$.
     \item We define the \emph{$\varphi$-finite part of $M$} by
\begin{align*}
     M_{\fin,W(l)}:=\{x\in M\mid \dim_{L_0}V^\varphi_{L_0}(x) <\infty\}.
\end{align*}
     If there is no risk of confusion, we abbreviate $M_{\fin,W(l)}$ as $M_{\fin}$.
\end{enumerate}
\end{definition}

\begin{lemma}\label{lem:another condition of phi-finite part}
Let $l$ be a perfect field of characteristic $p$. We put $L_0:=W(l)[1/p]$.
Let $M$ be a $p$-torsion free $W(l)$-module equipped with a semi-linear morphism $\varphi\colon M \rightarrow M[1/p]$ with respect to the Frobenius endomorphism on $W(l)$. Let $x\in M$ be an arbitrary element. Then, $x$ belongs to $M_\fin$ if and only if $x$ is contained in a finite dimensional $L_0$-vector subspace of $M[1/p]$ stable under $\varphi$. 
\end{lemma}
\begin{proof}
Assume $x\in M_\fin$. Then, $V^\varphi(x)$ is a finite dimensional $L_0$-vector subspace of $M[1/p]$ stable under $\varphi$ and $x\in V^\varphi(x)$. Conversely, if $x\in V$ for some finite dimensional $L_0$-vector subspace $V$ of $M[1/p]$ stable under $\varphi$, then $V^\varphi(x)$ is a subspace of $V$, hence it is finite dimensional.
\end{proof}

\begin{lemma}\label{lem:phi-finite part}
Let $l$ be a perfect field of characteristic $p$.
Let $M$ be a $p$-torsion free $W(l)$-module equipped with a semi-linear morphism $\varphi\colon M \rightarrow M[1/p]$ with respect to the Frobenius endomorphism on $W(l)$. The following hold.
\begin{enumerate}
     \item $M_\fin$ is a $W(l)$-submodule of $M$. 
     \item Let $N$ be a $W(l)$-submodule of $M$ satisfying $\varphi(N)\subseteq N[1/p]$. Then, we have $N_\fin= M_\fin \cap N$. In particular, $M_\fin=(M[1/p])_\fin\cap M$.
     \item We have $(M[1/p])_\fin=M_\fin[1/p]$ and $\varphi(M_\fin)\subseteq (M[1/p])_\fin$.
\end{enumerate}
\end{lemma}
\begin{proof}
     (1) We use \cref{lem:another condition of phi-finite part}. 
     Obviously $0\in M_\fin$. Take arbitrary elements $x\in M_\fin$ and $a\in W(l)$. Since $ax\in V^\varphi(x)$, $ax\in M_\fin$.
     Let $x'\in M_\fin$. Since $x+x'\in V^\varphi(x)+V^\varphi(x')$, we see that $x+x'\in M_\fin$. Therefore, $M_\fin$ is a $W(l)$-submodule of $M$. 
     
     (2) Let $x\in N$ be an arbitrary element. Then, $\varphi^n(x)\in N[1/p]$ for every $n\in \mathbb{Z}_{\geq 0}$. Since $N[1/p]$ is an $L_0$-vector subspace of $M[1/p]$, the $L_0$-vector subspace of $N[1/p]$ generated by $\varphi^n(x)$ for all $n\in \mathbb{Z}_{\geq 0}$ is the same as that of $M[1/p]$ generated by $\varphi^n(x)$ for all $n\in \mathbb{Z}_{\geq 0}$. This implies the assertion.

     (3) By (1) and (2), we see that $(M[1/p])_\fin \subseteq M_\fin[1/p]$. The other inclusion follows from the definition. Since $\varphi(x)\in V^\varphi(x)\subseteq M[1/p]$ for any $x\in M$, we see that $\varphi(M_\fin)\subseteq (M[1/p])_\fin$.
\end{proof}

\begin{corollary}\label{cor:phi-finite part and G-action}
     Let $l$ be a perfect field of characteristic $p$.
     Let $M$ be a $p$-torsion free $W(l)$-module equipped with a semi-linear morphism $\varphi\colon M \rightarrow M[1/p]$ with respect to the Frobenius endomorphism on $W(l)$. Let $G$ be a group. Assume that $W(l)$ has an action of $G$ and that $M$ has a semi-linear action of $G$ that commutes with $\varphi$.
     \begin{enumerate}
          \item The $W(l)$-submodule $M_\fin$ is stable under the $G$-action.
          \item If $G$ acts on $W(l)$ trivially and hence the $G$-action on $M$ is $W(l)$-linear, then we have an equality $(M_\fin)^G=(M^G)_\fin$, which we abbreviate to $M_\fin^G$.
     \end{enumerate}
\end{corollary}
\begin{proof}
Since the $G$-action on $M$ is semi-linear and commutes with $\varphi$, we have $g(V^\varphi(x))= V^\varphi(g(x))$ for any $x\in M $ and $g\in G$. Thus, $G$ acts on $M_\fin$. If $G$ acts on $W(l)$ trivially, $M^G$ is a $W(l)$-submodule of $M$. We have $\varphi(M^G)\subseteq (M[1/p])^G=M^G[1/p]$ and $(M_\fin)^G=M_\fin \cap M^G$, therefore (2) follows from \cref{lem:phi-finite part} (2).
\end{proof}

\begin{lemma}\label{lem:phi finite part as W(k) and W(l)}
Let $l'/l$ be an extension of perfect fields of characteristic $p$. Let $M$ be a $p$-torsion free $W(l')$-module equipped with a semi-linear morphism $\varphi\colon M \rightarrow M[1/p]$ with respect to the Frobenius endomorphism on $W(l')$. Then, we have $M_{\fin, W(l)}\subseteq M_{\fin, W(l')}$.
\end{lemma}
\begin{proof}
     An element $x\in M$ is contained in $M_{\fin, W(l)}$ if and only if there exist $n\in \mathbb{Z}_{\geq 0}$ and $a_0,\dots ,a_n\in W(l)[1/p]$ such that $\varphi^{n+1}(x)=\sum_{i=0}^n a_i \varphi^i(x)$. The assertion follows from this characterization. We give another proof. We set $L_0:= W(l)[1/p]$ and $L_0':= W(l')[1/p]$. For any $x\in M$, we have a canonical $L_0'$-linear map $V^{\varphi}_{L_0}(x)\otimes_{L_0} L_0' \rightarrow V^{\varphi}_{L_0'}(x)$. It is surjective since $L_0'$-vector space $V^{\varphi}_{L_0'}(x)$ is generated by elements $x, \varphi(x), \varphi^2(x),\dots$. Therefore, if $V^{\varphi}_{L_0}(x)$ is a finite dimensional $L_0$-vector space, then $V^{\varphi}_{L_0'}(x)$ is a finite dimensional $L_0'$-vector space.
\end{proof}

\begin{proposition}\label{prop:phi-finite part and multiplication}
Let $l$ be a perfect field of characteristic $p$.
Let $R$ be a $p$-torsion free $W(l)$-algebra equipped with a ring homomorphism $\varphi\colon R \rightarrow R[1/p]$ that is compatible with $\varphi\colon W(l)\rightarrow W(l)$. Then, for any $x,y\in R_\fin$, $xy\in R_\fin$.
\end{proposition}
\begin{proof}
We put $L_0:=W(l)[1/p]$.
Take arbitrary elements $x,y\in R_\fin$. We set $V:=\Ima (V^{\varphi}(x)\otimes_{L_0} V^{\varphi}(y)\rightarrow R[1/p])$, where the map is an $L_0$-linear map induced by the multiplication of $R[1/p]$. Then, $V$ is a finite dimensional $L_0$-vector space stable under $\varphi$ and containing $xy$. Therefore, $xy\in R_\fin$ by \cref{lem:another condition of phi-finite part}.
\end{proof}

The following proposition is a Frobenius analog of \cref{prop:Galois fix part}.

\begin{proposition}\label{prop:phi-finite part and tensor product}
Let $l$ be a perfect field of characteristic $p$ and we set $L_0:=W(l)[1/p]$.
Let $R$ be a $p$-torsion free $W(l)$-algebra and let $M$ be a finite free $W(l)$-module of rank $r$ equipped with $L_0$-linear homomorphisms
\begin{align*}
     \varPhi_M&\colon M\otimes_{W(l),\varphi}L_0  \xrightarrow{\cong } M[1/p],\\
     \varPhi_R&\colon R\otimes_{W(l),\varphi}L_0  \rightarrow R[1/p]. 
\end{align*}
Then, $(M\otimes_{W(l)}R)_\fin =M \otimes_{W(l)} R_\fin$ in $M\otimes_{W(l)}R$.
\end{proposition}
\begin{proof}
We first show the inclusion $(M\otimes_{W(l)}R)_\fin \supseteq M \otimes_{W(l)} R_\fin$. By \cref{lem:phi-finite part} (1), it is enough to show that $m \otimes r\in (M \otimes_{W(l)}R)_\fin$ for any $m\in M$ and $r\in R_\fin$.  
Obviously, $m \otimes r\in M[1/p] \otimes_{L_0} V^\varphi(r)$ and we see that $M[1/p] \otimes_{L_0} V^\varphi(r)$ is a finite dimensional $L_0$-vector subspace of $M[1/p] \otimes_{L_0} R[1/p]$ stable under $\varphi_M \otimes \varphi_R$. Hence, $m \otimes r\in (M \otimes_{W(l)}R)_\fin$ by \cref{lem:another condition of phi-finite part}.

We prove the other inclusion. Let $e_1,\dots ,e_r\in M$ be a basis of $M$ over $W(l)$.
Take an arbitrary element $x=\sum_{i=1}^r e_i \otimes r_i\in (M\otimes_{W(l)}R)_\fin$ with $r_i\in R$ for all $1\leq i\leq r$. Our goal is to show that $r_i\in R_\fin$ for all $1\leq i\leq r$. Define $A\in GL_r(L_0)$ by 
\begin{align*}
     \varphi_M(e_1\dots e_r)=(e_1\dots e_r)A.
\end{align*}
Let $f_1:=x,f_2,\dots ,f_d\in M\otimes_{W(l)}R[1/p]$ be a basis of $V^\varphi(x)$. Define $X=(x_{ij})_{i,j}\in M_{r,d}(R[1/p])$ and $B\in M_d(L_0)$ by 
\begin{align*}
     (f_1\dots f_d)&=(e_1\otimes1\dots e_r\otimes1)X\\
     \varphi(f_1\dots f_d)&=(f_1\dots f_d)B.
\end{align*}
Let $V$ denote the $L_0$-vector subspace of $R[1/p]$ generated by $\{x_{ij}\}_{i,j}$.
Since $r_i= x_{i1}$ for all $1\leq i \leq r$, $r_i$ is contained in $V$ for all $1\leq i\leq r$. Hence, it is enough to prove that $V$ is stable under $\varphi_R$ by virtue of \cref{lem:another condition of phi-finite part}. By the equalities
\begin{align*}
    \varphi(f_1\dots f_d)&=\varphi\bigl((e_1\otimes 1\dots e_r\otimes1)X\bigr) =(e_1\otimes 1\dots e_r\otimes1)A \varphi(X)\\
    \varphi(f_1\dots f_d)&=(f_1\dots f_d)B=(e_1\otimes1\dots e_r\otimes1)XB,
\end{align*}
we see that $\varphi(X)=A^{-1}XB$. This implies that $V$ is stable under $\varphi$, as desired.
\end{proof}

\begin{definition}[$\varphi$-isocrystal]\label{def:isocrystal}
     Let $l$ be a perfect field of characteristic $p$ and we set $L_0:=W(l)[1/p]$. A {\em $\varphi$-isocrystal over $L_0$} is a finite height $\varphi$-module over $(L_0, \varphi, p)$, i.e., a finite dimensional $L_0$-vector space $V$ equipped with a semi-linear bijective morphism $\varphi\colon V\rightarrow V$ with respect to the Frobenius endomorphism on $L_0$.
\end{definition}

\begin{corollary}\label{cor:comparison isomorphism and phi-finite}
Let $D$ be a $\varphi$-isocrystal over $K_0$ and let $V$ be a finite dimensional $\mathbb{Q}_p$-vector space. Assume that we have a $B_\crys$-linear isomorphism
\begin{align*}
    D \otimes_{K_0} B_\crys \cong V \otimes_{\mathbb{Q}_p} B_\crys
\end{align*}
that is compatible with the Frobenii. Here, the Frobenius map on the left-hand side is defined as $\varphi_D \otimes \varphi$ and that on the right-hand side is defined as $\id_V \otimes \varphi$. Then, the isomorphism induces a $(B_{\crys}^+)_{\fin,W(\overline{k})}[1/t]$-linear isomorphism
\begin{align*}
    D \otimes_{K_0} (B_{\crys}^+)_{\fin, W(\overline{k})}[1/t] \cong V \otimes_{\mathbb{Q}_p} (B_{\crys}^+)_{\fin, W(\overline{k})}[1/t]
\end{align*}
\end{corollary}
\begin{proof}
Since $\varphi(t)=pt$, we see that $(B_{\crys})_\fin= (B^+_\crys)_{\fin}[1/t]$. By taking the $\varphi$-finite parts of the given isomorphism, we get the desired isomorphism by virtue of \cref{prop:phi-finite part and tensor product}.
\end{proof}

\begin{definition}\label{def:D_crysphifin}
For any $p$-adic Galois representation $V$ of $G_K$, we define
\begin{align*}
    D_{\crys,\fin}(V):=(V\otimes_{\mathbb{Q}_p} (B_{\crys}^+)_{\fin, W(\overline{k})}[1/t])^{G_K}.
\end{align*}
\end{definition}

\begin{corollary}\label{cor:D_crysphifin}
Let $V$ be a $p$-adic Galois representation of $G_K$ of dimension $d$. Then, $D_{\crys,\fin}(V)$ is a $K_0$-vector space of dimension at most $d$ and the canonical homomorphism
\begin{align}\label{eq:comparison map for D_crysphifin}
    D_{\crys,\fin}(V)\otimes_{K_0} (B_{\crys}^+)_{\fin,W(\overline{k})}[1/t] \rightarrow V\otimes_{\mathbb{Q}_p} (B_{\crys}^+)_{\fin, W(\overline{k})}[1/t]
\end{align}
is injective. Moreover, the following conditions are equivalent:
\begin{enumerate}
     \item $V$ is crystalline.
     \item The canonical homomorphism \eqref{eq:comparison map for D_crysphifin} is an isomorphism.
     \item $\dim_{K_0} D_{\crys,\fin}(V)=d$.
\end{enumerate}
\end{corollary}
\begin{proof}
Since $D_{\crys,\fin}(V)$ is a $K_0$-vector subspace of $D_\crys(V)$, $\dim_{K_0}D_{\crys,\fin}(V)\leq d$ and by the commutative diagram
\[\begin{tikzcd}
D_{\crys,\fin}(V)\otimes_{K_0} (B_{\crys}^+)_{\fin,W(\overline{k})}[1/t] \arrow[r]\arrow[d,hook]& V\otimes_{\mathbb{Q}_p} (B_{\crys}^+)_{\fin, W(\overline{k})}[1/t]\arrow[d,hook] \\ 
D_{\crys}(V)\otimes_{K_0} B_{\crys} \arrow[r,hook]& V\otimes_{\mathbb{Q}_p} B_{\crys},
\end{tikzcd}\]\eqref{eq:comparison map for D_crysphifin} is injective. If $V$ is crystalline, $D_{\crys,\fin}(V)= D_\crys(V)$ and \eqref{eq:comparison map for D_crysphifin} is an isomorphism by \cref{prop:Galois fix part,cor:comparison isomorphism and phi-finite}. If \eqref{eq:comparison map for D_crysphifin} is an isomorphism, then $\dim_{K_0}D_{\crys,\fin}(V)=d$. If $\dim_{K_0}D_{\crys,\fin}(V)=d$, then $V$ is crystalline since $D_{\crys,\fin}(V)$ is a $K_0$-vector subspace of $D_{\crys}(V)$.
\end{proof}

\begin{remark}\label{rem:D_crysphifin}
The $(\mathbb{Q}_p,G_K)$-ring $(B_{\crys})_{\fin,W(\overline{k})}$ is $G_K$-regular in the sense of \cite[1.4.1]{Fontaine1994_2}. This can be shown by using \cite[Lemme 5.1.3. (ii)]{Fontaine1994_2}.
\end{remark}

We remark that $(B_{\crys}^+)_{\fin,W(\overline{k})}$ is ``very small''.

\begin{lemma}\label{lem:phi finite part when residue field is algebraically closed}
Let $V$ be a $P_0:=W(\overline{k})[1/p]$-vector space equipped with a semi-linear endomorphism $\varphi\colon V \rightarrow V$ with respect to the Frobenius endomorphism on $W(\overline{k})$. For any $h\in \mathbb{Z}_{\geq 1}$ and $d\in \mathbb{Z}$, let $V^{\varphi^{h}=p^d} \cdot P_0$ denote the $P_0$-vector subspace of $V$ generated by $V^{\varphi^{h}=p^d}$. Then, $V_{\fin,W(\overline{k})}$ is a sum of $V^{\varphi^{h}=p^d} \cdot P_0$ as a $P_0$-vector space, where $h$ runs through $\mathbb{Z}_{\geq 1}$ and $d$ runs through $\mathbb{Z}$ with $\gcd(h,d)=1$.
\end{lemma}
\begin{proof}
Obviously, $V^{\varphi^{h}=p^d}\subseteq V_\fin$ for all $h\in \mathbb{Z}_{\geq 1}$ and $d\in \mathbb{Z}$. Combining it with \cref{lem:phi-finite part} (1), we see that the sum of $V^{\varphi^{h}=p^d} \cdot P_0$ is contained in $V_\fin$. We prove the other inclusion. Let $x\in V_\fin$ be an arbitrary element. By Dieudonn\'{e}-Manin's theorem (\cite{Manin1963}), we have a basis $e_1,\dots ,e_d$ of $V^\varphi(x)$ such that for every $i=1,\dots ,d$, there exists $h_i\in \mathbb{Z}_{\geq 1}$ and $d_i\in \mathbb{Z}$ with $\gcd(h_i,d_i)=1$ such that $\varphi^{h_i}(e_i)=p^{d_i}e_i$. Then, $x$ is contained in the $P_0$-vector subspace generated by $V^{\varphi^{h_i}=p^{d_i}}$ for all $i=1,\dots ,d$.
This implies the other inclusion.
\end{proof}

\begin{proposition}\label{prop:B_crys fin is very small}
The Frobenius endomorphism on $(B_{\crys}^+)_{\fin,W(\overline{k})}$ induced by that on $B_{\crys}^+$ is bijective. Moreover, this ring is a proper subring of $\widetilde{B}_\rig^+:=\bigcap_{n\geq 0} \varphi^n(B_{\crys}^+)$.
\end{proposition}
\begin{proof}
By \cref{lem:phi-finite part} (3), $(B_{\crys}^+)_{\fin}$ is stable under the Frobenius endomorphism on $B_{\crys}^+$.
Since the Frobenius endomorphism on $B_\crys^+$ is injective, the induced Frobenius endomorphism is also injective. For any $x\in (B_{\crys}^+)_\fin$, the induced Frobenius endomorphism on $V^\varphi(x)$ is injective, hence bijective. This implies $x\in \varphi(V^\varphi(x))\subseteq \varphi((B_{\crys}^+)_\fin)$. Thus, the first claim holds. As a consequence, we see that $(B_{\crys}^+)_\fin \subseteq \widetilde{B}_\rig^+$. It remains to show that they are not equal. On the one hand, $A_\inf \subseteq \widetilde{B}_\rig^+$. On the other hand, we have $A_{\inf, \fin}= W(\overline{k})$, which can be shown by using \cref{lem:phi finite part when residue field is algebraically closed} and the facts that $A_\inf^{\varphi^h=p^d}= 0$ for any $h\in \mathbb{Z}_{\geq 1}$ and $0\neq d \in \mathbb{Z}$ and that $A_\inf^{\varphi^h=1}= W(\mathbb{F}_{p^h})$ for any $h\in \mathbb{Z}_{\geq 1}$. Therefore, they are not equal by virtue of \cref{lem:phi-finite part} (2).
\end{proof}

The following lemma is \cite[Proposition 3.5]{Du2022}, which is a consequences of \cite[Th\'{e}or\`{e}me 11.1.7, Corollaire 11.1.14]{FarguesFontaine2018}.

\begin{lemma}[{\cite[Proposition 3.5]{Du2022}}]\label{lem:phi^h=p^d}
Let $(M, \Phi)$ be a finite height $\varphi$-module over $(B_\crys^+, \varphi,p)$ whose underlying $B_\crys^+$-module is finite free. Then, for any $h\in \mathbb{Z}_{\geq 1}$ and $d\in \mathbb{Z}$, we have a canonical isomorphism
\begin{align*}
    M^{\varphi^{h}=p^d}\xrightarrow{\cong} (M\otimes_{B_\crys^+} \overline{B})^{\varphi^{h}=p^d}.
\end{align*}
\end{lemma}
\begin{proof}
For the sake of the reader, we recall the proof. Let $B\in \{B_\crys^+, \overline{B}\}$.
For any $h\in \mathbb{Z}_{\geq 1}$ and $d\in \mathbb{Z}$, we define $B\langle d,h\rangle$ to be the $\varphi$-module over $(B, \varphi,p)$ that has a basis $e_1,\dots ,e_h$ as a $B$-module such that $\varphi(e_i)=e_{i+1}$ for $1\leq i\leq h-1$ and $\varphi(e_h)=p^d e_1$. It is finite height since $p$ is invertible in $B$ and $\varphi\colon B \rightarrow B$ is injective. By \cref{prop:scalar extension of phi-modules}, $M \otimes_{B_\crys^+} \overline{B}$ is a finite height $\varphi$-module over $(\overline{B}, \varphi, p)$. Then, we have canonical isomorphisms
\begin{alignat*}{2}
    M^{\varphi^{h}=p^d} \xrightarrow{\cong}  &\Hom_{B_\crys^+,\varphi}( B_\crys^+\langle d,h\rangle, M),\quad &&x\mapsto (f\colon B_\crys^+\langle d,h\rangle \rightarrow M \text{ s.t. }f(e_1)= x),\\
     (M\otimes \overline{B})^{\varphi^{h}=p^d} \xrightarrow{\cong} &\Hom_{\overline{B},\varphi}( \overline{B}\langle d,h\rangle, M\otimes \overline{B}),\quad &&y\mapsto (g\colon \overline{B}\langle d,h\rangle \rightarrow M\otimes \overline{B} \text{ s.t. }g(e_1)= y).
\end{alignat*}
Then, we have a commutative diagram
\[\begin{tikzcd}
M^{\varphi^h=p^d}\arrow[r,"\cong"] \arrow[d]& \Hom_{B_\crys^+,\varphi}( B_\crys^+\langle d,h\rangle, M) \arrow[d] \\
 (M\otimes_{B_\crys^+} \overline{B})^{\varphi^{h}=p^d} \arrow[r,"\cong"]& \Hom_{\overline{B},\varphi}( \overline{B}\langle d,h\rangle, M\otimes_{B_\crys^+} \overline{B}).
\end{tikzcd}\]
The right vertical arrow is an isomorphism by \cite[Th\'{e}or\`{e}me 11.1.7, Corollaire 11.1.14]{FarguesFontaine2018}, hence the left vertical arrow is also an isomorphism, as desired.
\end{proof}

The following theorem is one of the motivations of the $\varphi$-finite part and $\overline{B}$.

\begin{theorem}\label{thm:phi-finite part of B_crys and Bbar}
Let $(M, \Phi)$ be a finite height $\varphi$-module over ($B_\crys^+, \varphi,p)$ whose underlying $B_\crys^+$-module is finite free. Then, we have a canonical isomorphism
\begin{align*}
    M_{\fin,W(\overline{k})}\xrightarrow{\cong} (M\otimes_{B_\crys^+} \overline{B})_{\fin,W(\overline{k})}. 
\end{align*}
In particular, we have a canonical isomorphism $(B_{\crys}^+)_{\fin,W(\overline{k})}\xrightarrow{\cong} \overline{B}_{\fin,W(\overline{k})}. $
\end{theorem}
\begin{proof}
This is a consequence of \cref{lem:phi^h=p^d,lem:phi finite part when residue field is algebraically closed}.
\end{proof}

The following proposition is due to Du and it is also one of the motivations of the $\varphi$-finite part.

\begin{proposition}[{\cite[Lemma 4.11]{Du2022}}]\label{prop:G-fixed part of Bbar}
The canonical map $W(k) \rightarrow \overline{A}_\infty  \hookrightarrow \overline{A}$ induces isomorphisms
\begin{align*}
     W(k) \cong ((\overline{A}_{\infty})_{\fin,W(k_\infty)})^{\Gamma_K} \cong (\overline{A}_{\fin,W(\overline{k})})^{G_K},\\
     K_0 \cong ((\overline{B}_{\infty})_{\fin,W(k_\infty)})^{\Gamma_K}\cong (\overline{B}_{\fin,W(\overline{k})})^{G_K}.
\end{align*}
\end{proposition}
\begin{proof}
     The isomorphism $K_0 \cong (\overline{B}_{\fin,W(\overline{k})})^{G_K}$ is proved in \cite[Lemma 4.11]{Du2022}. For the convenience of the reader, we recall the proof. 
     By \cref{thm:phi-finite part of B_crys and Bbar}, we have a canonical isomorphism $(B_\crys^+)_{\fin, W(\overline{k})}\xrightarrow{\cong} \overline{B}_{\fin,W(\overline{k})}$. By taking the $G_K$-fixed part, we obtain an isomorphism $K_0 \xrightarrow{\cong} (\overline{B}_{\fin,W(\overline{k})})^{G_K}$ because $(B_\crys^+)^{G_K}=K_0$.

     We deduce the other isomorphisms from $K_0\cong (\overline{B}_{\fin,W(\overline{k})})^{G_K}$. Since we have already shown in \cref{prop:inclusion of Bbar} that the canonical map $\overline{B}_\infty \rightarrow \overline{B}$ is injective, \cref{lem:phi finite part as W(k) and W(l)} implies that $K_0  \rightarrow ((\overline{B}_\infty)_{\fin,W(k_\infty)})^{\Gamma_K}$ is also an isomorphism. 
     \[\begin{tikzcd}
     K_0\arrow[r]\arrow[dr, "\cong "', sloped] &((\overline{B}_\infty)_{\fin,W(k_\infty)})^{\Gamma_K}\arrow[d, hook] \\ 
     {}&(\overline{B}_{\fin,W(\overline{k})})^{G_K}
     \end{tikzcd}\]
   Next we show that the canonical map $W(k) \rightarrow  (\overline{A}_{\fin,W(\overline{k})})^{G_K}$ is an isomorphism. Since $\overline{A}$ is $p$-torsion free by \cref{prop:property of Bbar}, $W(k) \rightarrow \overline{A}$ is injective. We may consider $W(k)$, $K_0$, and $\overline{A}$ as subrings of $\overline{B}$. Since we have already shown that $K_0 = (\overline{B}_{\fin,W(\overline{k})})^{G_K}$, it is enough to show the equality $W(k)= \overline{A}\cap K_0$ by \cref{lem:phi-finite part} (2). Indeed, it implies
   \begin{align*}
       W(k)=\overline{A}\cap (\overline{B}_{\fin, W(\overline{k})})^{G_K}=\overline{B}^{G_K}\cap \overline{A}\cap \overline{B}_{\fin,W(\overline{k})}=\overline{B}^{G_K}\cap \overline{A}_{\fin,W(\overline{k})}=(\overline{A}_{\fin,W(\overline{k})})^{G_K}.
   \end{align*}
      Take an arbitrary element $x\in \overline{A}\cap K_0$.
     Since the composite $K_0 \hookrightarrow \overline{B}\twoheadrightarrow P_0$ coincides with the natural inclusion $K_0 \hookrightarrow P_0$ and $x\in \overline{A}$, the image of $x\in \overline{B}$ by $\overline{B} \twoheadrightarrow  P_0$ is contained in $W(\overline{k})\cap K_0= W(k)$. Thus $x\in W(k)$ as desired. The bijectivity of the map $W(k) \rightarrow ((\overline{A}_\infty)_{\fin,W(k_\infty)})^{\Gamma_K}$ can be shown by the same argument or by using the fact that  $W(k)\rightarrow ((\overline{A}_\infty)_{\fin,W(k_\infty)})^{\Gamma_K}  \hookrightarrow (\overline{A}_{\fin,W(\overline{k})})^{G_K}$ is an isomorphism.
\end{proof}

\begin{remark}\label{rem:G-fixed part of Bbar}
     By virtue of \cref{lem:phi finite part as W(k) and W(l)}, we also have isomorphisms
     \begin{align*}
     W(k) \cong (\overline{A}_{\infty})_{\fin,W(k)}^{\Gamma_K} \cong \overline{A}_{\fin,W(k)}^{G_K},\\
     K_0 \cong (\overline{B}_{\infty})_{\fin,W(k)}^{\Gamma_K}\cong \overline{B}_{\fin,W(k)}^{G_K}.
\end{align*}
\end{remark}

\begin{definition}[$\varphi$-isocrystal part]\label{def:isocrystal part}
     Let $l$ be a perfect field of characteristic $p$ and we set $L_0:=W(l)[1/p]$. Let $M$ be a $p$-torsion free $W(l)$-module equipped with a semi-linear morphism $\varphi\colon M \rightarrow M[1/p]$ with respect to the Frobenius endomorphism on $W(l)$. We define the \emph{$\varphi$-isocrystal part of $M$} by
\begin{align*}
    M_{\icrys,W(l)}:=\{x\in M\mid V^{\varphi}_{L_0}(x)\text{ is a $\varphi$-isocrystal over $L_0$} \}.
\end{align*}
If there is no risk of confusion, we abbreviate $M_{\icrys,W(l)}$ as $M_\icrys$.
\end{definition}


\begin{lemma}\label{lem:another condition of phi-isocrystal part}
Let $l$ be a perfect field of characteristic $p$. We put $L_0:=W(l)[1/p]$.
Let $M$ be a $p$-torsion free $W(l)$-module equipped with a semi-linear morphism $\varphi\colon M \rightarrow M[1/p]$ with respect to the Frobenius endomorphism on $W(l)$. Let $x\in M$ be an arbitrary element. Then, $x$ belongs to $M_\icrys$ if and only if $x$ is contained in a finite dimensional $L_0$-vector subspace $V$ of $M[1/p]$ such that it is stable under $\varphi$ and the induced map $\varphi_V\colon V \rightarrow V$ is either injective or surjective. 
\end{lemma}
\begin{proof}
Assume $x\in M_\icrys$. Then, $x\in V^\varphi(x)$ and $V^\varphi(x)$ is a finite dimensional $L_0$-vector subspace of $M[1/p]$ stable under $\varphi$ and the induced map is bijective.
Conversely, assume that $x\in V$ for some finite dimensional $L_0$-vector subspace such that it is stable under $\varphi$ and $\varphi_V\colon V \rightarrow V$ is either injective or surjective. Then, $\varphi_V$ induces an $L_0$-linear map $\Phi_V\colon V \otimes_{L_0,\varphi}L_0 \rightarrow V$. Since the canonical map $V \rightarrow V \otimes_{L_0,\varphi}L_0,\ v\mapsto v \otimes1$ is bijective, $\varphi_V$ is injective (resp. surjective) if and only if $\Phi_V$ is injective (resp. surjective). Since $\Phi_V$ is an $L_0$-linear map between finite dimensional $L_0$-vector spaces of equal dimension, injectivity (resp. surjectivity) of $\Phi_V$ implies its bijectivity. Therefore, we see that $\varphi_V$ is bijective. 
Since $V$ is stable under $\varphi$, we see that $V^\varphi(x)$ is an $L_0$-vector subspace of $V$, hence it is finite dimensional. Moreover, we see that $\varphi\colon V^\varphi(x) \rightarrow V^\varphi(x)$ is injective. By the same argument as above, this implies that $\varphi\colon V^\varphi(x) \rightarrow V^\varphi(x)$ is bijective. Therefore, $x\in M_\icrys$.
\end{proof}

\begin{lemma}\label{lem:isocrystal part}
Let $l$ be a perfect field of characteristic $p$.
Let $M$ be a $p$-torsion free $W(l)$-module equipped with a semi-linear morphism $\varphi\colon M \rightarrow M[1/p]$ with respect to the Frobenius endomorphism on $W(l)$. The following hold.
\begin{enumerate}
     \item $M_\icrys$ is a $W(l)$-submodule of $M$. 
     \item Let $N$ be a $W(l)$-submodule of $M$ satisfying $\varphi(N)\subseteq N[1/p]$. Then, we have $N_\icrys= M_\icrys \cap N$. In particular, $M_\icrys=(M[1/p])_\icrys\cap M$.
     \item We have $(M[1/p])_\icrys=M_\icrys[1/p]$ and $\varphi(M_\icrys)\subseteq (M[1/p])_\icrys$. Moreover, $\varphi\colon M_\icrys[1/p] \rightarrow M_\icrys[1/p]$ is bijective.
\end{enumerate}
\end{lemma}
\begin{proof}
(1) We use \cref{lem:another condition of phi-isocrystal part}. Obviously, $0\in M_{\icrys}$. For any $a\in W(l)$ and $x\in M_\icrys$, $ax \in V^\varphi(x)$, hence $ax\in M_{\icrys}$. We also have $x+x'\in V^{\varphi}(x)+V^{\varphi}(x')$ for any $x,x'\in M_\icrys$. Since $V^\varphi(x)+ V^\varphi(x')$ is stable under $\varphi$ and the induced $\varphi\colon V^\varphi(x)+ V^\varphi(x') \rightarrow V^\varphi(x)+ V^\varphi(x')$ is surjective, we see that $x+x'\in M_{\icrys}$. Therefore, $M_\icrys$ is a $W(l)$-submodule of $M$.

(2) follows from the same argument as in the proof of \cref{lem:phi-finite part} (2).

(3) The first assertion follows from the same argument as in the proof of \cref{lem:phi-finite part} (3). 
Since $\varphi(x)\in V^\varphi(x)$, we see that $\varphi(x)\in (M[1/p])_\icrys$ for any $x\in M_\icrys$. Then, we obtain a homomorphism $\varphi\colon M_\icrys[1/p] \rightarrow M_\icrys[1/p]$.
Since $\varphi\colon V^\varphi(x)\rightarrow V^{\varphi}(x)$ is an isomorphism for any $x\in (M[1/p])_\icrys$, $\varphi\colon M_\icrys[1/p] \rightarrow M_\icrys[1/p]$ is injective. By the surjectivity of $\varphi\colon V^{\varphi}(x)\rightarrow V^\varphi(x)$, there exists $\varphi^{-1}(x)\in V^{\varphi}(x)$. Then, $\varphi^{-1}(x)\in M_\icrys[1/p]$. Therefore, $\varphi\colon M_\icrys[1/p] \rightarrow M_\icrys[1/p]$ is surjective, thereby bijective.
\end{proof}

\begin{corollary}\label{cor:isocrystal part and G-action}
      Let $l$ be a perfect field of characteristic $p$.
     Let $M$ be a $p$-torsion free $W(l)$-module equipped with a semi-linear morphism $\varphi\colon M \rightarrow M[1/p]$ with respect to the Frobenius endomorphism on $W(l)$. Let $G$ be a group. Assume that $W(l)$ has an action of $G$ and that $M$ has a semi-linear action of $G$ that commutes with $\varphi$.
     \begin{enumerate}
          \item The $W(l)$-submodule $M_\icrys$ is stable under the $G$-action.
          \item If $G$ acts on $W(l)$ trivially and hence the $G$-action on $M$ is $W(l)$-linear, then we have an equality $(M_\icrys)^G=(M^G)_\icrys$, which we abbreviate to $M_\icrys^G$.
     \end{enumerate}
\end{corollary}
\begin{proof}
This can be shown by the same argument as in the proof of \cref{cor:phi-finite part and G-action}.
\end{proof}

\begin{lemma}\label{lem:isocrys part as W(k) and W(l)}
Let $l'/l$ be an extension of perfect fields of characteristic $p$. Let $M$ be a $p$-torsion free $W(l')$-module equipped with a semi-linear morphism $\varphi\colon M \rightarrow M[1/p]$ with respect to the Frobenius endomorphism on $W(l')$. Then, we have $M_{\icrys, W(l)}\subseteq M_{\icrys, W(l')}$.
\end{lemma}
\begin{proof}
     We set $L_0':=W(l')[1/p]$ and $L_0:=W(l)[1/p]$. For any $x\in M$, consider the following commutative diagram:
     \[\begin{tikzcd}
     V^\varphi_{L_0}(x)\otimes_{L_0} L_0' \arrow[r,->>]\arrow[d,"\varphi \otimes \varphi"]&V^{\varphi}_{L_0'}(x) \arrow[d,"\varphi"]\\ 
     V^\varphi_{L_0}(x)\otimes_{L_0} L_0' \arrow[r,->>]&V^{\varphi}_{L_0'}(x).
     \end{tikzcd}\]
     If $x\in M_{\icrys, W(l)}$, $V^{\varphi}_{L_0'}(x)$ is a finite dimensional $L_0'$-vector space by \cref{lem:phi finite part as W(k) and W(l)} and the left vertical arrow is bijective. Hence, the right vertical arrow is surjective. By \cref{lem:another condition of phi-isocrystal part}, this implies $x\in M_{\icrys,W(l')}$.
\end{proof}

\begin{proposition}\label{prop:isocrystal part and multiplication}
Let $l$ be a perfect field of characteristic $p$.
Let $R$ be a $p$-torsion free $W(l)$-algebra equipped with a ring homomorphism $\varphi\colon R \rightarrow R[1/p]$ that is compatible with $\varphi\colon W(l)\rightarrow W(l)$. Then, for any $x,y\in R_\icrys$, $xy\in R_\icrys$.
\end{proposition}
\begin{proof}
We set $L_0:= W(l)[1/p]$.
Take arbitrary elements $x,y\in R_\icrys$. We set $V:=\Ima (V^{\varphi}(x)\otimes_{L_0} V^{\varphi}(y)\rightarrow R[1/p])$, where the map is an $L_0$-linear map induced by the multiplication of $R[1/p]$. As we have seen in \cref{prop:phi-finite part and multiplication}, $V$ is a finite dimensional $L_0$-vector space stable under $\varphi$ and containing $xy$. We also have a commutative diagram
\[\begin{tikzcd}
V^{\varphi}(x)\otimes_{L_0} V^{\varphi}(y) \arrow[r,->>] \arrow[d, "\varphi \otimes \varphi","\cong"'] & V \arrow[d, "\varphi"] \\
V^{\varphi}(x)\otimes_{L_0} V^{\varphi}(y) \arrow[r,->>] & V.
\end{tikzcd}\]
Therefore, $\varphi\colon V \rightarrow V$ is surjective. By \cref{lem:another condition of phi-isocrystal part}, $xy\in R_\icrys$ as desired.
\end{proof}

\begin{proposition}\label{prop:isocrystal part and tensor product}
Let $l$ be a perfect field of characteristic $p$ and we set $L_0:=W(l)[1/p]$.
Let $R$ be a $p$-torsion free $W(l)$-algebra and let $M$ be a finite free $W(l)$-module of rank $r$ equipped with $L_0$-linear homomorphisms
\begin{align*}
     \varPhi_M&\colon M\otimes_{W(l),\varphi}L_0  \xrightarrow{\cong } M[1/p],\\
     \varPhi_R&\colon R\otimes_{W(l),\varphi}L_0  \rightarrow R[1/p]. 
\end{align*}
Then $(M\otimes_{W(l)}R)_\icrys =M \otimes_{W(l)} R_\icrys$ in $M\otimes_{W(l)}R$.
\end{proposition}
\begin{proof}
This can be shown by an argument similar to that in the proof of \cref{prop:phi-finite part and tensor product}.
Indeed, for any $m\in M$ and $r\in R_\icrys$, $M[1/p] \otimes_{L_0} V^{\varphi}(r)$ is a finite dimensional $L_0$-vector subspace of $M[1/p] \otimes_{L_0} R[1/p]$ stable under $\varphi_M \otimes \varphi_R$ and the induced Frobenius map is bijective. By \cref{lem:another condition of phi-isocrystal part}, this implies $m \otimes r\in (M \otimes_{W(l)} R)_\icrys$. Combining this with \cref{lem:isocrystal part} (1), we obtain the inclusion $M \otimes_{W(l)} R_\icrys \subseteq (M \otimes_{W(l)} R)_\icrys$.
As for the other inclusion, under the notation in the proof of \cref{prop:phi-finite part and tensor product}, if $x\in (M \otimes_{W(l)}R)_{\icrys}$, then $B\in GL_d (L_0)$. In this case, we have $X=A \varphi(X)B^{-1}$. Hence, $\varphi$ on $V$ is surjective, which implies that $r_i\in R_{\icrys}$ by \cref{lem:another condition of phi-isocrystal part}.
\end{proof}

\begin{proposition}\label{prop:A_K^+/a and Abar}
For any $a\in \mathfrak{m}_{\Kinf}^{\flat}$ (resp. $a\in \mathfrak{m}_{\mathbb{C}_p}^\flat$), the canonical map $\AKp/[a] \twoheadrightarrow  \overline{A}_\infty$ (resp. $A_\inf/[a]\twoheadrightarrow \overline{A}$) induces injective homomorphisms
\begin{alignat*}{2}
     (\AKp/[a])_{\icrys,W(k_\infty)} &\hookrightarrow (\overline{A}_\infty)_{\fin,W(k_\infty)} &\qquad(\text{resp. }(A_\inf/[a])_{\icrys,W(\overline{k})} &\hookrightarrow \overline{A}_{\fin, W(\overline{k})}), \\
     (\BKp/[a])_{\icrys,W(k_\infty)} &\hookrightarrow (\overline{B}_\infty)_{\fin,W(k_\infty)} &\qquad(\text{resp. }(B_\inf/[a])_{\icrys,W(\overline{k})} &\hookrightarrow \overline{B}_{\fin, W(\overline{k})}).
\end{alignat*}
\end{proposition}
\begin{proof}
We only prove the injectivities of left homomorphisms since the right ones can be shown by the same argument. By \cref{lem:phi-finite part} (3) and \cref{lem:isocrystal part} (3), it is enough to show that $(\AKp/[a])_\icrys\rightarrow (\overline{A}_\infty)_\fin$ is injectivie. Let $x\in \AKp$ be an arbitrary element whose image $\overline{x}\in\AKp/[a]$ is contained in the kernel of the homomorphism. Then, $x\in \mathfrak{q}_\infty$, hence we may write $x=[b]y$ for some $b\in \mathfrak{m}_{\Kinf}^{\flat}$ and $y\in \AKp$. Since $\varphi([b])=[b^p]$, there exists $n\in \mathbb{Z}_{\geq 1}$ such that $\varphi^n(\overline{x})=0$. This implies $\overline{x}=0$ since $\varphi\colon V^\varphi(\overline{x})\rightarrow V^{\varphi}(\overline{x})$ is bijective and $\AKp/[a]\hookrightarrow \BKp/[a]$ by \cref{lem:Ainf/x is p-torsion free}.
\end{proof}

\begin{corollary}\label{cor:G-fixed part of A_K^+/a}
For any $a\in \mathfrak{m}_{\Kinf}^{\flat}$ (resp. $a\in \mathfrak{m}_{\mathbb{C}_p}^\flat$), the canonical map $W(k) \rightarrow \AKp/[a]$ (resp. $W(k)\rightarrow A_\inf/[a]$) induces isomorphisms
\begin{alignat*}{2}
     W(k) &\cong ((\AKp/[a])_{\icrys,W(k_\infty)})^{\Gamma_K} &\qquad(\text{resp. }W(k) &\cong ((A_\inf/[a])_{\icrys, W(\overline{k})})^{G_K}), \\
     K_0 &\cong ((\BKp/[a])_{\icrys,W(k_\infty)})^{\Gamma_K} &\qquad(\text{resp. }K_0 &\cong ((B_\inf/[a])_{\icrys, W(\overline{k})})^{G_K}).
\end{alignat*}
\end{corollary}
\begin{proof}
We only prove the first isomorphism since the others can be shown by the same argument. We have a canonical homomorphism
\begin{align*}
    W(k)\rightarrow ((\AKp/[a])_{\icrys,W(k_\infty)})^{\Gamma_K} \rightarrow  ((\overline{A}_\infty)_{\fin,W(k_\infty)})^{\Gamma_K}.
\end{align*}
By \cref{prop:G-fixed part of Bbar}, the composite is an isomorphism, thus the second homomorphism is surjective. By \cref{prop:A_K^+/a and Abar}, it is also injective, hence an isomorphism. Therefore, the first homomorphism is an isomorphism.
\end{proof}

\begin{remark}\label{rem:G-fixed part of A_K^+/a}
     By virtue of \cref{lem:isocrys part as W(k) and W(l)}, we also have isomorphisms
     \begin{alignat*}{2}
     W(k) &\cong (\AKp/[a])_{\icrys,W(k)}^{\Gamma_K} &\qquad W(k) &\cong (A_\inf/[a])_{\icrys, W(k)}^{G_K}, \\
     K_0 &\cong (\BKp/[a])_{\icrys,W(k)}^{\Gamma_K} &\qquad K_0 &\cong (B_\inf/[a])_{\icrys, W(k)}^{G_K}.
\end{alignat*}

\end{remark}

\subsection{Construction of the functor}
In the rest of this paper, we fix an element $\a\in \mathfrak{m}_{\Kinf}^{\flat}$ such that $v(\a)=p v_p(\pi)$. For the existence of such an element, see \cite[Proposition 3.1.9]{FarguesFontaine2018}.
Let $N$ be a $(\varphi,\Gamma)$-module over $\AKp$. 
Recall that $\xi_p:=pu-[\tau]\in \AKp$ defined in the proof of \cref{lem:Ainf is a local ring} is a generator of $\Ker  \theta$. We see that $\xi\in \xi_p\cdot(\AKp)^\times$ and $\varphi(\xi_p)\equiv p\varphi(u) \bmod [\a]\AKp$. Since $\theta(u)$ is a unit, \cref{lem:Ainf is a local ring} implies that $u$ is invertible in $\AKp$.
By tensoring $\varPhi_N\colon N\otimes_{\AKp,\varphi}\AKp[1/\varphi(\xi)] \xrightarrow{\cong }N[1/\varphi(\xi)] $ with $\AKp/[\a]$, we obtain a $\BKp/[\a]$-linear isomorphism
\begin{align*}
     \varPhi_{N/[\a]}\colon N/[\a] \otimes_{\AKp/[\a],\varphi}\BKp/[\a]  \xrightarrow{\cong } N/[\a][1/p]. 
\end{align*}
Thus the $W(k)$-module $N/[\a]$ has a semi-linear map $\varphi_{N/[\a]}\colon N/[\a] \rightarrow N/[\a][1/p]$ and we can think of the $\varphi$-finite part and the $\varphi$-isocrystal part of $N/[\a]$. Similarly, we have a $\overline{B}_\infty$-linear isomorphism 
\begin{align*}
    \varPhi_{N/\mathfrak{q}_\infty}\colon N/\mathfrak{q}_\infty \otimes_{\overline{A}_\infty,\varphi}\overline{B}_\infty  \xrightarrow{\cong } N/\mathfrak{q}_\infty[1/p] 
\end{align*} 
and we can think of the $\varphi$-finite part of $N/\mathfrak{q}_\infty$ as well.

\begin{definition}\label{def:crystalline (phi Gamma)-modules}
     A $(\varphi,\Gamma)$-module $N$ over $\AKp$ is said to be {\em crystalline} if it satisfies
     \begin{enumerate}
          \item The $\Gamma_K$-action on $N$ is trivial modulo $\mu$.
     \end{enumerate}
     and either of the following conditions:
     \begin{enumerate}
          \item[(2-A)] The $W(k)$-module $(N/[\a])^{\Gamma_K}_{\icrys,W(k)}$ is finite free  and the canonical map
               \begin{align}\label{eq:condition (2-A)}
                    (N/[\a])^{\Gamma_K}_{\icrys,W(k)} \otimes_{W(k)} \AKp/[\a]  \rightarrow N/[\a]
               \end{align}
               is an isomorphism.
          \item[(2-a)] The $W(k)$-module $(N/[\a])^{\Gamma_K}_{\icrys,W(k)}$ generates $N/[\a]$ as an $\AKp/[\a]$-module.
          \item[(2-B)] The $(\AKp/[\a])_{\fin,W(k)}^{\Gamma_K}$-module $(N/[\a])^{\Gamma_K}_{\fin,W(k)}$ is finite free and the canonical map
               \begin{align}\label{eq:condition (2-B)}
                    (N/[\a])^{\Gamma_K}_{\fin,W(k)} \otimes_{(\AKp/[\a])_{\fin,W(k)}^{\Gamma_K}} \AKp/[\a]  \rightarrow N/[\a]
               \end{align}
               is an isomorphism. 
          \item[(2-b)] The $(\AKp/[\a])_{\fin,W(k)}^{\Gamma_K}$-module $(N/[\a])^{\Gamma_K}_{\fin,W(k)}$ generates $N/[\a]$ as an $\AKp/[\a]$-module.
          \item[(2-C)] The $W(k)$-module $(N/\mathfrak{q}_{\infty})^{\Gamma_K}_{\fin,W(k)}$ is finite free and the canonical map
               \begin{align}\label{eq:condition (2-C)}
                    (N/\mathfrak{q}_\infty)^{\Gamma_K}_{\fin,W(k)} \otimes_{W(k)} \overline{A}_{\infty}  \rightarrow N/\mathfrak{q}_\infty
               \end{align}
               is an isomorphism.
          \item[(2-c)] The $W(k)$-module $(N/\mathfrak{q}_{\infty})^{\Gamma_K}_{\fin,W(k)}$ generates $N/\mathfrak{q}_\infty$ as an $\overline{A}_{\infty}$-module.
     \end{enumerate}
     The category of crystalline $(\varphi,\Gamma)$-modules over $\AKp$ is denoted by $\Mod_{\varphi,\Gamma}^{\fh, \crys}(\AKp)$.
\end{definition}

\begin{remark}\label{rem:crystalline (phi Gamma)-modules}
     We can consider similar conditions by using the $\varphi$-finite part and the $\varphi$-isocrystal part as $W(k_\infty)$-modules instead of $W(k)$-modules, but we omit this here.
\end{remark}

Our main theorem is the following:

\begin{theorem}[{\cref{thm:Main theorem for crystalline}}]
     There exists an equivalence of categories between the category of crystalline $(\varphi,\Gamma)$-modules over $\AKp$ and the category of free crystalline $\mathbb{Z}_p$-representations of $G_K$ via the functor
     \begin{align*}
          T\colon \Mod_{\varphi,\Gamma}^{\fh, \crys}(\AKp)  \xrightarrow{\sim} \Rep_{\mathbb{Z}_p}^{\crys}(G_K),\qquad N \mapsto T(N):=\bigl(N\otimes_{\AKp} W(\mathbb{C}_p^\flat)\bigr)^{\varphi=1}.
     \end{align*}

\end{theorem}

We have already shown that $T(N)$ is a free $\mathbb{Z}_p$-representation of $G_K$ in \cref{cor:well-definedness as p-adic rep}. The goal of this subsection is to show that $V(N):=T(N)[1/p]$ is a crystalline representation. 

\begin{lemma}[{\cite[Lemma 4.26]{BhattMorrowScholze}}, {\cite[Proposition 6.15]{MorrowTsuji2021}}]\label{lem:isomorphism between N and T after inverting mu}
Let $N$ be a $(\varphi,\Gamma)$-module over $\AKp$. Then, we have an $A_\inf[1/\mu]$-linear isomorphism
\begin{align}\label{eq:isomorphism between N and T after inverting mu}
    N\otimes_{\AKp} A_\inf[1/\mu] \cong T(N) \otimes_{\mathbb{Z}_p} A_\inf[1/\mu]
\end{align}
that is compatible with the $G_K$-actions and the Frobenii. Moreover, it is functorial in $N$.
\end{lemma}
\begin{proof}
This is \cite[Lemma 4.26]{BhattMorrowScholze} or \cite[Proposition 6.15]{MorrowTsuji2021}. We only remark that the isomorphism is induced by the isomorphism \eqref{eq:isomorphism of Fontaine} due to J.-M. Fontaine (\cite[Proof of Proposition 1.2.6]{Fontaine1990}). Therefore, it is compatible with the $G_K$-actions and the Frobenius maps and functorial in $N$.
\end{proof}

\begin{lemma}\label{lem:phi-module over A_inf to Bcrys}
Let $N$ be a $\varphi$-module over $(\AKp, \varphi, \varphi(\xi))$. Then, $N\otimes_{\AKp} B_\crys^+$ equipped with the Frobenius map $\varphi_N \otimes \varphi$ is a $\varphi$-module over $(B_\crys^+, \varphi, p)$. If $N$ is finite height, then $N\otimes_{\AKp} B_\crys^+$ is finite height.
\end{lemma}
\begin{proof}
By \cref{prop:scalar extension of phi-modules}, it is enough to show an equality of ideals $\varphi(\xi)A_\crys = p A_\crys$. 
\begin{align*}
    \varphi(\xi)=\frac{\varphi(\mu)}{\mu}=\frac{(\mu+1)^p-1}{\mu}=\sum_{i=1}^{p}\binom{p}{i}\mu^{i-1}=p\left(1+\sum_{i=2}^p\frac{1}{p}\binom{p}{i}\mu^{i-1}\right).
\end{align*}
As $\mu^{p-1}\in p A_\crys$, the above calculation implies that $\varphi(\xi)/p\in 1+\Fil^1 A_\crys$. Since $x^p\in pA_\crys$ for all $x\in \Fil^1 A_\crys$, every element in $\Fil^1 A_\crys$ is topologically nilpotent. Because $A_\crys$ is $p$-adically complete and separated, this implies that $1+ \Fil^1 A_\crys \subseteq (A_\crys)^\times$. Therefore, $\varphi(\xi)A_\crys = p A_\crys$.
\end{proof}
\begin{remark}\label{rem:phi-module over A_inf to Brig}
     We define $\widetilde{B}_\rig^+:= \bigcap_{n\geq 0} \varphi^n(B_\crys^+)$ as usual. We note that $\varphi(\xi)$ is not invertible in $\widetilde{B}_\rig^+$. Assume contrary that there exists $x\in \widetilde{B}_\rig^+$ such that $\varphi(\xi)x=1$. There exists $y\in B_\crys^+$ such that $x= \varphi(y)$. Then, $1=\varphi(\xi)x= \varphi(\xi)\varphi(y)=\varphi(\xi y)$ in $B_\crys^+$. Since $\varphi\colon B_\crys^+ \rightarrow B_\crys^+$ is injective, this implies $\xi y=1$ in $B_\crys^+$, i.e., $\xi$ is invertible in $B_\crys^+$. This is a contradiction since $B_\crys^+ \subseteq B_\dR^+$ and $\xi$ is not invertible in $B_\dR^+$.
\end{remark}

\begin{lemma}[{\cite[Lemma 12.2]{BhattMorrowScholze}, \cite[Lemma 90 (2)]{Tsuji_simons}}]\label{lem:mu is invertible in Acrys}
$t/\mu \in B_{\dR}^+$ lies in $A_\crys^\times$ and hence we have a canonical homomorphism $A_\inf[1/\mu] \rightarrow B_\crys$.
\end{lemma}
\begin{remark}\label{rem:mu is invertible in Acrys}
The same argument shows $t/\mu\in (A_{\crys,\infty})^{\times}$. In fact, this is a stronger statement since $A_{\crys,\infty}$ is contained in $A_\crys$.
\end{remark}

\begin{proof}[Proof of {\cref{lem:mu is invertible in Acrys}}]
     This is due to \cite[Lemma 12.2]{BhattMorrowScholze} or \cite[Lemma 90 (2)]{Tsuji_simons}. We write the proof here for the convenience of the reader.
     We first show $t/\mu\in A_\crys$. Since
     \begin{align*}
           \frac{t}{\mu}=\sum_{i=1}^{\infty} (-1)^{i+1}\frac{\mu^{i-1}}{i}
      \end{align*}
 in $B_{\dR}^+$, it is enough to show that $\mu^{i-1}/(i!)$ is an element of $A_\crys$.
As $\mu^{p-1}\in pA_\crys$ and $\Ker (\theta\colon A_\crys  \rightarrow \mathcal{O}_{\mathbb{C}_p})$ is a divided power ideal, we have
      \begin{align*}
           \frac{1}{n!}\left(\frac{\mu^{p-1}}{p}\right)^n\in A_\crys
      \end{align*}
for any $n\in \mathbb{Z}_{\geq 0}$. Let $i-1=(p-1)m+r$ with $m\in \mathbb{Z}_{\geq 0}$ and $0\leq r\leq p-2$. Then
\begin{align*}
     \frac{\mu^{i-1}}{i!}=\frac{1}{m!}\left(\frac{\mu^{p-1}}{p}\right)^m\cdot \frac{m!p^m}{i!}\cdot \mu^r.
\end{align*}
So, it is enough to show that $v_p(i!)\leq v_p(m!)+m$. As $i=(p-1)m+r+1=pm+(r+1-m)\leq pm+(p-1)$, we have $v_p(i!)\leq v_p\bigl((pm+p-1)!\bigr)=v_p\bigl((pm)!\bigr)$. We prove an equality $v_p\bigl((pm)!\bigr)=v_p(m!)+m$ by induction on $m$. When $m=0$, it is clear. Assume we prove an equality $v_p\bigl((pm)!\bigr)=v_p(m!)+m$ for some $m\in \mathbb{Z}_{\geq 0}$. On the one hand, we have $v_p\bigl((p(m+1))!\bigr)=v_p\bigl(p(m+1)\bigr)+ v_p\bigl((pm)!\bigr)=1+v_p(m+1)+v_p\bigl((pm)!\bigr)$. On the other hand, we have $v_p\bigl((m+1)! \bigr)=v_p(m+1)+v_p(m!)$. Thus, the induction hypothesis implies $v_p\bigl((p(m+1))!\bigr)=v_p((m+1)!)+m+1$, as desired. Therefore, $t/\mu\in A_{\crys}$. More precisely, we see that $t/\mu\in 1+\Fil^1 A_\crys$ as $\Fil^1 A_\crys$ is $p$-adically complete and separated. Since $1+\Fil^1 A_\crys \subseteq A_\crys^{\times}$ as shown in the proof of \cref{lem:phi-module over A_inf to Bcrys}, this shows that $t/\mu\in A_\crys^\times$.
\end{proof}

\begin{theorem}\label{thm:relationship between a lot of conditions}
Let $N$ be a $(\varphi,\Gamma)$-module over $\AKp$. Then, we have the following implications of the conditions:
\[\begin{tikzcd}
\mathrm{(2-A)} \arrow[r,Rightarrow]\arrow[d,Rightarrow]\arrow[rr, Rightarrow, bend left=30]&\mathrm{(2-B)}\arrow[d,Rightarrow]&\mathrm{(2-C)}\arrow[d,Rightarrow]& \\ 
\mathrm{(2-a)} \arrow[r,Rightarrow]&\mathrm{(2-b)}\arrow[r,Rightarrow]&\mathrm{(2-c)}\arrow[r,Rightarrow]&T(N)\text{ is crystalline}
\end{tikzcd}\]
\end{theorem}
\begin{proof}
The vertical implications are clear since (2-a) (resp. (2-b), resp. (2-c)) is equivalent to the surjectivity of the homomorphism \eqref{eq:condition (2-A)} (resp. \eqref{eq:condition (2-B)}, resp. \eqref{eq:condition (2-C)}). The implication (2-a) $\Rightarrow$ (2-b) is also clear since $(N/[\a])^{\Gamma_K}_{\icrys}\subseteq (N/[\a])^{\Gamma_K}_{\fin}$. 
Assume (2-A) holds. We first show that it implies (2-B). By \cref{prop:Galois fix part,prop:phi-finite part and tensor product}, taking the $\Gamma_K$-fixed part and the $\varphi$-finite part of \eqref{eq:condition (2-A)} gives us an isomorphism
\begin{align*}
    (N/[\a])^{\Gamma_K}_{\icrys} \otimes_{W(k)} (\AKp/[\a])_{\fin}^{\Gamma_K}  \xrightarrow{\cong } (N/[\a])^{\Gamma_K}_{\fin}.
\end{align*}
This implies (2-B). 
We next deduce (2-C) from (2-A). By tensoring \eqref{eq:condition (2-A)} with $\overline{A}_\infty$, we obtain an isomorphism 
\begin{align*}
    (N/[\a])^{\Gamma_K}_{\icrys} \otimes_{W(k)}\overline{A}_\infty \xrightarrow{\cong} N/\mathfrak{q}_\infty. 
\end{align*}
Combined with \cref{prop:Galois fix part,prop:phi-finite part and tensor product,rem:G-fixed part of Bbar}, it implies (2-C).
Assume (2-b) holds. Then the homomorphism \eqref{eq:condition (2-B)} is surjective. By tensoring \eqref{eq:condition (2-B)} with $\overline{A}_\infty$, we obtain a surjectivity of the top horizontal arrow in the following commutative diagram:
\[\begin{tikzcd}
(N/[\a])_{\fin}^{\Gamma_K} \otimes_{(\AKp/[\a])_{\fin}^{\Gamma_K}} \overline{A}_\infty  \arrow[r,->>] \arrow[d]& N/\mathfrak{q}_\infty\arrow[d,equal]\\
(N/\mathfrak{q}_\infty)^{\Gamma_K}_{\fin} \otimes_{W(k)} \overline{A}_{\infty} \arrow[r]& N/\mathfrak{q}_\infty.
\end{tikzcd}\]
Hence, the bottom arrow is also surjective, which implies (2-c).

It remains to show that (2-c) implies that $T(N)$ is crystalline. 
By \cref{lem:phi finite part as W(k) and W(l)}, we have a commutative diagram
\[\begin{tikzcd}
(N/\mathfrak{q}_\infty)^{\Gamma_K}_{\fin,W(k)}\otimes_{W(k)} \overline{B}\arrow[r,->>]\arrow[d]& N/\mathfrak{q}_\infty \otimes_{\overline{A}_\infty}\overline{B}\arrow[d,equal]\\ 
((N\otimes_{\AKp}\overline{B})_{\fin,W(\overline{k})})^{G_K}\otimes_{K_0} \overline{B} \arrow[r]& N\otimes_{\AKp}\overline{B}.
\end{tikzcd}\]
Hence, the condition (2-c) implies that $N \otimes_{\AKp}\overline{B}$ is generated by $((N\otimes_{\AKp}\overline{B})_{\fin,W(\overline{k})})^{G_K}$ as a $\overline{B}$-module. 
On the one hand, by \cref{lem:phi-module over A_inf to Bcrys}, $N\otimes_{\AKp} B_\crys^+$ is a finite height $\varphi$-module, and by \cref{thm:phi-finite part of B_crys and Bbar} we have a canonical isomorphism $(N\otimes_{\AKp}B_\crys^+)_{\fin,W(\overline{k})} \xrightarrow{\cong} (N \otimes_{\AKp}\overline{B})_{\fin,W(\overline{k})}$. By taking the $G_K$-fixed part, we obtain an isomorphism
\begin{align}\label{eq:isomorphism between B_crys and Bbar}
    ((N\otimes_{\AKp}B_\crys^+)_{\fin,W(\overline{k})})^{G_K} \xrightarrow{\cong} ((N \otimes_{\AKp}\overline{B})_{\fin,W(\overline{k})})^{G_K}.
\end{align}
On the other hand, by tensoring \eqref{eq:isomorphism between N and T after inverting mu} with $B_\crys$, which is possible by virtue of \cref{lem:mu is invertible in Acrys}, we obtain an isomorphism $N \otimes_{\AKp}B_\crys \cong T(N) \otimes_{\mathbb{Z}_p} B_\crys$. By taking the $G_K$-fixed part, we obtain
\begin{align*}
    (N\otimes_{\AKp}B_\crys^+)^{G_K} \subseteq  (N \otimes_{\AKp} B_\crys)^{G_K} \cong D_\crys(T(N)[1/p]).
\end{align*}
As the rightmost term is a $K_0$-vector space of dimension at most $d:=\rank_{\mathbb{Z}_p} T(N)$, $(N\otimes_{\AKp}B_\crys^+)^{G_K}$ is a $K_0$-vector space of dimension at most $d$. It implies that the leftmost term of \eqref{eq:isomorphism between B_crys and Bbar} is equal to $(N \otimes_{\AKp}B_\crys^+)^{G_K}$ and it is a $K_0$-vector space of dimension at most $d$. Since it generates the $\overline{B}$-module $N\otimes_{\AKp} \overline{B}$, which is free of rank $d$ by \cref{cor:well-definedness as p-adic rep}, we see that $\dim_{K_0} (N\otimes_{\AKp}B_\crys^+)^{G_K} = d$. As a consequence, $\dim_{K_0} D_\crys(T(N)[1/p])=d$, as desired.
\end{proof}

\begin{remark}\label{rem:relationship between a lot of conditions}
     From free crystalline $\mathbb{Z}_p$-representations, we will construct crystalline $(\varphi,\Gamma)$-modules satisfying the condition (2-A). Therefore, these conditions are equivalent for $(\varphi,\Gamma)$-modules over $\AKp$ whose $\Gamma_K$-action is trivial modulo $\mu$.
\end{remark}

\subsection{Some properties of crystalline $(\varphi,\Gamma)$-modules over $\AKp$}


\begin{proposition}\label{prop:tensor product}
Let $N_1$ and $N_2$ be crystalline $(\varphi,\Gamma)$-modules over $\AKp$. We define a Frobenius map on $N_1 \otimes_{\AKp} N_2$ by
\begin{align*}
    (N_1 \otimes_{\AKp}N_2)\otimes_{\AKp,\varphi} \AKp[1/\varphi(\xi)] &\cong (N_1\otimes_{\AKp,\varphi} \AKp[1/\varphi(\xi)]) \otimes_{\AKp[1/\varphi(\xi)]} (N_2\otimes_{\AKp,\varphi} \AKp[1/\varphi(\xi)])\\
     &\xrightarrow[\Phi_{N_1}\otimes \Phi_{N_2}]{\cong} N_1[1/\varphi(\xi)] \otimes_{\AKp[1/\varphi(\xi)]} N_2[1/\varphi(\xi)]
\end{align*}
and a $\Gamma_K$-action on it by $g \otimes g$ for any $g\in \Gamma_K$. Then, $N_1 \otimes_{\AKp} N_2$ is a crystalline $(\varphi,\Gamma)$-module over $\AKp$ that corresponds to the free $\mathbb{Z}_p$-representation $T(N_1) \otimes_{\mathbb{Z}_p} T(N_2)$. 
\end{proposition}
\begin{proof}
It is easy to see that $N_1 \otimes_{\AKp} N_2$ is finite free over $\AKp$ and that the Frobenius map is a $\Gamma_K$-equivariant isomorphism. We have an isomorphism
\begin{align*}
    (N_1 \otimes_{\AKp}N_2)/\mu &\cong N_1/\mu \otimes_{\AKp/\mu} N_2/\mu\\
     &\xleftarrow{\cong}  \bigl((N_1/\mu)^{\Gamma_K} \otimes_{(\AKp/\mu)^{\Gamma_K}} \AKp/\mu\bigr) \otimes_{\AKp/\mu} ((N_2/\mu)^{\Gamma_K}\otimes_{(\AKp/\mu)^{\Gamma_K}} \AKp/\mu)\\
     &\cong \bigl((N_1/\mu)^{\Gamma_K} \otimes_{(\AKp/\mu)^{\Gamma_K}} (N_2/\mu)^{\Gamma_K}\bigr) \otimes_{(\AKp/\mu)^{\Gamma_K}} \AKp/\mu.
\end{align*}
By \cref{prop:Galois fix part}, we see that the $\Gamma_K$-action on $N_1 \otimes_{\AKp} N_2$ is trivial modulo $\mu$. Combining them with \cref{prop:automatic continuity}, we see that $N_1 \otimes_{\AKp} N_2$ is a $(\varphi,\Gamma)$-module over $\AKp$. By \cref{thm:relationship between a lot of conditions}, crystalline $(\varphi,\Gamma)$-modules satisfy the condition (2-c) in \cref{def:crystalline (phi Gamma)-modules}. Hence, we have a surjective homomorphism
\begin{align*}
    (N_1 \otimes_{\AKp}N_2)/\mathfrak{q}_{\infty}&\cong N_1/\mathfrak{q}_\infty \otimes_{\overline{A}_\infty} N_2/\mathfrak{q}_\infty\\
    &\twoheadleftarrow  \bigl((N_1/\mathfrak{q}_\infty)^{\Gamma_K}_{\fin} \otimes_{W(k)} \overline{A}_\infty\bigr) \otimes_{\overline{A}_\infty} ((N_2/\mathfrak{q}_\infty)^{\Gamma_K}_{\fin}\otimes_{W(k)} \overline{A}_\infty)\\
     &\cong \bigl((N_1/\mathfrak{q}_\infty)^{\Gamma_K}_{\fin} \otimes_{W(k)} (N_2/\mathfrak{q}_\infty)^{\Gamma_K}_{\fin}\bigr) \otimes_{W(k)} \overline{A}_\infty.
\end{align*}
Since it factors through $((N_1 \otimes_{\AKp} N_2)/\mathfrak{q}_\infty)^{\Gamma_K}_{\fin}\otimes_{W(k)} \overline{A}_\infty$, we see that $N_1 \otimes_{\AKp}N_2$ satisfies the condition (2-c) in \cref{def:crystalline (phi Gamma)-modules}. Therefore, it is a crystalline $(\varphi,\Gamma)$-module over $\AKp$. By using the isomorphism \eqref{eq:isomorphism of Fontaine}, we see that the corresponding free $\mathbb{Z}_p$-representation is
\begin{align*}
    T(N_1 \otimes_{\AKp}N_2)&=\bigl((N_1 \otimes_{\AKp}N_2)\otimes_{\AKp} W(\mathbb{C}_p^\flat)\bigr)^{\varphi=1} \\
    &\cong \Bigl(\bigl(N_1 \otimes_{\AKp} W(\mathbb{C}_p^\flat)\bigr) \otimes_{W(\mathbb{C}_p^\flat)} \bigl(N_2 \otimes_{\AKp} W(\mathbb{C}_p^\flat)\bigr)\Bigr)^{\varphi=1}\\
    &\cong \Bigl(\bigl(T(N_1)\otimes_{\mathbb{Z}_p}W(\mathbb{C}_p^\flat)\bigr)\otimes_{W(\mathbb{C}_p^\flat)} \bigl(T(N_2)\otimes_{\mathbb{Z}_p}W(\mathbb{C}_p^\flat)\bigr)\Bigr)^{\varphi=1}\\
    &\cong \Bigl(\bigl(T(N_1) \otimes_{\mathbb{Z}_p} T(N_2)\bigr) \otimes_{\mathbb{Z}_p} W(\mathbb{C}_p^\flat)\Bigr)^{\varphi=1} \\
    &\cong T(N_1) \otimes_{\mathbb{Z}_p} T(N_2).
\end{align*}
\end{proof}

\begin{example}\label{ex:Tate twist}
Let $n$ be an integer. Then, $\mu^n \AKp(-n)$ equipped with the $\Gamma_K$-action and the Frobenius endomorphism induced by those of $\AKp$ is a crystalline $(\varphi,\Gamma)$-module over $\AKp$ that corresponds to $\mathbb{Z}_p(-n)$. We show this. Since $\AKp$ is $\mu$-torsion free, $\mu^n \AKp(-n)$ is finite free of rank $1$. 
The bijectivity of the Frobenius map follows from a calculation $\varphi(\mu^n x)=\varphi(\mu)^n \varphi(x)=\mu^n\cdot \varphi(\xi)^n \varphi(x)$. By \cref{lem:Tsuji Lemma 75}, we have a $\Gamma_K$-equivariant isomorphism
     \begin{align*}
          \AKp/\mu \cong \mu^n\AKp(-n)/\mu^{n+1}\AKp(-n),\qquad x\mapsto \mu^n x.
     \end{align*}
It follows that the $\Gamma_K$-action on $\mu^n\AKp(-n)$ is trivial modulo $\mu$. Combining them with \cref{prop:automatic continuity}, we conclude that $\mu^n\AKp(-n)$ is a $(\varphi,\Gamma)$-module over $\AKp$ whose $\Gamma_K$-action is trivial modulo $\mu$.
The corresponding $\mathbb{Z}_p$-representation is 
     \begin{align*}
          \Bigl(\mu^n \AKp(-n) \otimes_{\AKp}W(\mathbb{C}_p^\flat)\Bigr)^{\varphi=1}=\Bigr(\AKp\otimes_{\AKp}W(\mathbb{C}_p^\flat)\Bigl)^{\varphi=1}(-n)=\mathbb{Z}_p(-n).
     \end{align*}
     It remains to show that $\mu^n\AKp(-n)$ satisfies \cref{def:crystalline (phi Gamma)-modules} (2-C). The following proof is due to Tsuji. Recall that we constructed a canonical surjection $A_{\crys,\infty} \twoheadrightarrow \overline{A}_\infty$ in \cref{prop:Abar and Acrys}. By \cref{rem:mu is invertible in Acrys}, we have a homomorphism
     \begin{align*}
          t^n A_{\crys, \infty}(-n) =\mu^nA_{\crys,\infty}(-n) \twoheadrightarrow \mu^n \overline{A}_\infty(-n)
     \end{align*}
     which is compatible with the Frobenii and the $\Gamma_K$-actions. 
     Let $e\in \mu^n\overline{A}_\infty(-n)$ be the image of $t^n$ under the above map. 
     \cref{lem:Abar is mu-torsion free} implies that $e$ is a basis of $\mu^n \overline{A}_\infty(-n)$ over $\overline{A}_\infty$ and $e\in \bigl(\mu^n \overline{A}_\infty(-n)\bigr)_\fin^{\Gamma_K}$.
      Consider the $\Gamma_K$-equivariant $\overline{A}_\infty$-linear isomorphism
     \begin{align*}
          \overline{A}_\infty  \xrightarrow{\cong } \mu^n \overline{A}_\infty(-n),\qquad a\mapsto ae.
     \end{align*}
     Taking the $\Gamma_K$-fixed part gives us an isomorphism $(\overline{A}_\infty)^{\Gamma_K} \xrightarrow{\cong } \bigl(\mu^n \overline{A}_\infty(-n)\bigr)^{\Gamma_K} $. Since $\varphi(e)=p^n e$, one can show that for any $a\in \overline{A}_\infty$, $a\in (\overline{A}_\infty)_\fin$ if and only if $ae\in \bigl(\mu^n \overline{A}_\infty(-n)\bigr)_\fin$. Thus we obtain a $W(k)$-linear isomorphism
     \begin{align*}
          W(k)\cong (\overline{A}_\infty)^{\Gamma_K}_\fin  \xrightarrow{\cong } \bigl(\mu^n \overline{A}_\infty(-n)\bigr)_\fin^{\Gamma_K}
     \end{align*}
     and we see that the canonical map $\bigl(\mu^n\overline{A}_\infty(-n)\bigr)_\fin^{\Gamma_K} \otimes_{W(k)} \overline{A}_\infty  \rightarrow \mu^n \overline{A}_\infty(-n)$ is an isomorphism. This completes the proof.
\end{example}

\begin{corollary}\label{cor:Tate twist}
Let $N$ be a crystalline $(\varphi,\Gamma)$-module over $\AKp$ and let $n$ be an integer. Then, $\mu^n N(-n)$ is a crystalline $(\varphi,\Gamma)$-module over $\AKp$ that corresponds to the free $\mathbb{Z}_p$-representation $T(N)(-n)$.
\end{corollary}
\begin{proof}
This is a consequence of \cref{prop:tensor product,ex:Tate twist}.
\end{proof}

\begin{proposition}\label{prop:scalar extension of crystalline phi gamma module}
Let $N$ be a crystalline $(\varphi,\Gamma)$-module over $\AKp$. Let $L$ be a finite extension of $K$ inside $\overline{K}$. Let $L_\infty$ denote the $p$-cyclotomic extension $\bigcup_{n=1}^\infty L(\zeta_{p^n})$ of $L$. We define $G_L:=\Gal(\overline{K}/L)$, $\Gamma_L := \Gal(L_\infty/L)$, and $\ALp:=W(\mathcal{O}_{\widehat{L_\infty}}^\flat)$. Note that $N$ has a $\Gamma_L$-action via the homomorphism $\Gamma_L \rightarrow \Gal(K_\infty/(L\cap K_\infty))\subseteq \Gamma_K$. Then, $N\otimes_{\AKp} \ALp$ equipped with the canonical $\Gamma_L$-action and the Frobenius map is a crystalline $(\varphi,\Gamma)$-module over $\ALp$ that corresponds to the free $\mathbb{Z}_p$-representation $\left.T(N)\right|_{G_L}$.
\end{proposition}
\begin{proof}
By \cref{prop:scalar extension of phi-modules} (3), $N \otimes_{\AKp}\ALp$ is a $(\varphi,\Gamma)$-module over $\ALp$. Then, we have an isomorphism
\begin{align*}
    (N\otimes_{\AKp}\ALp)/\mu &\cong N/\mu \otimes_{\AKp/\mu} \ALp/\mu\\
     &\xleftarrow{\cong}  \bigl((N/\mu)^{\Gamma_K} \otimes_{(\AKp/\mu)^{\Gamma_K}} \AKp/\mu\bigr) \otimes_{\AKp/\mu} \ALp/\mu\\
     &\cong (N/\mu)^{\Gamma_K} \otimes_{(\AKp/\mu)^{\Gamma_K}} \ALp/\mu.
\end{align*}
Combining this with \cref{prop:Galois fix part}, we see that the $\Gamma_L$-action on $N\otimes_{\AKp}\ALp$ is trivial modulo $\mu$. 
By \cref{thm:relationship between a lot of conditions}, crystalline $(\varphi,\Gamma)$-modules satisfy the condition (2-c) in \cref{def:crystalline (phi Gamma)-modules}.
We set $\mathfrak{q}_{\infty,L}=\{[a]x\mid a\in \mathfrak{m}_{\widehat{L_\infty}}^\flat, x\in \ALp\}$. Then, we have a surjective homomorphism
\begin{align*}
    (N\otimes_{\AKp}\ALp)/\mathfrak{q}_{\infty,L} &\cong N/\mathfrak{q}_\infty \otimes_{\overline{A}_\infty} \ALp/\mathfrak{q}_{\infty,L}\\
     &\twoheadleftarrow  \bigl((N/\mathfrak{q}_\infty)^{\Gamma_K}_{\fin,W(k)} \otimes_{W(k)} \overline{A}_\infty\bigr) \otimes_{\overline{A}_\infty} \ALp/\mathfrak{q}_{\infty,L}\\
     &\cong (N/\mathfrak{q}_\infty)^{\Gamma_K}_{\fin,W(k)} \otimes_{W(k)} \ALp/\mathfrak{q}_{\infty,L}.
\end{align*}
Since it factors through $((N\otimes_{\AKp}\ALp)/\mathfrak{q}_{\infty,L})^{\Gamma_L}_{\fin, W(l)}\otimes_{W(l)} \ALp/\mathfrak{q}_{\infty,L}$ where $l$ denotes the residue field of $L$, this implies that $N\otimes_{\AKp}\ALp$ satisfies the condition (2-c) in \cref{def:crystalline (phi Gamma)-modules}. Therefore, it is a crystalline $(\varphi,\Gamma)$-module over $\ALp$. It is easy to see that the corresponding free $\mathbb{Z}_p$-representation is $\left.T(N)\right|_{G_L}$.
\end{proof}

\subsection{Summary of prisms}
In order to construct a quasi-inverse functor of $T\colon \Mod_{\varphi,\Gamma}^{\fh, \crys}(\AKp) \rightarrow \Rep_{\mathbb{Z}_p}^\crys(G_K)$, we use the theory of prisms. We recall some facts on prisms. For details, see \cite{BhattScholze2022,BhattScholze2023}. 

\begin{definition}[{\cite[Definition 2.1, Definition 3.2]{BhattScholze2022}}]\label{def:delta ring}
A {\em$\delta$-ring} is a pair $(A, \delta)$ where $A$ is a $\mathbb{Z}_{(p)}$-algebra and $\delta\colon A \rightarrow A$ is a map of sets satisfying the following conditions:
\begin{align*}
     \delta(0)&=\delta(1)=0,\\
     \delta(xy)&=x^p \delta(y)+y^p \delta(x)+p \delta(x) \delta(y),\\
     \delta(x+y)&= \delta(x)+ \delta(y)-\sum_{i=1}^{p-1}\frac{1}{p}\begin{pmatrix}
     p\\ 
     i
     \end{pmatrix}x^i y^{p-i}.
\end{align*}
A \emph{morphism of $\delta$-rings} is a ring homomorphism of underlying rings that is compatible with $\delta$.
A {\em$\delta$-pair $(A, \delta,I)$} is a pair of a $\delta$-ring $(A,\delta)$ and an ideal $I$ of $A$. A {\em morphism of $\delta$-pairs} $f\colon (A,\delta_A,I) \rightarrow (B,\delta_B,J)$ is a morphism of $\delta$-rings satisfying $f(I)\subseteq J$.
\end{definition}

For any $\delta$-ring  $(A,\delta)$, the map 
\begin{align*}
     \varphi\colon A \rightarrow A,\qquad \varphi(x):=x^p+p \delta(x)
\end{align*}
is a Frobenius lift, i.e., $\varphi$ is a ring homomorphism satisfying $\varphi(x)\equiv x^p\bmod p$. If $A$ is a $p$-torsion free $\mathbb{Z}_{(p)}$-algebra, then every Frobenius lift $\varphi\colon A \rightarrow A$ comes from a unique $\delta$-structure on $A$ (\cite[Remark 2.2]{BhattScholze2022}). In the following, we often omit $\delta$ from the notation. 

\begin{definition}[{\cite[Definition 3.2]{BhattScholze2022}}]\label{def:prism}
A {\em bounded prism $(A,I)$} is a $\delta$-pair $(A,I)$ satisfying the following conditions:
\begin{enumerate}
     \item $I$ is an invertible ideal.
     \item $A$ is $(p,I)$-adically complete and separated.
     \item $p\in I+\varphi(I)A$ where $\varphi$ is the induced Frobenius lift.
     \item $A/I$ has bounded $p^{\infty}$-torsion, i.e., $(A/I)[p^\infty]=(A/I)[p^N]$ for some $N\gg 0$. Here, we define $(A/I)[p^\infty]$ and $(A/I)[p^N]$ by
     \begin{align*}
          (A/I)[p^\infty]&:=\{x\in A/I\mid p^nx=0 \text{ for some } n\in \mathbb{Z}_{\geq 0}\},\\
          (A/I)[p^N]&:=\{x\in A/I\mid p^Nx=0\}.
     \end{align*} 
\end{enumerate}
\end{definition}

\begin{example}\label{ex:(W(k) p)}
For any $p$-adically complete and separated and $p$-torsion free ring $A$ equipped with a Frobenius lift, $(A,p)$ is a bounded prism. In particular, $(W(k),p)$ is a bounded prism with $\delta$-structure induced by the Frobenius map of Witt vectors. 
\end{example}

\begin{example}\label{ex:(AKp xi)}
$(\AKp, \xi)$ is a bounded prism with $\delta$-structure induced by the Frobenius map of Witt vectors. Let us show this. Since $\AKp$ and $\AKp/\xi\cong \mathcal{O}_{\Kinf}$ are integral domains, the conditions (1) and (4) are satisfied. \cref{lem:Ainf is a local ring} implies the condition (2). Since
\begin{align*}
     \xi:=\mu/\varphi^{-1}(\mu)=1+[\epsilon]^{1/p}+\dots +[\epsilon]^{(p-1)/p}=p+\varphi^{-1}(\mu)y
\end{align*}
for some $y\in \AKp$, $p=\varphi(\xi)-\mu \varphi(y)\in (\xi, \varphi(\xi))\AKp$, so the condition (3) is satisfied. 
Thus, $(\AKp, \xi)$ is a bounded prism. Similar argument shows that $(\AKp, \varphi(\xi))$ is also a bounded prism and the Frobenius endomorphism of $\AKp$ induces an isomorphism $(\AKp,\xi) \xrightarrow[\cong ]{\varphi} (\AKp,\varphi(\xi))$ as prisms.
\end{example}

\begin{definition}[{\cite[Definition 2.3]{BhattScholze2023}}]\label{def:prismatic site}
Let $\mathcal{O}_{K,\Prism}$ denote the opposite category of bounded prisms $(A,I)$ equipped with an $\mathcal{O}_K$-algebra structure $\mathcal{O}_K \rightarrow A/I$. A morphism $(B,J) \rightarrow (A,I)$ in $\mathcal{O}_{K,\Prism}$ is a morphism of prisms $(A,I) \rightarrow (B,J)$ such that $A/I \rightarrow B/J$ is an $\mathcal{O}_K$-algebra homomorphism.

We endow $\mathcal{O}_{K,\Prism}$ with a faithfully flat topology, i.e., a cover is a morphism $(B,J) \rightarrow (A,I)$ of $\mathcal{O}_{K,\Prism}$ such that the underlying ring homomorphism $A \rightarrow B$ is $(p,I)$-completely faithfully flat. We call the site $\mathcal{O}_{K,\Prism}$ the {\em absolute prismatic site over $\mathcal{O}_K$}.
\end{definition}

\begin{definition}[{\cite[Proposition 2.7, Definition 4.1]{BhattScholze2023}}]\label{def:prismatic crystal}
Let $\mathcal{O}_\Prism$ (resp. $\mathcal{O}_\Prism[1/\mathcal{I}_\Prism]^{\wedge}_p$) denote the sheaf on $\mathcal{O}_{K,\Prism}$ defined by $(A,I)\mapsto A$ (resp. $(A,I)\mapsto  A[1/I]^{\wedge}_p$). A {\em prismatic $F$-crystal in $\mathcal{O}_\Prism$-modules (resp. $\mathcal{O}_\Prism[1/\mathcal{I}_\Prism]^{\wedge}_p$-modules) on $\mathcal{O}_{K,\Prism}$} is an $\mathcal{O}_\Prism$-module (resp. $\mathcal{O}_\Prism[1/\mathcal{I}_\Prism]^{\wedge}_p$-module) $\mathcal{M}$ on $\mathcal{O}_{K,\Prism}$ equipped with an $\mathcal{O}_\Prism[1/\mathcal{I}_\Prism]$-linear (resp. $\mathcal{O}_\Prism[1/\mathcal{I}_\Prism]^{\wedge}_p$-linear) isomorphism $\varPhi_{\mathcal{M}}\colon (\varphi^\ast \mathcal{M})[1/\mathcal{I}_\Prism]\cong \mathcal{M}[1/\mathcal{I}_\Prism]$ satisfying the following conditions:
\begin{enumerate}
     \item For any object $(A,I)$ in $\mathcal{O}_{K,\Prism}$, $\mathcal{M}(A,I)$ is a finite projective $A$-module (resp. finite projective $A[1/I]^\wedge_p$-module).
     \item (Crystal property) For any morphism $(B,J) \rightarrow (A,I)$ in $\mathcal{O}_{K,\Prism}$, the canonical $B$-linear (resp. $B[1/J]^\wedge_p$-linear) homomorphism
     \begin{align*}
          \mathcal{M}(A,I) \otimes_{A} B  \rightarrow \mathcal{M}(B,J)\\
          (\text{resp. }\mathcal{M}(A,I) \otimes_{A[1/I]^\wedge_p} B[1/J]^\wedge_p  \rightarrow \mathcal{M}(B,J))
     \end{align*}
     is an isomorphism.
\end{enumerate}
Let $\Vect^\varphi(\mathcal{O}_{K,\Prism}, \mathcal{O}_\Prism)$ (resp. $\Vect^\varphi(\mathcal{O}_{K,\Prism}, \mathcal{O}_\Prism[1/\mathcal{I}_\Prism]^{\wedge}_p)$ ) denote the category of prismatic $F$-crystals in $\mathcal{O}_\Prism$-modules (resp. $\mathcal{O}_{\Prism}[1/\mathcal{I}_\Prism]^{\wedge}_p$-modules) on $\mathcal{O}_{K,\Prism}$.
\end{definition}

Z. Wu proved the following theorem which states the relationship between prismatic $F$-crystals in $\mathcal{O}_{\Prism}[1/\mathcal{I}_\Prism]^{\wedge}_p$-modules on $\mathcal{O}_{K,\Prism}$ and $(\varphi,\Gamma)$-modules over $\AK$.

\begin{theorem}[{\cite[Theorem 5.2]{Wu2021}}]\label{thm:prismatic F-crystal and p-adic rep}
     Let $(\AKp, \varphi(\xi))$ denote the object of $\mathcal{O}_{K,\Prism}$ with $\mathcal{O}_K$-algebra structure 
     \begin{align*}
          \mathcal{O}_K \hookrightarrow  \mathcal{O}_{\Kinf} \xleftarrow[\cong ]{\theta } \AKp/\xi  \xrightarrow[\cong ]{\varphi} \AKp/\varphi(\xi).
     \end{align*}
     Then, there exists an equivalence of categories between the category of prismatic $F$-crystals in $\mathcal{O}_{\Prism}[1/\mathcal{I}_\Prism]^{\wedge}_p$-modules on $\mathcal{O}_{K,\Prism}$ and the category of $(\varphi,\Gamma)$-modules over $\AK$ via 
     \begin{align*}
          \Vect^\varphi(\mathcal{O}_{K,\Prism},\mathcal{O}_\Prism[1/\mathcal{I}_\Prism]^\wedge_p) \rightarrow \Mod_{\varphi,\Gamma}^{\et}(\AK),\qquad \mathcal{M} \mapsto \mathcal{M} (\AKp, \varphi(\xi)).
     \end{align*}
     Here, the $\Gamma_K$-action on $\mathcal{M}(\AKp, \varphi(\xi))$ is defined as follows. 
     The $\Gamma_K$-action on $\AKp$ commutes with the Frobenius map of $\AKp$ and preserves $\varphi(\xi)\AKp$. Also for any $g\in \Gamma_K$, the ring homomorphism $g\colon \AKp/\varphi(\xi) \rightarrow \AKp/\varphi(\xi)$ is a morphism of $\mathcal{O}_K$-algebras. Therefore, $g\colon (\AKp, \varphi(\xi)) \rightarrow (\AKp, \varphi(\xi))$ defines a morphism in $\mathcal{O}_{K,\Prism}$, which induces a semi-linear and continuous $\Gamma_K$-action on $\mathcal{M}(\AKp,\varphi(\xi))$ that commutes with the Frobenius map.
\end{theorem}

\begin{remark}
     \cite[Lemma 2.5]{Wu2021} does not hold; see \cite{WatanabeAKp}. However, Wu's method can be applied by using the prism $(\AKp,\varphi(\xi))$ instead of $(\mathbb{A}_K^+, \phi^n([p]_q))$ and the above theorem can be proved by a similar argument.
\end{remark}

Also, B. Bhatt and P. Scholze discovered the relationship between prismatic $F$-crystals in $\mathcal{O}_\Prism$-modules on $\mathcal{O}_{K,\Prism}$ and free crystalline $\mathbb{Z}_p$-representations of $G_K$ as follows:

\begin{theorem}[{\cite[Theorem 5.6]{BhattScholze2023}}]\label{thm:prismatic F-crystal and crystalline rep}
There exists an equivalence of categories between the category of prismatic $F$-crystals in $\mathcal{O}_\Prism$-modules on $\mathcal{O}_{K,\Prism}$ and the category of free crystalline $\mathbb{Z}_p$-representations of $G_K$ via 
\begin{align*}
     \Vect^\varphi(\mathcal{O}_{K,\Prism}, \mathcal{O}_\Prism) \rightarrow \Rep_{\mathbb{Z}_p}^{\crys}(G_K),\qquad \mathcal{N} \mapsto \Bigl(\mathcal{N}(\AKp, \varphi(\xi))\otimes_{\AKp} W(\mathbb{C}_p^\flat)\Bigr)^{\varphi=1}.
\end{align*}
\end{theorem}

\subsection{Proof of the main theorem}
To construct a quasi-inverse functor of $T\colon \Mod_{\varphi,\Gamma}^{\fh, \crys}(\AKp) \rightarrow \Rep_{\mathbb{Z}_p}^\crys(G_K)$, we define a functor from $\Vect^\varphi(\mathcal{O}_{K,\Prism}, \mathcal{O}_\Prism)$ to $\Mod_{\varphi,\Gamma}^{\fh,\crys}(\AKp)$ by evaluating at $(\AKp,\varphi(\xi))$. We first define some objects in $\mathcal{O}_{K,\Prism}$. As we have already mentioned, $(\AKp, \varphi(\xi))$ is an object in $\mathcal{O}_{K,\Prism}$ via the $\mathcal{O}_K$-algebra structure
\begin{align*}
     \mathcal{O}_K  \hookrightarrow  \mathcal{O}_{\Kinf}  \xleftarrow[\cong ]{\theta} \AKp/\xi  \xrightarrow[\cong ]{\varphi} \AKp/\varphi(\xi).
\end{align*}
By \cref{lem:Ainf/x is p-torsion free}, $\AKp/\mu$ is $p$-torsion free and $p$-adically complete and separated. Since $\varphi(\mu)\in \mu\AKp$, $\AKp/\mu$ has a Frobenius lift induced by that of $\AKp$. Thus $(\AKp/\mu, p)$ is a bounded prism. As $\varphi(\xi)\in (p,\mu)\AKp$, the canonical surjection defines a morphism of prism $(\AKp, \varphi(\xi))\twoheadrightarrow  (\AKp/\mu, p)$. We endow $\AKp/(\mu,p)$ with an $\mathcal{O}_K$-algebra structure via
\begin{align*}
     \mathcal{O}_K  \hookrightarrow  \mathcal{O}_{\Kinf}  \xleftarrow[\cong ]{\theta} \AKp/\xi  \xrightarrow[\cong ]{\varphi} \AKp/\varphi(\xi)\twoheadrightarrow \AKp/(\mu,p).
\end{align*}
Then, the canonical map $(\AKp, \varphi(\xi)) \rightarrow (\AKp/\mu, p)$ defines a morphism in $\mathcal{O}_{K,\Prism}$. 

\begin{remark}\label{rem:reason of phi xi}
     This is the reason why we use the prism $(\AKp, \varphi(\xi))$ instead of the prism $(\AKp, \xi)$.
\end{remark}

\begin{lemma}\label{lem:Gamma-fixed part of AKp/mu is a prism}
The ring $\AKp/\mu$ has a natural $\Gamma_K$-action and $(\AKp/\mu)^{\Gamma_K}$ is $p$-adically complete and separated and $p$-torsion free.
\end{lemma}
\begin{proof}
Since $g(\mu)=(\mu+1)^{\chi_\cyc(g)}-1\in \mu\AKp$, $\AKp/\mu$ has the induced $\Gamma_K$-action. It follows from those properties of $\AKp/\mu$ that $(\AKp/\mu)^{\Gamma_K}\subseteq \AKp/\mu$ is $p$-torsion free and $p$-adically separated. By taking the $\Gamma_K$-fixed part of the exact sequence
\begin{align*}
     0 \rightarrow \AKp/\mu \xrightarrow{\times p^n}\AKp/\mu  \rightarrow (\AKp/\mu)/p^n  \rightarrow 0,
\end{align*}
we see that the canonical map $(\AKp/\mu)^{\Gamma_K}/p^n  \rightarrow (\AKp/\mu)/p^n$ is injective. Thus we have the following commutative diagram:
\[\begin{tikzcd}
\AKp/\mu \arrow[r, "\cong "] & \varprojlim_{n}(\AKp/\mu)/p^n \\ 
(\AKp/\mu)^{\Gamma_K}\arrow[u, hook]\arrow[r,hook]& \varprojlim_{n}(\AKp/\mu)^{\Gamma_K}/p^n. \arrow[u,hook]
\end{tikzcd}\]
The surjectivity of the bottom arrow follows from the fact that the top arrow is $\Gamma_K$-equivariant and the right arrow is injective.
\end{proof}

\begin{proposition}\label{prop:Gamma-fixed part of AKp/mu is a prism}
The ring $(\AKp/\mu)^{\Gamma_K}$ has an induced Frobenius lift and $\bigl((\AKp/\mu)^{\Gamma_K}, p\bigr)$ becomes a bounded prism. Also we can define an $\mathcal{O}_K$-algebra structure on $(\AKp/\mu)^{\Gamma_K}/p$ in such a way that the canonical map $\bigl((\AKp/\mu)^{\Gamma_K}, p\bigr) \rightarrow (\AKp/\mu, p)$ defines a morphism in $\mathcal{O}_{K,\Prism}$.
\end{proposition}
\begin{proof}
     Since the Frobenius endomorphism $\varphi\colon \AKp  \rightarrow \AKp$ commutes with the $\Gamma_K$-action on $\AKp$, $(\AKp/\mu)^{\Gamma_K}$ is stable under the Frobenius endomorphism and it becomes a Frobenius lift because $\AKp/\mu$ is $p$-torsion free. Hence $((\AKp/\mu)^{\Gamma_K},p)$ becomes a bounded prism by \cref{lem:Gamma-fixed part of AKp/mu is a prism}. 

     We have seen in the proof of \cref{lem:Gamma-fixed part of AKp/mu is a prism} that the canonical homomorphism $(\AKp/\mu)^{\Gamma_K}/p  \rightarrow \AKp/(\mu, p)$ is injective. Therefore, in order to define an $\mathcal{O}_K$-algebra structure on $(\AKp/\mu)^{\Gamma_K}/p$, it suffices to show that the image of $\mathcal{O}_K \rightarrow \AKp/(\mu,p)$ is contained in $(\AKp/\mu)^{\Gamma_K}/p$. 
     Since the composite $W(k)\hookrightarrow \mathcal{O}_{K} \rightarrow \AKp/(\mu,p)$ coincides with the composite $W(k)\xrightarrow{\varphi} W(k)\hookrightarrow \AKp \twoheadrightarrow \AKp/(\mu,p)$ and the latter factors through $(\AKp/\mu)^{\Gamma_K}/p$, it suffices to show that the image of $\pi\in \mathcal{O}_K$ is contained in $(\AKp/\mu)^{\Gamma_K}/p$. Our strategy is that we first define an element in $ (\AKp/\mu)^{\Gamma_K}/p$ and then show that it is the image of $\pi$. 
     By \cite[Lemma 3.23]{BhattMorrowScholze}, we have an exact sequence
     \begin{align*}
          0 \rightarrow \AKp/\mu  \xrightarrow{\theta_\infty} W(\mathcal{O}_{\Kinf}) \rightarrow \Coker \theta_{\infty}  \rightarrow 0 
     \end{align*}
     such that $\Coker \theta_\infty$ is killed by $W(\mathfrak{m}_{\Kinf}^\flat)$ via $\theta_{\infty}$.
     We can construct an element $\Pi\in \mathfrak{m}_{\Kinf}^\flat$ such that $v(\Pi)=v_p(\Pi^{(0)})<v_p(\pi)$, for example $\Pi= (\epsilon-1)^{1/p^n}$ for $n\gg 0$. Then $[\Pi]\in W(\mathfrak{m}_{\Kinf}^\flat)$ and 
     \begin{align*}
          [\pi]=[\Pi^{(0)}]\cdot [\pi/\Pi^{(0)}]=\theta_\infty([\Pi])\cdot [\pi/\Pi^{(0)}]\qquad \text{in }W(\mathcal{O}_{\Kinf}).
     \end{align*}
     For the second equality, see the sentence written after \cite[Lemma 3.4]{BhattMorrowScholze}. Thus there exists a unique element $\varpi\in \AKp/\mu$ such that $\theta_\infty(\varpi)=[\pi]\in W(\mathcal{O}_{\Kinf})$. Since $\theta_\infty\colon \AKp/\mu  \rightarrow W(\mathcal{O}_{\Kinf})$ is a $\Gamma_K$-equivariant injection, $\varpi\in (\AKp/\mu)^{\Gamma_K}$. Then the image of $\pi\in \mathcal{O}_K$ coincides with $\varpi^p\bmod p\in (\AKp/\mu)^{\Gamma_K}/p$ by the following commutative diagram:
     \[\begin{tikzcd}
     \mathcal{O}_K\arrow[r, hook]&\mathcal{O}_{\Kinf}&\AKp/\xi\arrow[r, "\varphi", "\cong "']\arrow[l, "\theta"', "\cong "] &\AKp/\varphi(\xi) \arrow[d, ->>]\\ 
     \pi\arrow[u, "\rotatebox{90}{$\in$}",phantom]\arrow[r, mapsto]&\pi\arrow[u, "\rotatebox{90}{$\in$}",phantom]&\AKp/\mu \arrow[r, "x\mapsto x^p"]\arrow[u, ->>] &\AKp/(\mu,p) \\ 
     &&(\AKp/\mu)^{\Gamma_K}\arrow[u, hook]\arrow[r, "x\mapsto x^p"]&(\AKp/\mu)^{\Gamma_K}/p\arrow[u,hook] \\ 
     &&\varpi\arrow[u, "\rotatebox{90}{$\in$}",phantom]\arrow[r, mapsto]\arrow[uul, mapsto, bend left=40,"\text{By the definition of $\varpi$}"]&\varpi^p\bmod p.\arrow[u, "\rotatebox{90}{$\in$}",phantom]
     \end{tikzcd}\]
\end{proof}

In summary, we obtain the following figure:
\[\begin{tikzcd}
(\AKp, \varphi(\xi))\arrow[d, ->>]&\mathcal{O}_K \arrow[r, hook]\arrow[rrrdd]&\mathcal{O}_{\Kinf}&\AKp/\xi\arrow[l, "\theta"', "\cong "]\arrow[r, "\cong "', "\varphi"]&\AKp/\varphi(\xi) \arrow[d, ->>]\\ 
(\AKp/\mu, p)&&&&\AKp/(\mu,p) \\ 
\bigl((\AKp/\mu)^{\Gamma_K},p\bigr)\arrow[u, hook]& &&&(\AKp/\mu)^{\Gamma_K}/p.\arrow[u, hook]
\end{tikzcd}\]
We will see that this diagram gives us a triviality modulo $\mu$ of the $\Gamma_K$-action on $\mathcal{N}\bigl(\AKp,\varphi(\xi)\bigr)$ for any prismatic $F$-crystal $\mathcal{N}$ in $\mathcal{O}_{\Prism}$-modules on $\mathcal{O}_{K,\Prism}$. 

We next construct some prisms which will be needed to show that $\mathcal{N}(\AKp,\varphi(\xi))$ satisfies the condition \cref{def:crystalline (phi Gamma)-modules} (2-A). Recall that we chose $\a\in \mathfrak{m}_{\Kinf}^\flat$ such that $v(\a)=pv_p(\pi)$. Since $\varphi([\a]x)=[\a]^p \varphi(x)$ for any $x\in \AKp$, $\varphi\colon \AKp \rightarrow \AKp$ induces a Frobenius lift $\varphi\colon \AKp/[\a]\rightarrow \AKp/[\a]$. Combining it with \cref{lem:Ainf/x is p-torsion free}, we see that $(\AKp/[\a],p)$ is a bounded prism. Recall the element $\xi_p:=pu-[\tau]$ defined in the proof of \cref{lem:Ainf is a local ring}. We have $(\varphi(\xi),p)=(\varphi(\xi_p),p)=(p,[\tau]^p)$ as ideals of $\AKp$. Since $v(\a)= pv_p(\pi)\leq p=v(\tau^p)$, we can  define the $\mathcal{O}_K$-algebra structure on $\AKp/([\a],p)$ by 
\begin{align*}
     \mathcal{O}_K  \hookrightarrow  \mathcal{O}_{\Kinf} \xleftarrow[\cong ]{\theta} \AKp/\xi \xrightarrow[\cong ]{\varphi} \AKp/\varphi(\xi)\twoheadrightarrow \AKp/([\a],p).
\end{align*}
Then the canonical map $(\AKp, \varphi(\xi)) \rightarrow (\AKp/[\a],p)$ defines a morphism in $\mathcal{O}_{K,\Prism}$. Note that the above homomorphism $\mathcal{O}_K  \rightarrow \AKp/([\a],p)$ factors through $\mathcal{O}_K/\pi\cong k$ since $v_p(\a^{(1)})= v_p(\pi)$ and $(\varphi(\xi), [\a])= (p \varphi(u)-[\tau]^p, [\a])=([\a],p)$.
\[\begin{tikzcd}
\mathcal{O}_K\arrow[r, hook]\arrow[d, ->>] &\mathcal{O}_{\Kinf}\arrow[r, "\cong "]\arrow[d, ->>] &\AKp/\varphi(\xi)\arrow[d,->>]\\ 
 \mathcal{O}_K/\pi\arrow[r, dashed]&\mathcal{O}_{\Kinf}/\a^{(1)}\arrow[r, "\cong "]&\AKp/([\a],p)
\end{tikzcd}\]

Note that the map $k\cong \mathcal{O}_K/\pi  \rightarrow \mathcal{O}_{\Kinf}/\a^{(1)} \xrightarrow{\cong } \AKp/([\a],p) $ defined above is the same as the composite of the Frobenius map on $k$ and the canonical inclusion $k\hookrightarrow \AKp/([\a],p)$. So if we define the $\mathcal{O}_K$-algebra structure on $W(k)/p$ by
\begin{align*}
     \mathcal{O}_K \twoheadrightarrow \mathcal{O}_K/\pi\cong k  \xrightarrow{x\mapsto x^p} k \cong W(k)/p,
\end{align*}
the canonical map $(W(k),p) \rightarrow \bigl(\AKp/[\a], p\bigr)$ defines a morphism in $\mathcal{O}_{K,\Prism}$. In summary, we obtain the following figure:
\[\begin{tikzcd}
(\AKp, \varphi(\xi))\arrow[d, ->>]&\mathcal{O}_K \arrow[r, hook]\arrow[d, ->>]&\mathcal{O}_{\Kinf} \arrow[d,->>]&\AKp/\xi\arrow[l, "\theta"', "\cong "]\arrow[r, "\varphi", "\cong "']&\AKp/\varphi(\xi)\arrow[d, ->>] \\ 
(\AKp/[\a],p)&\mathcal{O}_K/\pi\cong k\arrow[r]\arrow[drrr, "x\mapsto x^p"']&\mathcal{O}_{\Kinf}/\a^{(1)}\arrow[rr, "\cong "]&&\AKp/([\a],p) \\ 
(W(k),p)\arrow[u,"\text{canonical map}"']&&{}&&W(k)/p.\arrow[u]
\end{tikzcd}\]
We finally construct the quasi-inverse functor and show the equivalence of categories.

\begin{proposition}\label{prop:quasi-inverse functor}
We have a well-defined evaluation functor from the category of prismatic $F$-crystals in $\mathcal{O}_\Prism$-modules on $\mathcal{O}_{K,\Prism}$ to the category of crystalline $(\varphi,\Gamma)$-modules over $\AKp$
\begin{align*}
     \ev\colon \Vect^\varphi(\mathcal{O}_{K,\Prism}, \mathcal{O}_\Prism)  \rightarrow \Mod_{\varphi,\Gamma}^{\fh, \crys}(\AKp),\qquad \mathcal{N}\mapsto \ev(\mathcal{N}):= \mathcal{N}\bigl(\AKp, \varphi(\xi)\bigr).
\end{align*}
\end{proposition}
\begin{proof}
Let $\mathcal{N}$ be a prismatic $F$-crystal in $\mathcal{O}_\Prism$-modules on $\mathcal{O}_{K,\Prism}$ and we put $N:=\mathcal{N}(\AKp, \varphi(\xi))$. By the definition of prismatic $F$-crystals, $N$ is a finite height $\varphi$-module over $\AKp$. Since $g\colon \AKp  \rightarrow \AKp$ for every $g\in \Gamma_K$ defines a morphism $(\AKp, \varphi(\xi)) \rightarrow (\AKp, \varphi(\xi))$ in $\mathcal{O}_{K,\Prism}$ as we have explained in \cref{thm:prismatic F-crystal and p-adic rep}, $N$ has a semi-linear $\Gamma_K$-action which commutes with $\varphi_N$. 

We check that the $\Gamma_K$-action is trivial modulo $\mu$. Note that $g\colon \AKp/\mu  \rightarrow \AKp/\mu$ for any $g\in \Gamma_K$ defines a morphism $(\AKp/\mu,p) \rightarrow (\AKp/\mu,p)$ in $\mathcal{O}_{K,\Prism}$, which induces a $\Gamma_K$-action on $\mathcal{N}(\AKp/\mu,p)$. The crystal property implies the following isomorphisms compatible with the $\Gamma_K$-actions and the Frobenii:
\begin{align*}
     N/\mu  \xrightarrow{\cong } \mathcal{N}(\AKp/\mu,p), \\
     \mathcal{N}\bigl((\AKp/\mu)^{\Gamma_K},p\bigr)\otimes_{(\AKp/\mu)^{\Gamma_K}}\AKp/\mu \xrightarrow{\cong } \mathcal{N}(\AKp/\mu,p).
\end{align*}
Since $(\AKp/\mu)^{\Gamma_K}$ is fixed by $\Gamma_K$, $\mathcal{N}\bigl((\AKp/\mu)^{\Gamma_K},p\bigr)$ is also fixed by $\Gamma_K$. So \cref{prop:Galois fix part} shows $\mathcal{N}\bigl((\AKp/\mu)^{\Gamma_K},p\bigr)=\bigl(\mathcal{N}(\AKp/\mu,p)\bigr)^{\Gamma_K}\cong (N/\mu)^{\Gamma_K}$. Therefore, we conclude that $(N/\mu)^{\Gamma_K}$ is a finite projective $(\AKp/\mu)^{\Gamma_K}$-module and the canonical map
\begin{align*}
     (N/\mu)^{\Gamma_K} \otimes_{(\AKp/\mu)^{\Gamma_K}}\AKp/\mu \rightarrow N/\mu
\end{align*}
is an isomorphism.

Taking \cref{prop:automatic continuity,thm:relationship between a lot of conditions} into consideration, it remains to show that $N$ satisfies \cref{def:crystalline (phi Gamma)-modules} (2-A). Note that the ideal $[\a]\AKp$ is stable under the $\Gamma_K$-action since the $\Gamma_K$-action on $\mathcal{O}_{\Kinf}^\flat$ preserves the valuation, and that $g\colon \AKp/[\a]  \rightarrow \AKp/[\a]$ for any $g\in \Gamma_K$ defines a morphism $(\AKp/[\a],p) \rightarrow (\AKp/[\a],p)$ in $\mathcal{O}_{K,\Prism}$, which induces a $\Gamma_K$-action on $\mathcal{N}(\AKp/[\a],p)$. By applying the same argument to the morphism $(\AKp/[\a],p) \rightarrow (W(k),p)$ in $\mathcal{O}_{K,\Prism}$, we see that $\mathcal{N}(W(k),p)$ is a finite free $W(k)$-module and that we have an isomorphism 
\begin{align*}
    \mathcal{N}(W(k),p) \otimes_{W(k)} \AKp/[\a] \xrightarrow{\cong } N/[\a]
\end{align*}
compatible with the $\Gamma_K$-actions and the Frobenii. It suffices to show $\mathcal{N}(W(k),p)\cong (N/[\a])_\icrys^{\Gamma_K}$; this follows from \cref{prop:Galois fix part,prop:isocrystal part and tensor product}.
\end{proof}



\begin{theorem}\label{thm:Main theorem for crystalline}
     There exists an equivalence of categories between the category of crystalline $(\varphi,\Gamma)$-modules over $\AKp$ and the category of free crystalline $\mathbb{Z}_p$-representations of $G_K$ via the functor
     \begin{align*}
          T\colon \Mod_{\varphi,\Gamma}^{\fh, \crys}(\AKp)  \xrightarrow{\sim} \Rep_{\mathbb{Z}_p}^{\crys}(G_K),\qquad N \mapsto T(N):=\bigl(N\otimes_{\AKp} W(\mathbb{C}_p^\flat)\bigr)^{\varphi=1}.
     \end{align*}
\end{theorem}
\begin{proof}
We first show that the functor $T$ is full and faithful. 
Recall that for any $(\varphi,\Gamma)$-module $N$ over $\AKp$, we have an $A_\inf[1/\mu]$-linear isomorphism \eqref{eq:isomorphism between N and T after inverting mu}
\begin{align*}
     N\otimes_{\AKp} A_\inf[1/\mu]\cong T(N)\otimes_{\mathbb{Z}_p}A_\inf[1/\mu]
\end{align*}
which is compatible with the Frobenii and the $G_K$-actions and functorial in $N$. 

Let $f,g\colon N_1 \rightarrow N_2$ be morphisms of crystalline $(\varphi,\Gamma)$-modules over $\AKp$. Assume $T(f)$ is equal to $T(g)$. Then the following commutative diagram shows $f \otimes \id_{A_\inf[1/\mu]}=g\otimes\id_{A_\inf[1/\mu]}$. 
\[\begin{tikzcd}
N_1 \otimes_{\AKp} A_\inf[1/\mu]\arrow[r, "\cong ",phantom]\arrow[d, "f\otimes\id"',shift right=1.5ex]\arrow[d, "g \otimes \id", shift left=1.5ex] &T(N_1)\otimes_{\mathbb{Z}_p} A_\inf[1/\mu] \arrow[d, "T(f)\otimes \id=T(g)\otimes \id"]\\ 
N_2 \otimes_{\AKp} A_\inf[1/\mu]\arrow[r, "\cong ",phantom]&T(N_2)\otimes_{\mathbb{Z}_p} A_\inf[1/\mu]
\end{tikzcd}\]
Since $N_2 \hookrightarrow  N_2\otimes_{\AKp}A_\inf[1/\mu]$ is injective, we conclude that $f=g$. 

Let $N_1$ and $N_2$ be crystalline $(\varphi,\Gamma)$-modules over $\AKp$ and let $f_T\colon T(N_1) \rightarrow T(N_2)$ be a morphism of free $\mathbb{Z}_p$-representations of $G_K$. 
By the isomorphism \eqref{eq:isomorphism between N and T after inverting mu}, we have a morphism $N_1 \otimes_{\AKp} A_\inf[1/\mu] \rightarrow N_2 \otimes_{\AKp} A_\inf[1/\mu]$ which is compatible with the $G_K$-actions and the Frobenii. 
Since $A_\inf^{H_K}=\AKp$ by the Ax-Sen-Tate theorem \cite{Ax1964,Sen1969,Tate1967}, taking the $H_K$-fixed part gives us a morphism $ N_1 \otimes_{\AKp} \AKp[1/\mu] \rightarrow N_2 \otimes_{\AKp} \AKp[1/\mu]$.
By \cref{prop:Tsuji Prop 76} (2), this homomorphism induces a morphism $f_N\colon N_1  \rightarrow N_2$. By the construction of $f_N$, we have the following commutative diagram:
\[\begin{tikzcd}
N_1 \otimes_{\AKp} W(\mathbb{C}_p^\flat)\arrow[d, "f_N \otimes \id_{W(\mathbb{C}_p^\flat)}"'] &T(N_1)\otimes_{\mathbb{Z}_p} W(\mathbb{C}_p^\flat)\arrow[l, "\cong "'] \arrow[d, "f_T \otimes\id_{W(\mathbb{C}_p^\flat)}"]\\ 
N_2 \otimes_{\AKp} W(\mathbb{C}_p^\flat)&T(N_2)\otimes_{\mathbb{Z}_p} W(\mathbb{C}_p^\flat).\arrow[l, "\cong "']
\end{tikzcd}\]
By taking the $\varphi=1$ part, we see $T(f_N)=f_T$, as desired.

We show that the composite
     \begin{align*}
          N\colon \Rep_{\mathbb{Z}_p}^{\crys}(G_K) \xleftarrow{\cong } \Vect^\varphi(\mathcal{O}_{K,\Prism}, \mathcal{O}_\Prism)  \xrightarrow{\ev} \Mod_{\varphi,\Gamma}^{\fh, \crys}(\AKp)
     \end{align*}
     gives us the quasi-inverse functor of $T$, where the equivalence $\Rep_{\mathbb{Z}_p}^{\crys}(G_K) \xleftarrow{\cong } \Vect^\varphi(\mathcal{O}_{K,\Prism}, \mathcal{O}_\Prism)$ is \cref{thm:prismatic F-crystal and crystalline rep} due to Bhatt and Scholze. By the definition of the functors, $T\circ \ev$ is equal to the functor $\Vect^\varphi(\mathcal{O}_{K,\Prism}, \mathcal{O}_\Prism) \xrightarrow{\cong } \Rep_{\mathbb{Z}_p}^{\crys}(G_K)$. Therefore, we see that $T\circ N\cong \id$. This shows $T$ is essentially surjective and $N$ is the quasi-inverse functor of $T$.
\end{proof}

\bibliographystyle{halpha}
\bibliography{Master_thesis}
\end{document}